\documentclass[usletter]{amsart} 

\usepackage{amsmath,amsthm,amssymb,amsfonts,mathrsfs,color,hyperref, mathtools,crop, graphicx, enumitem, todonotes}
\usepackage[left=3cm, right = 3cm, top = 3cm, bottom = 3cm]{geometry}

\theoremstyle{plain}
\begingroup
\newtheorem{theorem}{Theorem}[section]
\newtheorem*{theorem*}{Theorem}
\newtheorem*{"theorem"}{``Theorem''}
\newtheorem{corollary}[theorem]{Corollary}

\newtheorem{lemma}[theorem]{Lemma}
\endgroup

\theoremstyle{definition}
\begingroup

\endgroup

\theoremstyle{remark}
\begingroup
\newtheorem{remark}[theorem]{Remark}
\newtheorem{example}[theorem]{Example}
\endgroup 

\numberwithin{equation}{section}
\newenvironment{pde}{\left\{\begin{array}{rll} } {\end{array}\right.}

\newcommand{\R}{\mathbb R}

\newcommand{\dist}{{\rm dist}}

\newcommand{\sdist}{{\mathrm{sdist}}}

\newcommand{\sign}{\mathrm{sign}}
\newcommand{\Per}{\mathrm{Per}}

\newcommand{\LRa} {\Leftrightarrow}
\newcommand{\Ra} {\Rightarrow}

\renewcommand{\d}{\mathrm{d}}

\newcommand{\dx}{\,\mathrm{d}x}

\newcommand{\dz}{\,\mathrm{d}z}
\newcommand{\ds}{\,\mathrm{d}s}
\newcommand{\dt}{\,\mathrm{d}t}
\newcommand{\dr}{\,\mathrm{d}r}

\newcommand{\cc}{\Subset}

\newcommand{\eps}{\varepsilon}
\newcommand{\average}{{\mathchoice {\kern1ex\vcenter{\hrule height.4pt
width 6pt depth0pt} \kern-9.7pt} {\kern1ex\vcenter{\hrule
height.4pt width 4.3pt depth0pt} \kern-7pt} {} {} }}

\allowdisplaybreaks

 \makeatletter
\@namedef{subjclassname@2020}{2020 Mathematics Subject Classification}
\makeatother

\DeclareMathOperator*{\argmin}{argmin}

\begin{document}

\title[Convex-concave splitting and slow movement]{Convex-concave splitting for the Allen-Cahn equation\\ leads to $\varepsilon^2$-slow movement of interfaces}

\author{Patrick Dondl}
\address{
Patrick Dondl\\
Abteilung für Angewandte Mathematik,
Albert-Ludwigs-Universit\"at Freiburg,
Hermann-Herder-Straße 10,
79104 Freiburg i.\ Br., 
Germany 
}
\email{patrick.dondl@mathematik.uni-freiburg.de}

\author{Akwum Onwunta}
\address{
Akwum Onwunta\\
Industrial and Systems Engineering,
Lehigh University,
200 West Packer Avenue,
Bethlehem, PA 18015, USA
}
\email{ako221@lehigh.edu}

\author{Ludwig Striet}
\address{
Ludwig Striet\\
Abteilung für Angewandte Mathematik,
Albert-Ludwigs-Universit\"at Freiburg,
Hermann-Herder-Straße 10,
79104 Freiburg i.\ Br., 
Germany 
}
\email{ludwig.striet@mathematik.uni-freiburg.de}

\author{Stephan Wojtowytsch}
\address{Stephan Wojtowytsch\\
Department of Mathematics\\
University of Pittsburgh\\
Pittsburgh, PA 15213}
\email{s.woj@pitt.edu}

\date{\today}

\subjclass[2020]{
65M12, 
35A35, 
49Q05
}
\keywords{Allen-Cahn equation, MBO scheme, convex-concave splitting, overstabilizing}

\begin{abstract}
The convex-concave splitting discretization of the Allen-Cahn is easy to implement and guaranteed to be energy decreasing even for large time-steps. We analyze the time-stepping scheme for a large class of potentials which includes the standard potential as well as two extreme settings: Potentials with quadratic convex part (uniform positive curvature), and potentials which are concave between the potential wells and either linear or infinite outside (highly concentrated curvature). In all three scenarios, the `effective time step size' of the scheme scales with the {\em square} of the small parameter $\varepsilon$ governing the width of transition layers. A weaker `slow motion' result is proved under much more general assumptions. Thus, stability is achieved by effectively `freezing' the interfaces in place. The time step limitation is not geometric in origin, but depends on the phase-field parameter $\varepsilon$. Along the way, we establish a new link between an Allen-Cahn type equation and a thresholding approximation of mean curvature flow.
\end{abstract}

\maketitle


\section{Introduction}

The Allen-Cahn equation
\[
\partial_t u = \Delta u - \frac{W'(u)}{\eps^2}, \qquad W(u) = \frac{(u^2-1)^2}4, \qquad W'(u) = u^3-u
\]
arises as the time-normalized $L^2$-gradient flow of the Modica-Mortola functional
\[
E_\eps : H^1(\Omega) \cap L^4(\Omega)\to [0,\infty), \qquad E_\eps (u) = \int_\Omega \frac\eps2\,\|\nabla u\|^2 + \frac{W(u)}\eps\dx.
\]
The functionals $E_\eps$ converge to the $c_W = \int_{-1}^1\sqrt{2\,W(z)}\dz$-fold of the perimeter functional in the sense of $\Gamma$-convergence as $\eps\to 0^+$, or to the perimeter relative to $\Omega$, depending on what boundary conditions are chosen for $u$ \cite{modica1977esempio, modica1987gradient}. Solutions $u$ to the Allen-Cahn equation converge to functions $\chi_V - \chi_{V^c}$ as $\eps\to 0^+$ where $V = \{V_t\}_{t\in [0,\infty}$ is a family of sets $V_t\subseteq\Omega$ whose boundaries evolve by mean curvature flow, the gradient flow of the perimeter functional \cite{bronsard1991motion, ilmanen1993convergence, laux2018convergence, fischer2020convergence}. In this article, we study the numerical solution of the Allen-Cahn equation by the popular `convex-concave splitting' scheme.

The convergence results stated above for the `classical' or `standard' double-well potential $W(u) = (u^2-1)^2/4$ are valid in much greater generality. Typical requirements for the convergence of perimeter-functionals are that $W$ is continuous, even, and $W(u)>0$ unless $u\in\{-1,1\}$. For compactness, it is convenient to prescribe an asymptotic growth condition at $\pm \infty$. To study evolution problems, smoothness of $W$ is generally assumed. 

For finer control of the functionals, it is necessary to consider the behavior of $W$ at the potential wells $u = \pm 1$. Since $W\geq 0$ and $W(u) = 0$ if and only if $u\in \{-1,1\}$, we automatically find that $W'(\pm 1) = 0$ (if $W\in C^1$) and $W''(\pm 1)\geq 0$ (if $W\in C^2$). The optimal transition profile $\phi$ between the potential wells is the solution of the ODE
\[
\phi'' = W'(\phi), \qquad \lim_{x\to\pm \infty} \phi(x) = \pm 1, \qquad \phi(0) = 0
\]
We note that
\[
\frac{d}{dx} \left(\frac{(\phi')^2}2 - W(\phi)\right) = \big(\phi'' - W'(\phi)\big)\phi' = 0\qquad \Ra \quad \phi' \equiv \sqrt{2\,W(\phi)}
\]
since $\phi'(\pm\infty) = W(\phi(\pm\infty)) = 0$. This equation can alternatively be used to analyze existence, uniqueness, and fine properties of $\phi$. In particular, a unique monotone solution $\phi$ of the equation $\phi'= \sqrt{2W(\phi)}$ exists with $\phi(0) = 0$ since $W>0$ inside $(-1,1)$. If $W''(1) >0$, then $(1-\phi)'' = -W'(\phi) \approx W''(1)\,(1-\phi)$ implies that $\phi(t) \sim 1 - C\,\exp(-\sqrt{W''(0)}\, t)$ as $t\to\infty$, i.e.\ $\phi$ approaches the potential wells exponentially fast. If on the other hand $W(\phi) \sim (1-|\phi|)^\gamma$ for $\gamma<2$ at $\pm 1$, then $\phi$ transitions from $-1$ to $1$ on a finite segment of the real line. Non-smooth potentials with this behavior behave advantageously in certain applications certain applications \cite{bosch2018generalizing}.

In this note we consider efficient time-discretizations of the Allen-Cahn equations for double-well potentials fashioned after 
\[
W(u) = |u|^p + W_{conc}(u) 
\]
where $p\geq 2$ and $W_{conc}$ is a concave function such that $W(u)> 0$ unless $u\in \{-1,1\}$ and $W(\pm 1)=0$. The exact conditions allow for slightly more generality. We pay special attention to these three cases:

\begin{enumerate}
\item Potentials with quadratic convex part, i.e.\ we assume that $W(u) = u^2 + W_{conc}(u)$ where we assume that $W_{conc}$ is concave (and smooth enough for our purposes).

\item The standard potential $W(u) = (u^2-1)^2 = u^4 + 1-2u^2$ with $p=4$ and $W_{conc}(u) = 1-2u^2$. This is the most common potential (up to a factor of $1/4$, which is occasionally present in the literature). The choice of $W$ in place of $W/4$ corresponds to replacing $\eps$ by $2\eps$.

\item Barrier potentials of the form $W(u) = \infty\cdot 1_{\{|u|>1\}} + W_{conc}(u)$, where again $W_{conc}$ is assumed to be concave. This special case arises as the formal limit $p\to \infty$ and choosing the lower semi-continuous envelope with $W(\pm 1) = 0$ in order to guarantee the existence of minimizers. 
\end{enumerate}

The first class of potentials is particularly convenient as many statements simplify compared to the standard potential which grows as $u^4$ at $\pm \infty$ -- for instance, its associated energy $E_\eps$ is defined and finite on the whole space $H^1$. Both the first family and the standard potential fall into the category where $W'(\pm 1) = 0$, $W''(\pm 1) > 0$, while potentials in the final family satisfy $W(u) \geq c\,\big|1-|u|\big|$ for some $c>0$. As prototypical examples, we can consider
\[
W(u) = \big(|u|-1\big)^2 = u^2 + \underbrace{1- 2\,|u|}_{=:W_{conc}(u)}\qquad\text{and}\quad W(u) = \infty \cdot 1_{\{|u|>1\}} + \underbrace{1-|u|}_{=: W_{conc}(u)} = \begin{cases} 1-|u| & |u|\leq 1\\ +\infty &\text{else}\end{cases}
\]
or approximations which are smooth inside $(-1,1)$. In particular the first class allows for a simple implementation and analysis of the `convex-concave splitting' time discretization of the Allen-Cahn equation which treats the convex part of $E_\eps$ implicitly in time and the concave part explicitly, i.e.\
\[
\frac{u_{n+1} - u_n} \tau = \Delta u_{n+1} - \frac{W_{vex}'(u_{n+1}) + W_{conc}'(u_n)}{\eps^2}
\]
where $W_{vex}, W_{conc}$ are the `convex part' and `concave part' of $W$ respectively. Such discretizations combine useful features of both implicit and explicit time-stepping schemes:
\begin{itemize}
\item As for an explicit scheme, the next iterate is uniquely determined and can be found by solving a computationally tractable problem.
\item As for an implicit (minimizing movements) scheme, $E_\eps$ is guaranteed to decrease along the iteration, even for large time-steps (assuming that the `minimizing movements' solution is found, which may not be the unique solution to the implicit Euler equation).
\end{itemize}
 If the convex part of $W$ is quadratic, the implicit time-step reduces to solving a linear system with the same operator in every time step, meaning that e.g.\ LU-factorization methods can be used efficiently. In contrast, with the classical potential $W(u) = (u^2-1)^2 /4$ or a barrier potential, each time-step requires the solution of a convex minimization problem and may itself be expensive.

\begin{figure}
\includegraphics[width = .48\linewidth]{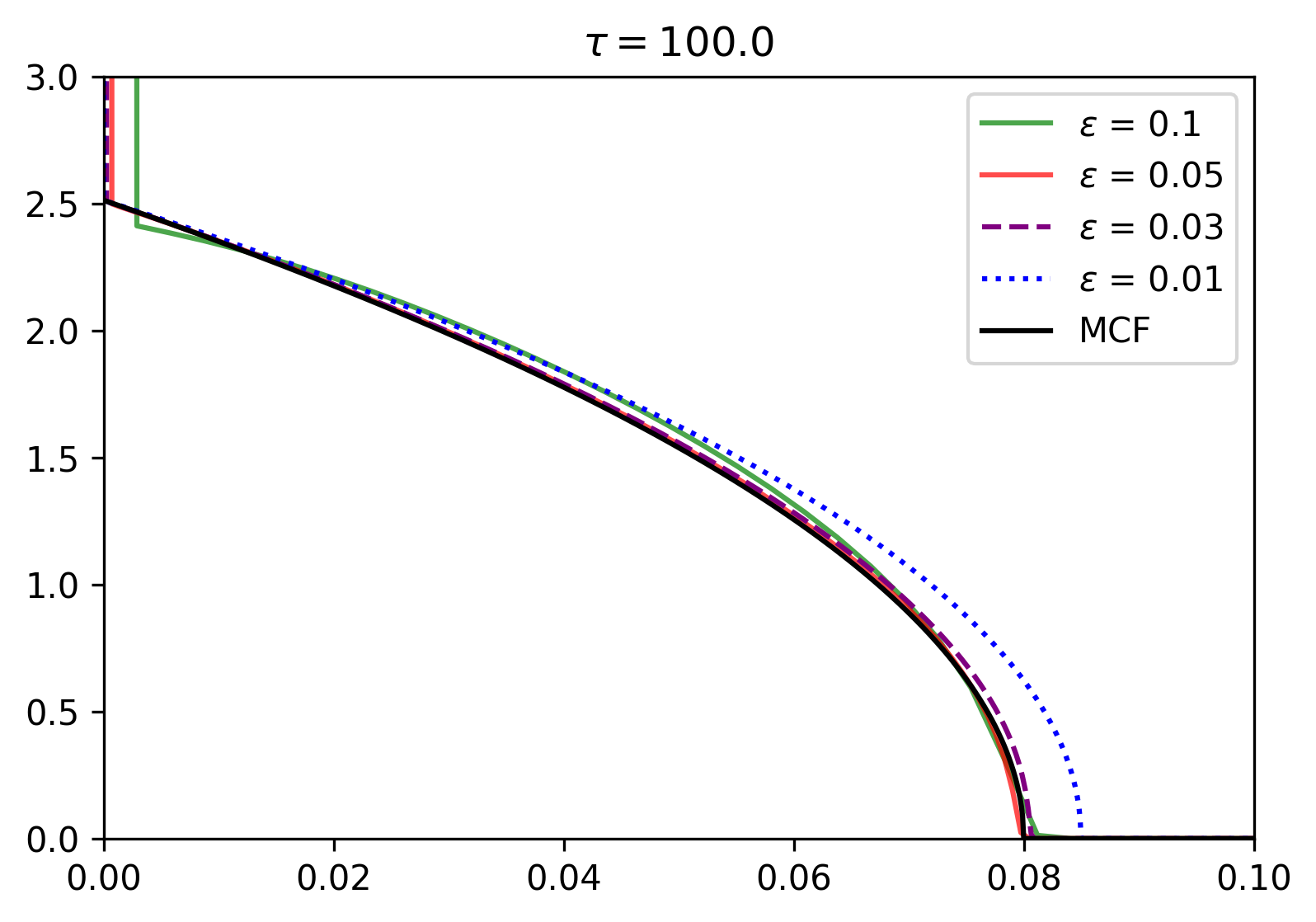}
\includegraphics[width = .48\linewidth]{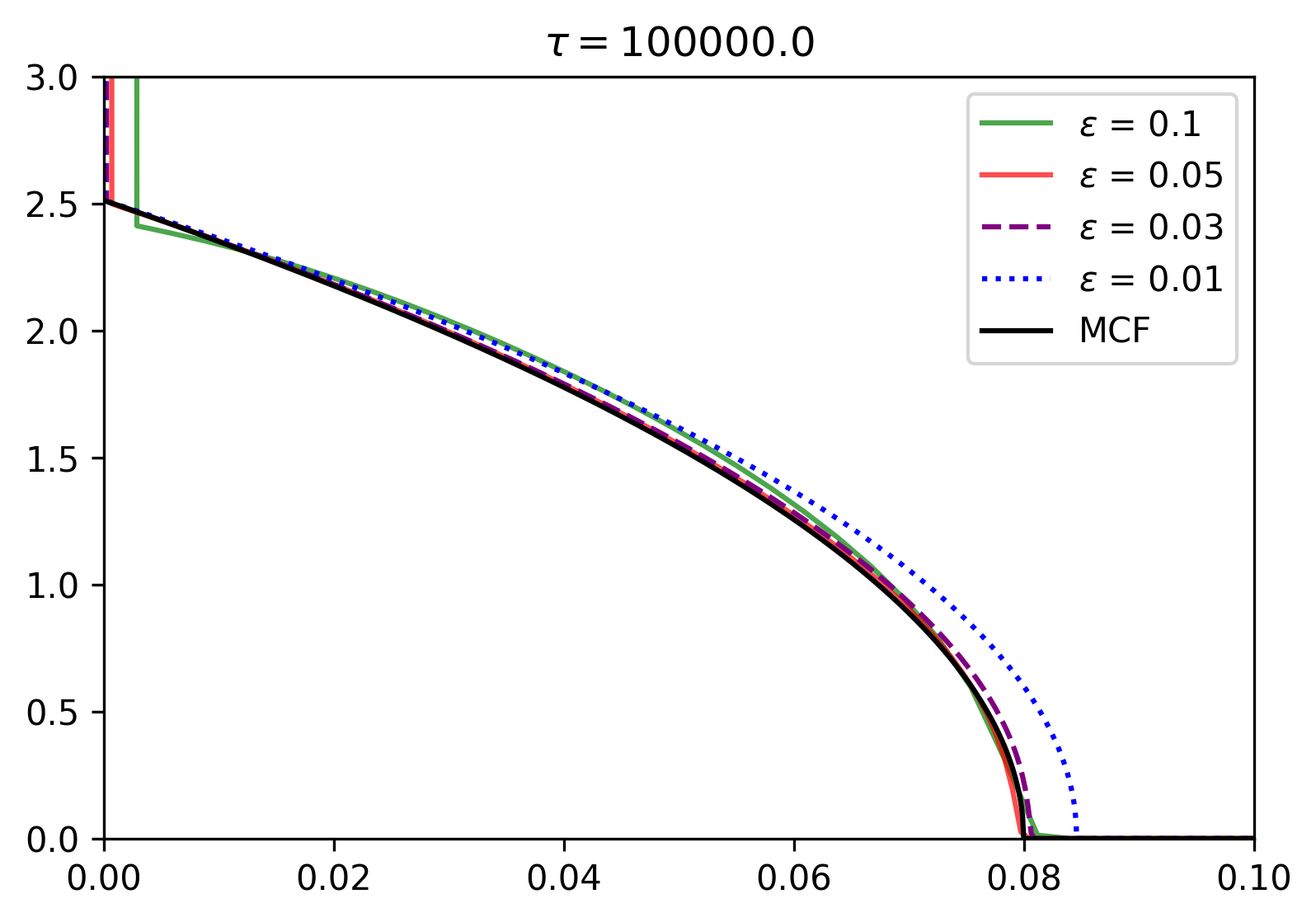}

\caption{Comparing the (normalized) Modica-Mortola energies $c_W^{-1}E_\eps$ along the convex-concave splitting time-stepping scheme for $\tau= 100$ (left)  and $\tau = 100,000$ (right). The initial condition is $2\cdot \chi_E - 1$ where $E$ is a circle of radius $r_0 = 0.4$ in the square $(0,1)^2$ with periodic boundary conditions. The potential is $W_R$ for $R=100$ as in Appendix \ref{appendix W_R}.\\
In each plot, the times are computed not according to the `nominal step-size' $\tau$, but according to an `effective step size' $0.29\cdot \eps^2$, i.e.\ we take 100 times more time steps for $\eps = 0.01$ compared to $\eps = 0.1$. All computations are performed on a uniform $n\times n$ grid for $n=512$ using a Fourier domain representation.\\
The plots visually agree well with the perimeter $2\pi \sqrt{r_0^2-2t}$ of the analytic solution to mean curvature flow/curve shortening flow $\dot r = -1/r$ of a circle. The increase in time step size $\tau$ makes no difference to the behavior of the curves.
\label{figure explore eps and tau}
}
\end{figure}    

Our main results for the convex-concave splitting time discretization of the Allen-Cahn equation are that
\begin{enumerate}
\item If either the convex or the concave part of the potential $W$ satisfies a type of uniform convexity condition, the effective time step size is limited by a power of $\eps$. For instance, if $W(u) = |u|^p + W_{conc}(u)$ for $2\leq p<\infty$, even formally large time-steps in the convex-concave splitting scheme effectively scale at best as $\eps^{1/(p-1)}$, independently of the formal time step size $\tau$ in the scheme. The same result holds for potentials with concave part $W_{conc}(u) = c - |u|^p$. For a more precise statement, see Section \ref{section eps-slow}.

\item For potentials with quadratic convex part, the convex-concave splitting is especially cheap to implement, but the effective step size scales as $\eps^2$ -- see Figure \ref{figure explore eps and tau}. This is even slower than the time scale $\eps$ guaranteed above. The analysis in this setting builds on a link to another computational approximation of mean curvature flow, the thresholding or Merriman-Bence-Osher (MBO) scheme. A full analysis of the modified MBO scheme is beyond the scope of this article, and our analysis is highly suggestive, but only fully rigorous on the entire space $\R^3$.

\item For the barrier potential with concave part $W_{conc}(u) = 1- |u|$, the `curvature' type conditions do not apply. Nevertheless, we demonstrate by example that the motion of interfaces occurs on the $\eps^2$-time scale also in this case.

\item For the  standard potential, we demonstrate the same $\eps^2$-slow motion in a numerical example.

\item We give a heuristic explanation for the universality of the $\eps^2$-time scaling.
\end{enumerate} 

More precise statements can be found in the text. In this work, we focus on the setting of discrete time, but continuous space. We believe that the methods apply equally to spatial discretizations, especially those which preserve the $L^2$-gradient flow structure. Morally, our contribution is the following: {\em Convex-concave splitting time-discretizations of the Allen-Cahn equation are numerically stable and at times easy to implement, but in all cases considered in this note, they achieve stability by drastically limiting the effective size. This allows the practitioner to choose an essentially arbitrary time-step without fear of blowing up, but in general it does not improve upon, for instance, a discretization which treats the Laplacian implicitly and the zeroth order term explicitly (with a sufficiently small step size). For potentials with convex parts that are not merely quadratic, the convex-concave splitting scheme is in fact {\em slower} in practice since each time-step requires the solution of a different linear system, while the na\"ive semi-implicit scheme requires the same system in each step.} 

The article is structured as follows. In Section \ref{section splitting basics}, we explore the motivation and guarantees of the convex-concave splitting algorithm. In Section \ref{section eps-slow}, we establish the result that the effective step size of the convex-concave splitting scales at most as $\eps^{1/(p-1)}$ for any potential with convex part $|u|^p$, based on a simple energy-dissipation estimate. In Section \ref{section quadratic}, we study potentials with quadratic convex part. Here, we also consider initial conditions which are not well-prepared and establish a link between the convex-concave splitting discretization of the Allen-Cahn equation and the MBO scheme for the prototypical double-well potential $W(u) = (|u|-1)^2$. Using the easier structure of linear equations, we observe that the effective step size in fact only scales as $\eps^2$, rather than the sub-optimal rate $\eps$ obtained previously. In Section \ref{section barrier potentials}, we analyze the scheme for barrier potentials which do not fall into one of the previous categories. We present a numerical experiment for the classical potential $W(u) = (u^2-1)^2/4$ in Section \ref{section numerical}. An intuitive explanation for $\eps^2$-slow motion and open questions are discussed in Section \ref{section conclusion}. Smooth potentials with quadratic convex part are constructed in Appendix \ref{appendix W_R}.

\subsection{Conventions and notations}\label{section conventions}

We will assume that the solution operator $(a-b\Delta)^{-1} : L^2(\Omega)\to L^2(\Omega)$ is densely defined, positive semi-definite and self-adjoint  for $a, b\geq 0$ and $a+b>0$. We will write $(a-b\Delta)^{-1}$ without specifying the boundary condition outside of specific examples, but always assume self-adjointness. The specific conditions of interest are:

\begin{enumerate}
\item Constant Dirichlet boundary conditions $u\equiv 1$ on a relatively closed and open subset $\Gamma\subseteq \partial\Omega$ and $u\equiv -1$ on $\partial\Omega\setminus \Gamma$,
\item Homogeneous Neumann boundary conditions $\partial_\nu u \equiv 0$ on $\partial\Omega$, and
\item Periodic boundary conditions.
\end{enumerate}

In the first setting, the solution operator is naturally not linear, but affine-linear. In principle, many results should also apply to evolutions on closed Riemannian manifolds, although we do not pay closer attention to this setting. 

We write $V\cc \Omega$ to signify that $\overline V$ is compact and $V\subseteq \overline V \subset \Omega^\circ \subseteq \Omega$.

The double-well potentials $W$ are always assumed to be positive inside $(-1,1)$ and vanish at $\pm 1$. Norms are denoted by $\|\cdot\|$ and the norm of vectors is assumed to be Euclidean throughout.

\section{Background on convex-concave splitting}\label{section splitting basics}

\subsection{Theoretical Foundations}
Let us briefly consider the optimization of general functions which can be written as $F+G$ with $F$ convex and $G$ concave.
Despite the simple ingredients, $F+G$ can be quite complicated -- indeed, any $C^2$-function $f:\R^d\to\R$ with globally bounded Hessian can be written as
\[
f (x) = \underbrace{ f(x)+ \frac\lambda2\,\|x\|^2}_{\text{convex}} + \underbrace{\frac{-\lambda}2\,\|x\|^2}_{\text{concave}}
\]
for sufficiently large $\lambda$. We assume that such a splitting $f=F+G$ is given.

An explicit gradient descent scheme $x_{n+1} = x_n - \tau \,\nabla f(x_n)$ is only monotonically energy-decreasing for general $f$ if $\tau$ is small enough -- specifically, if $\tau< 2/L$ is sufficient where $L$ is the Lipschitz-constant of $\nabla f$. The necessity of this condition can be proved by considering $f(x) = \frac L2\,x^2$, and its sufficiency is well-known in optimization -- see e.g.\ \cite[Theorem 3.1]{wojtowytsch2023stochastic}. Only in special cases can we choose `large' step sizes to attain faster convergence \cite{altschuler2024acceleration, altschuler2025acceleration, grimmer2025accelerated}.

The implicit gradient descent scheme $x_{n+1} = x_n - \tau\,\nabla f(x_{n+1})$ on the other hand always has an energy decreasing solution $x_{n+1} = \argmin_z \frac12\,\|x-z\|^2 + \tau\,f(z)$, but if $\tau$ is too large, there are generally additional solutions with possibly higher energy -- see Figure \ref{figure implicit euler}. In general, finding any solution requires solving a non-linear system of equations. Finding the energy decreasing solution may not be possible.

    \begin{figure}
    \centering
    \includegraphics[width=0.3\linewidth]{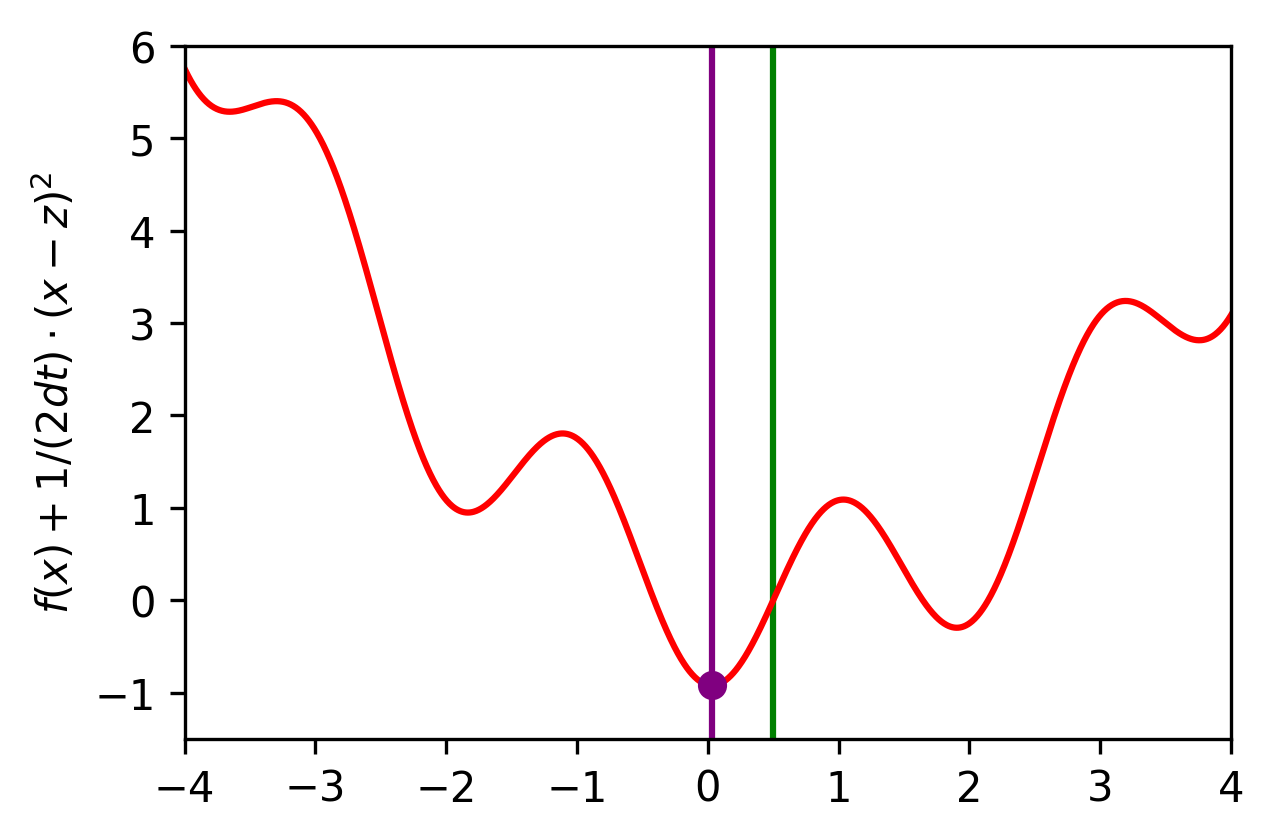}
    \includegraphics[width=0.3\linewidth]{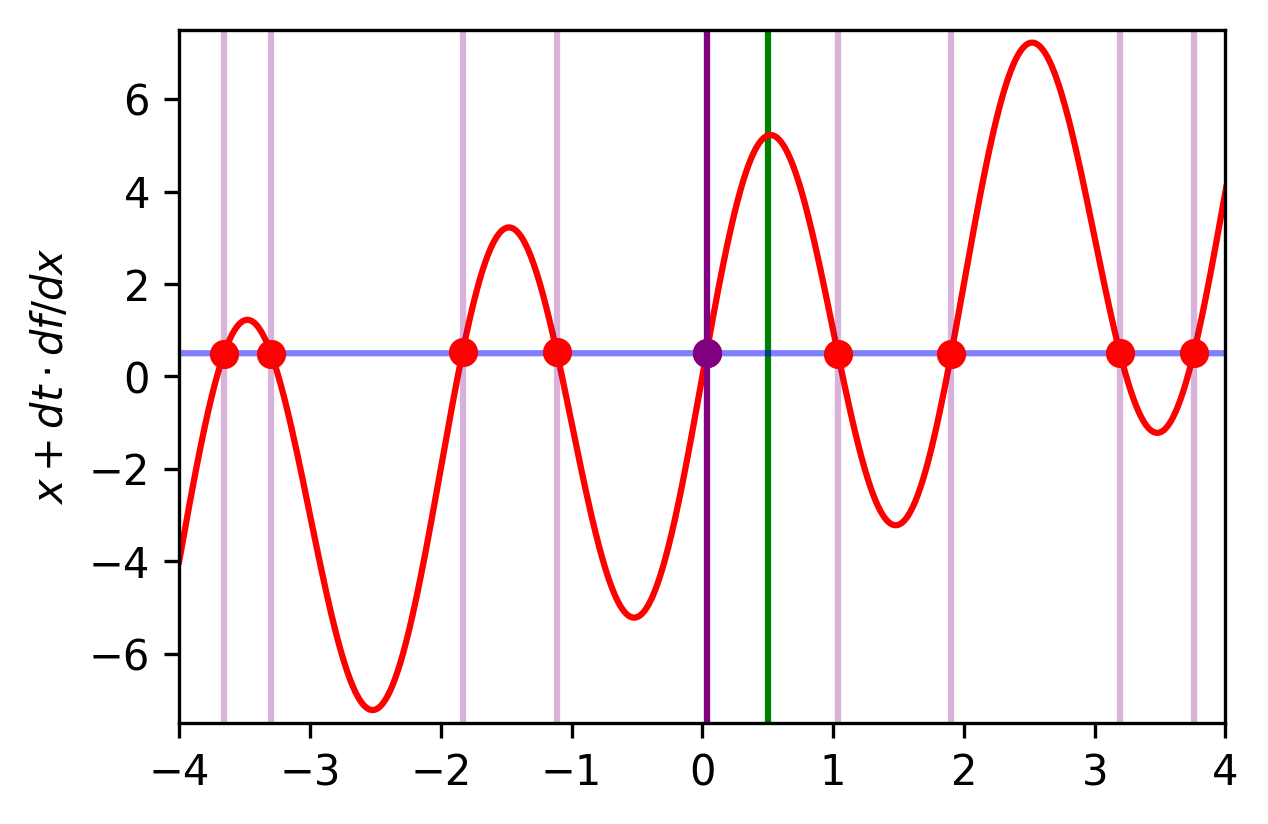}
    \hspace{2mm}       
    \includegraphics[width=0.3\linewidth]{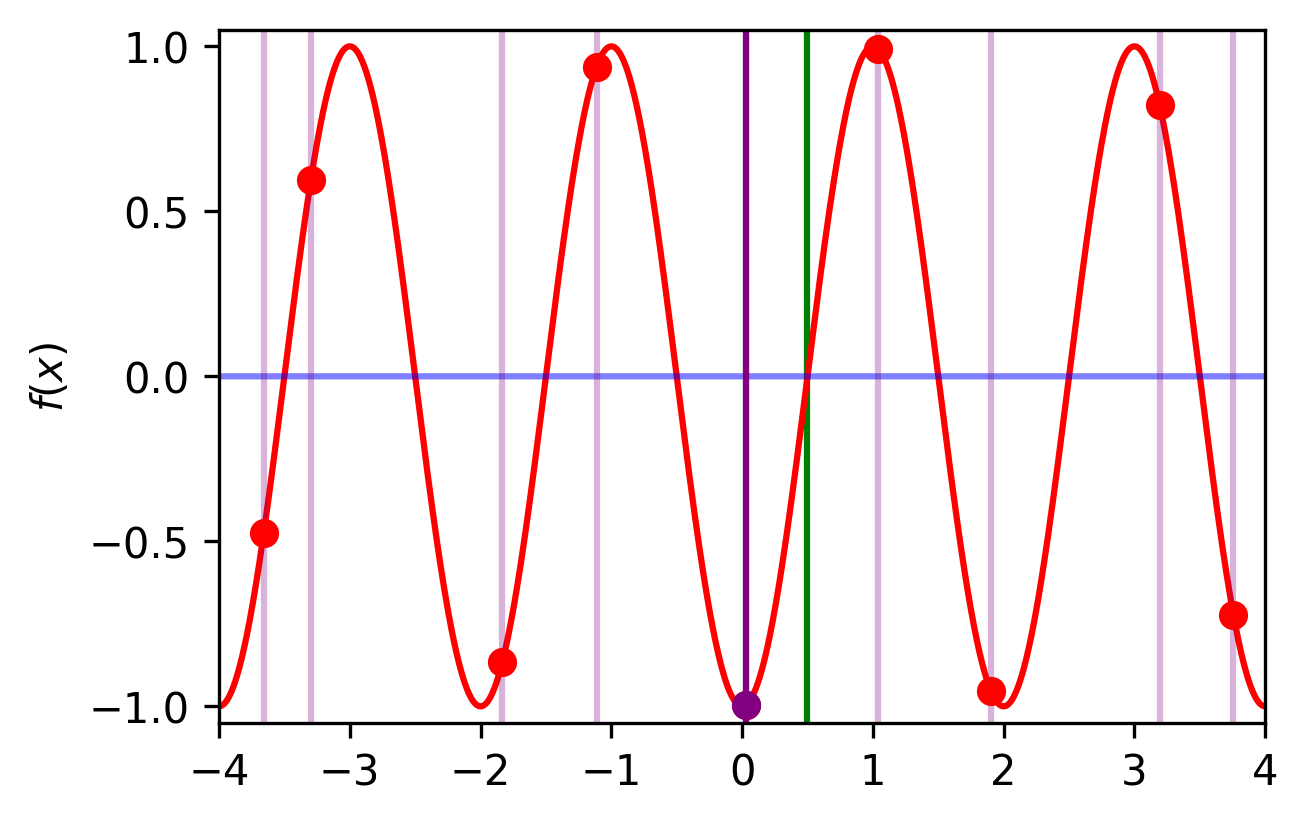}
        \caption{\label{figure implicit euler}
        {\bf Left:} $f(x) + \frac1{2\tau} \|x-z\|^2$ for $f(x) = -\cos(\pi x)$, $\tau = 1.5$ (red line) and $z= 0.5$ (vertical green line). The minimizer $x\approx 0.03$ (marked in purple) is the solution of the `minimizing movements' scheme.
        {\bf Middle:} $x+ \tau f'(x)$  as well as $z= 0.5$ (blue line). The intersections of the lines are the spurious solutions to the implicit Euler gradient descent equation $x+ \tau f'(x) = z$ starting at $z$ (local extrema in the left plot). 
        {\bf Right:} The objective function $f(x)$ (red line) and the energy level $f(z)$ at the initial point (blue line). The admissible implicit Euler solutions and the initial point $z$ are marked by vertical lines. The minimizing movements solution is marked by the solid purple line and dot. We see that the energy at the possible next steps is above the energy at the starting point for four out of nine options.}
    \end{figure}

If a splitting $f= F+ G$ into a convex part $F$ and a concave part $G$ is available, the mixed explicit/implicit iteration scheme
\begin{equation}\label{eq convex concave splitting}
x_{n+1} = x_n - \tau\big(\nabla F(x_{n+1}) + G(x_n)\big)
\end{equation}
defines the sequence $x_n$ uniquely (given $x_0$) and is energy-stable for any step size, see e.g.\ \cite[Chapter 6]{bartels2015numerical}. Namely, $x_{n+1}$ is the {\em unique} minimizer of the strongly convex function 
\[
f_{x_n, \tau}^{aux}(z) := \frac12\,\|z-x_n\|^2 + \tau \big(F(z) + G(x_n) + \langle \nabla G(x_n), z-x_n\rangle\big).
\]
To see that the scheme is energy-stable, recall that a convex $C^1$-function $F:\R^d\to\R$ satisfies the first order convexity condition
\[
F(x) \geq F(z) + \langle \nabla F(z), x-z\rangle\qquad \forall\ x, z\in \R^d.
\]
For a concave function $G$ on the other hand, we note that
\[
G(z) \leq G(x) + \nabla \langle G(x), z-x\rangle\qquad \forall\ x, z\in \R^d,
\]
where we reversed the roles of $x$ and $z$ in the notation. In particular, if $z = x_{n+1}$ and $x = x_n$ as in \eqref{eq convex concave splitting}, then we may add the two inequalities to obtain
\begin{align*}
(F+G)(x_{n+1}) 
    &\leq (F+G)(x_n) + \langle\nabla G(x_n), x_{n+1}-x_n\rangle - \langle \nabla F(x_{n+1}), x_n - x_{n+1}\rangle\\
    &= (F+G)(x_n) + \langle \nabla G(x_n) + \nabla F(x_{n+1}), \:x_{n+1}- x_n\rangle\\
    &= (F+G)(x_n) - \tau\,\big\|\nabla F(x_{n+1}) +\nabla G(x_n)\big\|^2.
\end{align*}
In particular, the sequence $(F+G)(x_n)$ is monotone decreasing independently of the step size.

The assumptions on $F, G$ can be relaxed somewhat: For instance, we may assume that $F, G: H\to \R$ such that $F$ is convex, lower semi-continuous, finite on a dense subspace $H'\subseteq H$ with a sub-differential on a dense sub-space $H''\subseteq H$, and that $G$ is continuous and has a Gateaux-derivative $\delta G(x; v) = \langle w_x, v\rangle$ for some $w_x\in H$ (i.e.\ the Gateaux-derivative is a linear map). These extensions are indeed necessary for the Allen-Cahn equation. For further details, see \cite[Theorem 4.1]{akande2024momentum} in a slightly different context.

For geometric intuition, we provide the following example.

\begin{example}\label{example quadratic convex-concave}
Assume that $F(x) = \frac12 \,x^TAx$ where $A$ is a symmetric positive semi-definite $d\times d$-matrix, and $G:\R^d\to\R$ is a smooth convex function. Then the convex-concave splitting gradient descent scheme reads
\[
    x_{n+1} = x_n - \tau\,Ax_{n+1} - \tau\,\nabla G(x_n),
\]
or equivalently 
\begin{align*}
    (I+\tau A)x_{n+1} &= x_n - \tau\,\nabla G(x_n) = (I+ \tau A) x_n - \tau\,(Ax_n + \nabla G)(x_n)\\
    &= (I+\tau A) x_n - \tau \,\nabla (F+G)(x_n).
\end{align*}
This reduces to the preconditioned explicit gradient descent scheme
\begin{align*}
    x_{n+1} &= x_n - \tau\big(I+\tau\,A\big)^{-1}\nabla (F+G)(x_n)
\end{align*}
which can be interpreted as an interpolation between the explicit gradient descent scheme and a partial Newton iteration since $A = D^2F(x_n)$ is the Hessian of the {\em convex} part of $F+G$. The time step size $\tau$ governs the trade-off between gradient descent and Newton iteration.
    
In one dimension $A = \lambda \in (0,\infty)$ and the effect of the convex-concave splitting is merely to automatically select a suitable step size $\tau / (1+\tau \lambda)$ for the explicit gradient descent scheme. As $\tau\to \infty$, the step size approaches $1/\lambda = 1/[\nabla F]_{Lip}$, which is the minmax optimal step size for gradient descent considering only $F$ and not $G$. 
    
In higher dimensions, the effect of preconditioning is more subtle as $A$ is matrix-valued. If $A$ is strictly positive definite, then
\begin{align*}
 x_{n+1} &= x_n - \tau\big(I+\tau\,A\big)^{-1}\nabla (F+G)(x_n)
    	&&= x_n - \left(\frac 1\tau\,I + A\right)^{-1}\big(Ax_n + \nabla G(x_n)\big)\\
	&\xrightarrow{\tau\to +\infty} x_n - A^{-1} Ax_n - A^{-1} \nabla G(x_n)
	&&= -A^{-1}\,\nabla G(x_n),
\end{align*}
i.e.\ the limiting time-step algorithm as $\tau\to \infty$ exists.
\end{example}

This observation also illustrates the limitations of the convex-concave splitting: A large step size may in reality correspond to a fairly small (explicit) time-step. 
For any convex function $\Phi$, also $(F+\Phi) + (G-\Phi)$ is a convex-concave splitting of $F+G$. Choosing $\Phi$ as a convex quadratic in the example above leads to a more restrictive step size limiter of the same form than $\Phi\equiv 0$.

\subsection{Application to the Allen-Cahn Equation}

A splitting of the potential $W= W_{vex}+W_{conc}$ into a convex part $W_{vex}$ and a concave part $W_{conc}$ induces a splitting of $E_\eps$ 
\[
E_\eps(u) = \int_\Omega \frac\eps2 \,\|\nabla u\|^2 + \frac{W_{vex}(u)}\eps\dx + \int_\Omega \frac{W_{conc}(u)}\eps\dx
\]
into a convex part and a concave part. The (normalized) time-step equation is
\[
u_{n+1} = u_n - \frac\tau\eps\,\left(-\eps\Delta u_{n+1}+ \frac{W_{vex}'(u_{n+1})}{\eps} + \frac{W_{conc}'(u_n)}{\eps}\right)
\]
with a time-normalized gradient flow as in the Allen-Cahn equation, or equivalently
\[
\left(1- \tau\,\Delta\right) u_{n+1} + \frac\tau{\eps^2}\,{W_{vex}'(u_{n+1})} = u_n - \frac\tau{\eps^2}\,W_{conc}'(u_n).
\]
Again, we refer to \cite{akande2024momentum} for a more detailed presentation. The equation takes a particularly simple form if $W_{vex}$ is quadratic, making the equation linear in $u_{n+1}$. 
We construct smooth potentials $W_R$ with quadratic convex part -- smooth approximations of $\overline W(u) = (|u|-1)^2$ -- in Appendix \ref{appendix W_R}.

\section{Theoretical guarantees: Slow movement for general potentials}\label{section eps-slow}

\subsection{A simple slowness bound} \label{section general bound}

Throughout this section, we assume that $W$ satisfies the following properties:
\begin{enumerate}
    \item $W\geq 0$ and $W(u) = 0$ if and only if $u\in\{-1,1\}$,
    \item $W = W_{vex} + W_{conc}(u)$ where $W_{conc}$ is concave and $W_{vex}$ is convex,
    \item $W_{conc}, W_{vex} \in C^1(\R)$ and
    \item $W_{vex}, W_{conc}$ jointly satisfy a uniform curvature condition
    \[
    W(u) \geq W(U) + \big(W_{vex}'(U) + W_{conc}'(u)\big)\,(u-U) + \bar c_W |u-U|^p\qquad\forall\ u, U\in\R
    \]
    for some $\bar c_W>0$ and $p\geq 2$.
\end{enumerate}

The assumption of joint uniform curvature is mild. Since $W_{vex}$ is convex and $W_{conc}$ is concave, the inequalities
\[
W_{vex}(u) \geq W_{vex}(v) + W_{vex}'(v)\,(u-v), \qquad W_{conc}(u) \geq W_{conc}(v) + W_{conc}'(u)(u-v)
\]
hold, noting that we evaluate the derivative at $v$ for the convex part and at $u$ for the concave part. The proposed inequality thus holds with $c_W = 0$ by simply adding the convexity and concavity conditions. If one of the inequalities can be sharpened by including a term such as $\bar c_W |u-v|^p$, then so can the joint inequality. Such functions are in fact common.

\begin{lemma}\label{lemma power law}
Let $p\geq 2$. Then there exists $\beta = \beta_p \geq 2^{-p}$ such that
\[
\big|u+w\big|^p \geq |u|^p + p|u|^{p-2}u\,w + \beta\,|w|^p\qquad \forall\ u, w\in\R.
\]
If $p=2$, we can choose $\beta=1$ and if $p=4$, we can choose $\beta = 1/3$.
\end{lemma}

In particular, the standard potential $W(u) = \frac14 (u^2-1)^2$ satisfies 
\[
W(u)  = \underbrace{\frac{u^4}4}_{=: W_{vex}} +\underbrace{\frac{1}4 - \frac{u^2}2}_{=: W_{conc}} \geq W(U) + W_{vex}'(U)\,(u-U) + \frac{|u-U|^4}{12} + W_{conc}'(u)(u-U) + \frac{|u-U|^2}2.
\]
A sharp constant for the lemma is also available if $W$ has a quadratic convex part. We believe Lemma \ref{lemma power law} to be known to experts, but have been unable to locate a reference. We prove it for the reader's convenience in Appendix \ref{appendix power law}.

Let $u_1, u_2, \dots$ be a sequence generated by the convex-concave splitting approximation to the Allen-Cahn equation with step size $\tau>0$ and spatial length scale $\eps>0$, starting at $u_0\in H^1(\Omega)$. In this section, we prove that 
\[
\|u_N - u_0\|_{L^p(\Omega)} \lesssim \big(N\eps^{1/(p-1)}\big)^{1-1/p}
\]
independently of $\tau$. This is a H\"older continuity estimate with the discrete time $t_N = N \eps^{1/(p-1)}$ and H\"older exponent $1-\frac1{p} = \frac{p-1}{p}$. In particular, even for large $\tau$, taking $N$ steps will not result in a meaningful distance from the initial point $u_0$ if $\eps \ll N^{1-p}$. 

\begin{lemma}\label{lemma energy descent}
Assume that $W$ satisfies the conditions above. Let $E_\eps(u_0)<\infty$ and $u_1$ the next time-step for the Allen-Cahn equation with the splitting associated to the decomposition of $W$ and with step-size $\tau>0$. Then
\[
\frac\eps2 \,[u_1-u_0]_{H^1_0(\Omega)}^2 + \frac\eps\tau \,\|u_1-u_0\|_{L^2(\Omega)}^2 + \frac{\bar c_W}\eps\,\|u_1-u_0\|_{L^p(\Omega)}^p  \leq E_\eps(u_0) - E_\eps(u_1).
\]
\end{lemma}

\begin{proof}
Recall that
\[
u_1 = u_0 + \tau\left( \Delta u - \frac{W_{vex}'(u_1) + W_{conc}'(u_0)}{\eps^2}\right) = u_0 +\frac{\tau}{\eps}\left(\eps\,\Delta u - \frac{W_{vex}'(u_1) + W_{conc}'(u_0)}{\eps}\right).
\]
We denote 
\[
v= u_1-u_0 =  \frac{\tau}{\eps}\left(\eps\,\Delta u - \frac{W_{vex}'(u_1) + W_{conc}'(u_0)}{\eps}\right).
\] 
By assumption, we have
\begin{align*}
E_\eps(u_0) &= \int_\Omega \frac\eps2 \,\|\nabla(u_1-v)\|^2 + \frac{W(u_0)}{\eps} \dx \\
	&\geq \int_\Omega \frac\eps2\,\|\nabla u_1\|^2+ \frac{W(u_1)}{\eps} - \eps\,\langle \nabla u_1, \nabla v\rangle + \frac{W(u_0) - W(u_1)}\eps  + \frac\eps2\,\|\nabla v\|^2 \dx\\
	&\geq E_\eps(u_1) +\int_\Omega v\,\eps \Delta u_1 + \frac{W_{vex}'(u_1) + W_{conc}'(u_0)}\eps (u_0-u_1) + \frac{\bar c_W}\eps\,|u_1-u_0|^p  + \frac\eps2\,\|\nabla v\|^2 \dx\\
	&= E_\eps(u_1) +\int_\Omega \left(\eps \Delta u_1- \frac{W_{vex}'(u_1) + W_{conc}'(u_0)}\eps\right) v+ \frac{\bar c_W}\eps\,|u_1-u_0|^p  + \frac\eps2\,\|\nabla v\|^2 \dx\\
	&= E_\eps(u_1) +\int_\Omega \frac\eps\tau\,v^2 + \frac{\bar c_W}\eps\,|v|^p  + \frac\eps2\,\|\nabla v\|^2 \dx.\qedhere
\end{align*}
\end{proof}

\begin{corollary}\label{corollary simple}
Assume that $W$ satisfies the conditions above. Let $u_0\in H^1(\Omega)$, $\tau, \eps>0$ and $u_1, u_2, \dots$ the sequence generated by the gradient descent scheme associated with the convex-concave splitting of $W$. Then
\[
\|u_N-u_0\|_{L^p(\Omega)} \leq E_\eps(u_0)^{1/p}
\left(N\eps^\frac1{p-1}\right)^{1-1/p}.
\]
\end{corollary}

The possibly simplest form of the statement applies to potentials with a quadratic convex part $u^2/2$ and the standard potential as 
\[
\|u_N-u_0\|_{L^2(\Omega)} \leq \sqrt{2\,N\eps\,E_\eps(u_0)}.
\]

\begin{proof}
From Lemma \ref{lemma energy descent}, we conclude by H\"older's inequality that
\begin{align*}
\|u_N-u_0\|_{L^p(\Omega)} &\leq \sum_{n=1}^N\|u_n-u_{n-1}\|_{L^p(\Omega)}
    &&\leq \left(\sum_{n=1}^N \|u_n-u_{n-1}\|_{L^p(\Omega)}^p\right)^\frac1p\left(\sum_{n=1}^N1\right)^{1-\frac1p}\\
    &\leq N^{1- 1/p}\left(\sum_{n=1}^N \big(E_\eps(u_{n-1}) - E_\eps(u_n)\big)\right)^\frac1p
    &&\leq E_\eps(u_0)^{1/p}\,\eps^{1/p}N^{(p-1)/p}.\qedhere
\end{align*}
\end{proof}

\begin{remark}
We only used that $E_\eps(u_0) \geq E_\eps(u_N)\geq 0$. However, if $u_N$ is known to be close to $u_0$, then $E_\eps(u_N)$ should be not much smaller than $E_\eps(u_0)$. This could be used to show that $u_N$ is in fact even closer to $u_0$ than we already knew. 

This very natural intuition is complicated by the fact that $E_\eps$ fails to be continuous in the $L^p$-topology. Indeed, if $u_0$ is far from being a well-prepared initial condition, then $E_\eps(u_0)$ may be very large or even infinite and $E_\eps(u_1)$ can be much smaller for $u_1$ close to $u_0$ in $L^p$. We conjecture that a geometric condition on the `modulus of lower semi-continuity' 
\begin{align*}
\omega (u_0; r) &:= \inf\left\{ E_\eps(u) - E_\eps(u_0): \|u-u_0\|_{L^p(\Omega)}<r\right\}
\end{align*}
could be used to improve the H\"older exponent, but not the scaling between $N$ and $\eps$. 

However, at least for potentials $W$ with $p=2$, Lemma \ref{lemma energy descent} is suggestive of the fact that the natural time scale should be $\eps^2$, not $\eps$. Namely, the estimate
\[
\inf_u E_\eps(u) + \frac1{2\,(\eps/(2\bar c_W))}\,\|u-u_0\|_{L^2(\Omega)}^2 \leq E_\eps(u_1) + \frac{\bar c_W}\eps \,\|u_1-u_0\|_{L^2(\Omega)}^2 \leq E(u_0)
\]
relates $u_1$ to the fully implicit minimizing movements scheme with time-step $\tau = \eps/(2\bar c_W)$, i.e.\ the implicit Euler scheme of the Allen-Cahn equation
\[
u_t = \eps\, \Delta u - \frac{W'(u)}\eps
\]
in the regular time-scale. The PDE is well-known to be $\eps$-slow, which is the reason for the usual time rescaling $t\mapsto t/\eps$. Making this analysis rigorous, however, would require a detailed analysis of the energy $E_\eps$ and the minimizing movements scheme, in particular the geometric properties along the trajectory. We return to the question of the `correct' $\eps$-scaling with a different approach in Section \ref{section thresholding}, but only for potentials with quadratic convex part.
\end{remark}

\section{Potentials with quadratic convex part}\label{section quadratic}

\subsection{Sharp Interface Initial Condition}

The previously obtained bounds are meaningless if the initial condition $u_0\in BV(\Omega; \{-1,1\})$ has the correct geometry -- a function which essentially only takes two values -- but low regularity, for instance, a sharp jump discontinuity along a hypersurface. On a grid in a finite difference or Fourier space discretization, these are the easiest initializations to prescribe, but they have infinite energy (continuum) or energy of order $O(1/h^2)$ where $h$ is the spatial length-scale of the discretization. We prove that, at least for `nice' initial sets $\{u=1\}$ and time steps of size $\tau \geq \eps$, the energy $E_\eps(u_1)$ is bounded in terms of the geometry of $u_0$ after a single gradient descent step with convex-concave splitting, independently of $\eps>0$. In particular, the bounds of Section \ref{section general bound} carry over for such initialization.

For simplicity, we focus on potentials $W$ with a quadratic convex part. To formalize our result, we must consider the potential $\overline W(u) = u^2 - 2|u|+1 = (|u|-1)^2$, which bounds any other potential with convex part $u^2$ from above (Lemma \ref{lemma w bound} in the appendix).

For future use, we compute the optimal profile $\overline \phi$ and normalizing constant $c_{\overline W}$ such that $E_\eps \xrightarrow{\Gamma} c_{\overline W}\Per$ for the potential $\overline W(z) = (1-|z|)^2$. The optimal profile is unique solution to the ODE
\begin{equation}\label{eq optimal profile 1}
\phi'' = \overline W'(\phi) = 2\big(\phi - \sign(\phi)\big), \qquad \phi(0) = 0, \qquad \lim_{x\to \pm \infty} \phi(x) = \pm 1,
\end{equation}
in this case
\begin{equation}\label{eq optimal profile 2}
\phi(x) = \sign(x) \,\big( 1- e^{-\sqrt 2\,|x|}\big).
\end{equation}
The normalizing constant is
\begin{equation}\label{eq normalizing constant}
c_{\overline W} = \int_{-1}^1 \sqrt{2\,\overline W(z)}\dz = \sqrt 2 \int_{-1}^1 1-|z|\dz = \sqrt 2.
\end{equation}

\begin{lemma}
Assume that $W(u) = u^2+ W_{conc}(u)$ where $W_{conc}$ is a concave $C^1$-function, $W\geq 0$ and $W(\pm 1) = 0$. Assume that $V\Subset \Omega$ is an open set with a $C^2$-boundary $\partial V$ and $u_0 = \chi_V - \chi_{V^c}$. Let
\[
u_1\in \argmin_{u\in H_0^1(\Omega)} \frac1{2\tau} \|u-u_0\|^2_{L^2(\Omega)} + \int_\Omega \frac\eps2\,\|\nabla u\|^2 + \frac{u^2 + W'_{conc}(u_0)\,u}{\eps}\dx
\]
be the unique solution of the gradient descent scheme with convex-concave splitting. Then $E_\eps(u_1) \leq (\sqrt 2+\eps/\tau)\,\Per(V)$. In particular, if $\tau\geq \eps$, then $E_\eps(u_1) \leq (\sqrt 2+1)\,\Per(V)$.
\end{lemma}

\begin{proof}
We begin by constructing an energy competitor. Let 
\[
r(x) = \sdist(x, \partial V) = \dist(x, V^c) - \dist(x, V)
\]
be the signed distance function from $\partial V$ which is taken to be positive inside $V$ and negative outside of $V$. Take $\overline\phi$ as in \eqref{eq optimal profile 1} and \eqref{eq optimal profile 2} and observe that $\overline\phi' \equiv \sqrt{2\,W(\overline\phi)}$. The function $v(x) = \overline \phi(r(x)/\eps)$ satisfies $v\cdot u_0 \geq 0$ and 
\[
\int_{-1}^1 \frac\eps2 \|\nabla v\|^2 + \frac{W(v)}\eps \dx \leq \int_{-1}^1 \frac\eps2 \|\nabla v\|^2 + \frac{\overline W(v)}\eps \dx \leq c_{\overline W}\, \Per(V)
\]
for $c_{\overline W}$ as in \eqref{eq normalizing constant} by combining Lemma \ref{lemma w bound} and the proof of the $\Gamma-\limsup$ inequality for the convergence $E_\eps\to\Per$. The same computation shows that
\begin{align*}
\| v -u_0\|_{L^2(\Omega)}^2 &\leq \Per(V) \,\eps \int_{-\infty}^\infty \big||\bar \phi(r)| -1 \big|\dr = 2\,\Per(V)\eps\int_0^\infty e^{-\sqrt 2\,r}\dr = \sqrt{2}\,\Per(V)\,\eps.
\end{align*}

Note that $W_{conc}'(u_0)\,u = -2\,sign(u_0)\,u \geq -2|u|$ and that the inequality is an equality if $uu_0\geq 0$. In particular, $W_{conc}'(u_0) v = -2|v|$.
Hence
\begin{align*}
\frac1{2\tau}\|u_1-u_0\|_{L^2(\Omega)}^2 + E_\eps(u_1) 
	&=\frac1{2\tau} \|u_1-u_0\|^2_{L^2(\Omega)} + \int_\Omega \frac\eps2\,\|\nabla u_1\|^2 + \frac{\overline W(u_1)}{\eps}\dx\\
	&\leq \frac1{2\tau} \|u_1-u_0\|^2_{L^2(\Omega)} + \int_\Omega \frac\eps2\,\|\nabla u_1\|^2 + \frac{u^2_1 + W_{conc}'(u_0) u_1 + 1}{\eps}\dx\\
	&\leq \frac1{2\tau} \|v-u_0\|^2_{L^2(\Omega)} + \int_\Omega \frac\eps2\,\|\nabla v\|^2 + \frac{v^2 + W_{conc}'(u_0) v + 1}{\eps}\dx\\
	&= \frac1{2\tau}\|v-u_0\|_{L^2(\Omega)}^2 +  \int_\Omega \frac\eps2\,\|\nabla v\|^2 + \frac{\big(|v|-1\big)^2}{\eps}\dx\\
	&\leq \left( \frac{\eps}\tau + 1\right)\sqrt 2\,\Per(V).\qedhere
\end{align*}
 \end{proof}
 
So, while all subsequent time steps are at best $\eps$-slow, the initial time-step from a sharp interface initial condition to a diffuse approximation relaxes to uniformly finite energy independently of $\eps$ and (almost) independently of $\tau$. 

The Allen-Cahn equation has no interesting dynamics on the `curvature motion' time-scale in one dimension \cite{carr1989metastable, fusco1989slow, bronsard1990slowness, fusco1996traveling}, but it is an interesting toy problem to observe the formation of interfaces in the first time-step from a step function.

\begin{example}\label{example first step 1d}
Let $u_0:\R\to \R$ be given by $u_0(x) = \sign(x)$ and $u_1$ the solution of the problem
\[
\frac{u_1-u_0}\tau - u_1'' + \frac{2 u_1 + W_{conc}'(u_0)}{\eps^2} = 0
\]
with boundary conditions $\lim_{x\to\pm\infty}u_1(x) = \pm 1$. This can be reformulated as
\[
-\eps^2\,u_1'' + \left(1+ \frac{\eps^2}{2\tau}\right) 2u_1 = \left(1+ \frac{\eps^2}{2\tau}\right)\,2\,\sign(x)
\]
since $W_{conc}'(u_0) = - 2u_0 = -2 \, \sign(x)$. The solution is $u_1(x) = \overline \phi\left( {\big(1+ \frac{\eps^2}{2\tau}}\big)^{1/2}\,\eps^{-1}\,x\right)$ where $\overline \phi$ is as in \eqref{eq optimal profile 1} and \eqref{eq optimal profile 2}. In particular, as $\tau\to \infty$, $u_1$ rapidly approaches the optimal transition rescaled to the correct length-scale $\eps$. Even for $\tau=1$, the error is of negligible order $\eps^3$.
\end{example}

\subsection{The Allen-Cahn Equation, the Thresholding Scheme, and \texorpdfstring{$\eps^2$-}{}slow motion}\label{section thresholding}

Consider the Allen-Cahn equation with the double-well potential $\overline W(u) = \big(|u|-1\big)^2 = u^2 - 2|u| + 1$ where we interpret $u^2$ as the convex part and $1-2|u|$ as the concave part of $\overline W$. This corresponds to a split of the Modica-Mortola functional
\[
E_\eps = F_\eps + G_\eps, \qquad F_\eps(u) = \int_\Omega \frac12\,\|\nabla u\|^2 + \frac{u^2}{\eps^2} \dx, \qquad G_\eps(u) = \int_\Omega \frac{1-2|u|}{\eps^2}\dx.
\]
With a little technical care, Example \ref{example quadratic convex-concave} shows that the convex-concave splitting gradient descent scheme with formal step-size $\tau = +\infty$ corresponds to the iteration
\begin{equation}\label{eq quasi-thresholding}
u_{n+1} = -\left(\frac2{\eps^2} - \Delta\right)^{-1} \frac{\overline W'_{conc}(u_n)}{\eps^2} = \left(1 - \frac{\eps^2}2\Delta \right)^{-1} \sign(u_n)
\end{equation}
since $\overline W_{conc}'(u) = - 2 \,\sign(u)$. It is convenient to assume that $\{u_n=0\}$ is a null set, but we may also break the tie arbitrarily by assigning $\sign(0)$ as either $-1$ or $1$. Naturally, the same remains valid if $W_{conc}$ is a smooth concave function. The iteration can be seen as a first-order (implicit Euler) approximation 
\[
\frac{v(h) - v_0}{h} = \Delta v(h) \qquad \LRa\quad (1- h\,\Delta) v(h) = v_0
\]
to the solution of the heat equation
\[
\begin{pde} (\partial_t - \Delta)U &= 0\\ U &= \sign(u_n)\end{pde}
\]
at the time $h = \eps^2/2$ and with initial condition $v_0 = \sign(u_n)$. We can split this into a two-step iteration scheme
\begin{equation}\label{eq quasi-thresholding 2}
\tilde u_n = \sign(u_n), \qquad \left(1-\frac{\eps^2}2\right) u_{n+1} = \tilde u_n,
\end{equation}
which strongly resembles the Merriman-Bence-Osher (MBO) or `thresholding' scheme
\[
\tilde u_n = \sign(u_n), \qquad u_{n+1} = \left[\left(\partial_t - \Delta\right)^{-1}\tilde u_n\right]_{t = \frac{\eps^2}2},
\]
with step-size $h = \frac{\eps^2}2$. For us, the threshold values are $-1$ and $1$ rather than $0, 1$ since we chose the wells of $W$ at $\pm 1$, and the heat kernel convolution is replaced by a first order approximation to the heat equation. Such an approximation is common on graphs or complicated domains where the heat equation cannot be solved cheaply by a known heat kernel or Fourier-domain approach.

We already know that the energy $E_\eps(u) = \int_\Omega \frac\eps 2\,\|\nabla u\|^2 + \frac{\overline W(u)}\eps\dx$ is a monotone quantity for the sequence $u_n$ in the iteration scheme. We use the analogy to find a second monotone quantity, but for the sequence $\tilde u_n$. While $E_\eps$ is the Modica-Mortola approximation to the perimeter functional, the new quantity resembles the Esedoglu-Otto approximation of the perimeter \cite{esedoglu2015threshold}.

\begin{lemma}\label{lemma esedoglu otto}
Let $\tau>0$. For $u\in L^2(\Omega)$, denote by $G_\tau u$ the solution $w$ of the PDE $(1-\tau\,\Delta) w = u$. Define
\[
F_\tau(u) = \begin{cases}\frac1{2\tau} \int_\Omega (1-u)\,G_\tau (1+u)\dx &\text{if }|u|\leq 1 \text{ a.e.}\\ +\infty &\text{else}\end{cases}, \qquad d_\tau(u, v) = \sqrt{ \int_\Omega (u-v)\,G_\tau (u-v)\dx}.
\]
Then $\tilde u_{n+1}$ in \eqref{eq quasi-thresholding 2} is the solution of the minimizing movements scheme 
\[
u\in \argmin F_\tau(u) + \frac1{2\tau}\,d_\tau^2(u, \tilde u_n)
\]
for $\tau = \frac{\eps^2}2$. In particular, $F_{\eps^2/2}(\tilde u_{n+1})\leq F_{\eps^2/2}(\tilde u_n)$.
\end{lemma}

\begin{proof}
Recall that $G_\tau: L^2(\Omega)\to L^2(\Omega)$ is a well-defined self-adjoint operator (recall Section \ref{section conventions}). We can rewrite
\begin{align*}
F_\tau(u) + \frac1{2\tau}\,d_\tau^2(u, \tilde u_n) &= \frac1{2\tau} \int_\Omega (1-u)\,G_\tau (1+u)\dx + \frac1{2\tau} \int_\Omega (u-\tilde u_n)\,G_\tau (u-\tilde u_n)\dx\\
	&= \frac1{2\tau} \int_\Omega -u\,G_\tau u +1\,G_\tau u - u\,G_\tau 1 + 1\,G_\tau 1+ u\,G_\tau u - 2\tilde u_n\,G_\tau u + \tilde u_n G_\tau\tilde u_n\dx\\
	&= \frac1{2\tau} \int_\Omega -2\,u\, G_\tau \tilde u_n\dx + \frac1{2\tau} \int_\Omega \tilde u_n\,G_\tau\tilde u_n- 1\,G_\tau 1\dx
\end{align*}
under the assumption that $|u|\leq 1$ almost everywhere. The function $u$ should be maximally aligned with $G_\tau\tilde u_n$, suggesting that
\[
u = \sign\big(G_\tau \tilde u_n\big),
\]
which coincides with the expression for $\tilde u_{n+1}$ if $\tau = \frac{\eps^2}2$.
\end{proof}

 We assert that standard proofs apply to show that \eqref{eq quasi-thresholding} is a time-discretization of MCF with time-step size $\eps^2/2$:
\begin{enumerate}
\item The time-steps respect the maximum principle/inclusions of sets, which is the key property for convergence to viscosity solutions.

\item The constructions of BV-type solutions in the fashion of \cite{esedoglu2015threshold, laux2016convergence, laux2020thresholding} can be recovered based on Lemma \ref{lemma esedoglu otto}.
\end{enumerate}

Indeed, it is known that thresholding dynamics of the more general type
\begin{equation}\label{eq thresholding convolution}
\tilde u_{n+1, \tau} = K_\tau* u_{n,\tau}, \qquad u_{n+1, \tau} = \sign(\tilde u_{n+1, \tau}), \qquad K_\tau(x) = \tau^{-d/2} K\left(\frac{x}{\sqrt\tau}\right)
\end{equation}
converge to mean curvature flow on the whole space \cite{ishii1995generalization, ishii1999threshold} when taken as time steps of length $\tau>0$, i.e.\ $u_\tau(s) = u_{\lfloor s/\tau\rfloor}$  under the assumptions that $K$ is radially symmetric and
\[
\int_{\R^d}K(x)\dx =1, \qquad \int_{\R^d}K(x) \,\|x\|^2\dx < +\infty,
\]
as well as a technical condition discussed in Appendix \ref{appendix thresholding}. On the entire space in three dimensions, we observe that
\[
(1-\tau \Delta) u = f\qquad\LRa\quad \left(\frac1\tau - \Delta\right) u = \frac1\tau f \qquad \LRa \quad u = K_\tau * \left(\frac f \tau\right) = \frac{K_\tau}{\tau} f
\]
where $K_\tau$ is the Green's function for the screened Poisson equation
\[
K_\tau(x) = \frac{\exp(-\|x\|/\sqrt \tau)}{4\pi\,\|x\|} = \frac1{\tau^{1/2}}\,\frac{\exp(-\|x/\sqrt \tau\|)}{4\pi\,\|x/\sqrt \tau\|}.
\]
Since the convolution is with $\tau^{-1}K_\tau = \tau^{-3/2} K_1(x/\sqrt\tau)$, we conclude that at least on the whole space $\R^3$, the iteration \eqref{eq quasi-thresholding} falls precisely into the framework of \eqref{eq thresholding convolution} with $\tau = \frac{\eps^2}2$ and we conclude that the effective time scale of the convex-concave splitting Allen-Cahn solver is indeed $\eps^2/2$.

A rigorous analysis of thresholding dynamics like \eqref{eq quasi-thresholding} outside of this special case is beyond the scope of our study of the convex-concave splitting scheme. The scaling in $\eps$ derives from the analysis of the quadratic part $W_{vex}(u) = u^2$ of the potential and similar results would hold for smoother potentials with quadratic convex part, such as $W_R$ in Appendix \ref{appendix W_R}, leading to `smooth thresholding' schemes.

The analytic results are valid only for the limiting iteration with formally infinite step size. This is the least restrictive minimization problem since the $L^2$-distance to the previous iterate is not controlled. We have not rigorously demonstrated that this is optimal, and that there is no finite `sweet spot' for the step size $\tau$. We show in Figure \ref{figure no of iterations} that this is not expected.

\begin{figure}
\includegraphics[width =.32\textwidth]{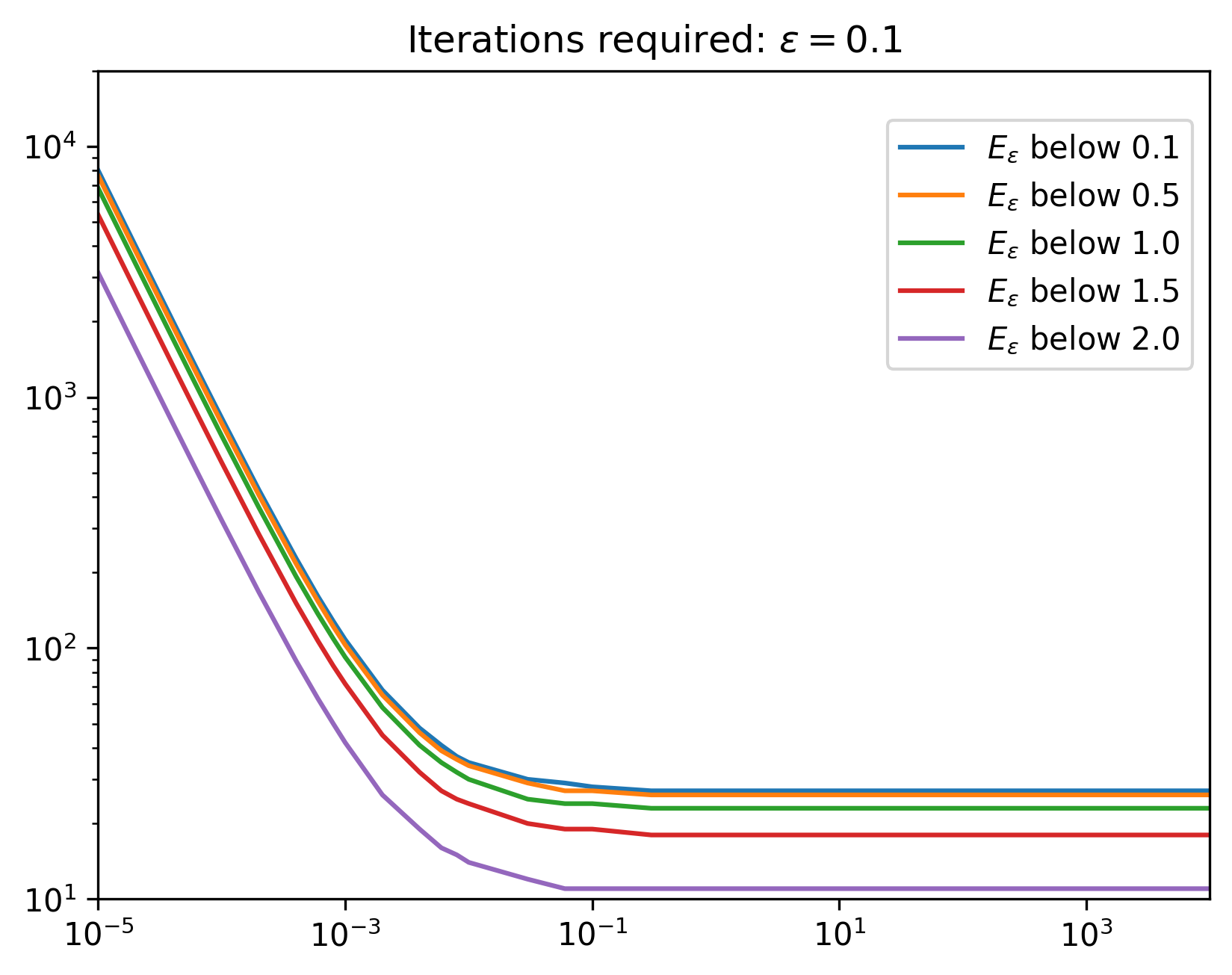}\hfill
\includegraphics[width =.32\textwidth]{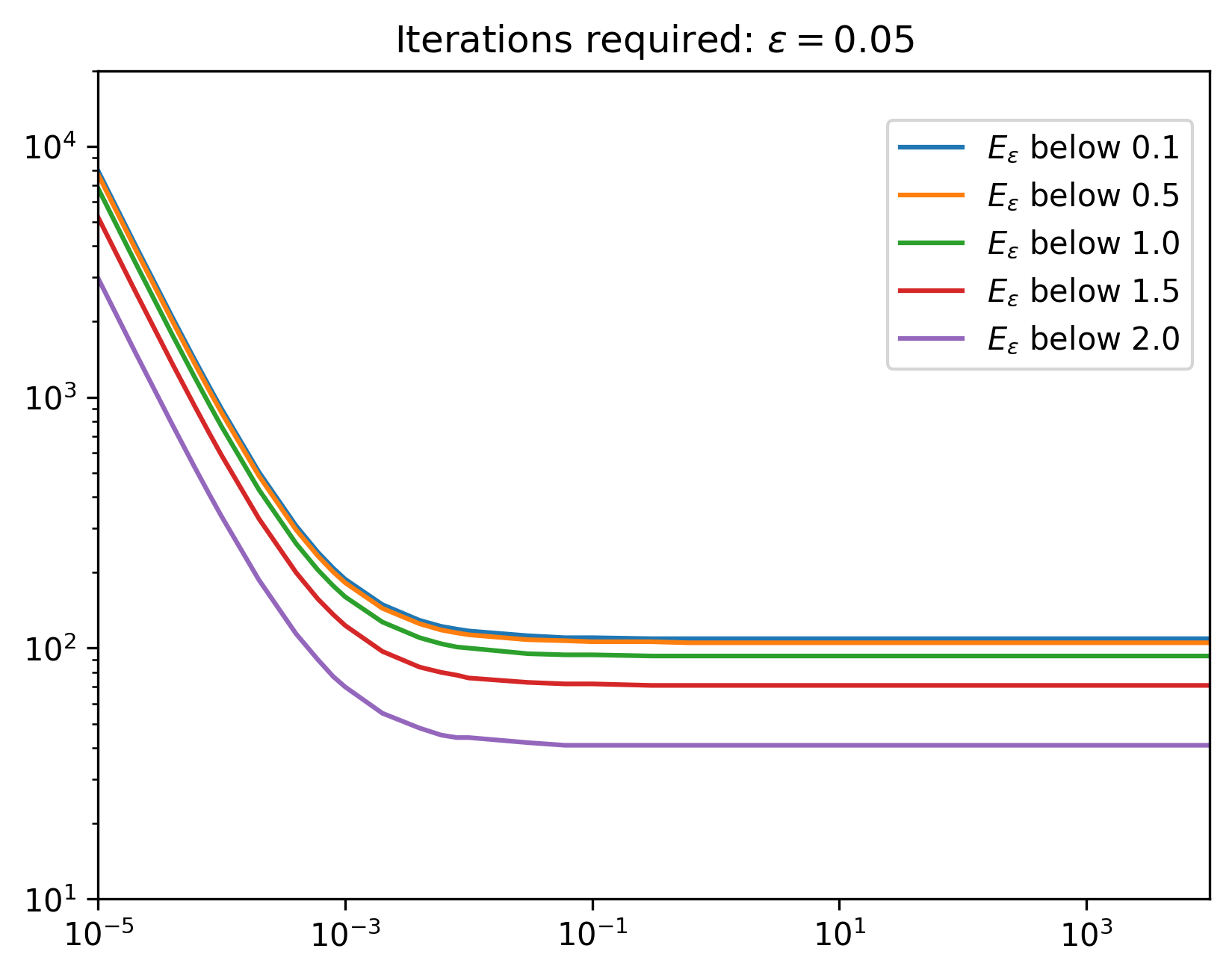}\hfill
\includegraphics[width =.32\textwidth]{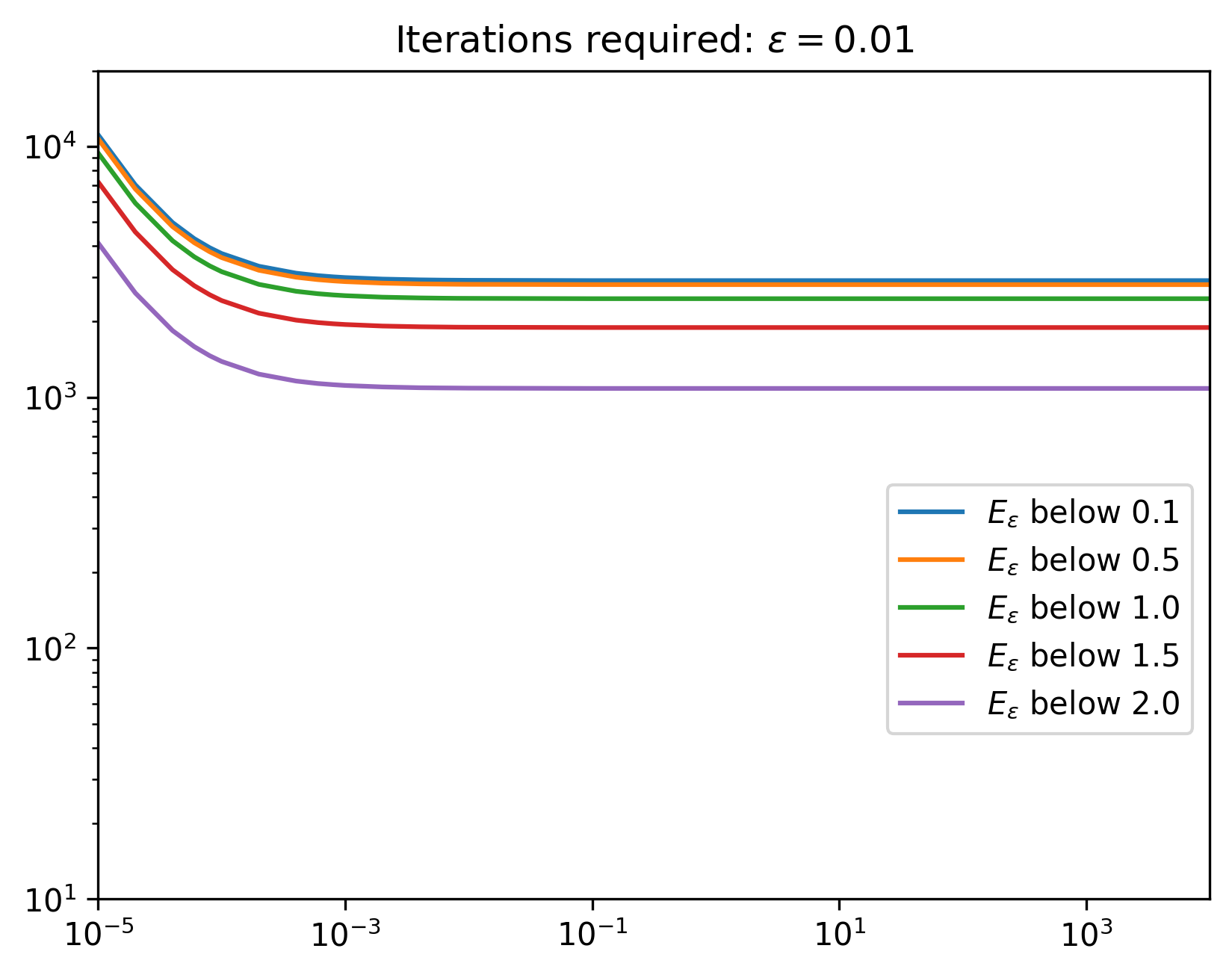}
\caption{
The number of convex-concave splitting gradient descent iterations required until $E_\eps$ falls below a given level with $\eps = 0.1$ (left), $\eps = 0.05$ (middle) and $\eps = 0.01$ (right) as a function of time step size. Initially, increasing the step size reduces the number of iterations required, but the effect tapers of at a (surprisingly small) finite threshold. The threshold is lower for smaller $\eps$. The setting is otherwise as in Figure \ref{figure explore eps and tau}.
\label{figure no of iterations}
}
\end{figure}

\section{Singular Potentials}\label{section barrier potentials}

In this section, we consider two other classes of potentials with a particularly simple convex part. Our previous choice to study potentials of the form $W(u) = u^2 + W_{conc}(u)$ was motivated by the observation that the convex part of the energy $E_\eps$ can be taken to be quadratic, simplifying the minimization problem in the implicit part of the time step to solving {\em the same} linear system with varying right hand sides in each step. However, our analysis demonstrates that while the steps are fast, the resulting scheme is slow due to overstabilization. 

Lemma \ref{lemma energy descent} suggests that functionals for which $W_{vex}$ is strongly convex or $W_{conc}$ is strongly concave are always going to be slow. In this section, we consider the opposite extreme where the second derivative is maximally concentrated. Namely, we consider `barrier potentials'  
\[
W_{bar}(u) = 1-|u| + \infty \cdot 1_{\{|u|>1\}}\qquad\text{ and $\ell^1$-type potentials}\quad W_\alpha(u) = \begin{cases} 1-|u| & |u|\leq 1\\ \alpha(|u|-1) & |u|>1\end{cases}
\]
for $\alpha \geq 0$, which have second derivative zero outside of $\{-1,0,1\}$. These potentials pose analytic challenges considering e.g.\ the existence of solutions to the Allen-Cahn equation. They are considered primarily as an analytic complement for the opposite extreme in the class of potentials, but we note that non-smooth potentials are known to have favorable numerical properties in certain applications \cite{bosch2018generalizing}. For analytic results in this direction, see e.g.\ \cite{budd2019graph, budd2021classification}. 

We illustrate why we consider both types of potentials simultaneously for the convex-concave splitting scheme with 
\[
W_{conc}(u) = 1-|u| \qquad\text{and}\quad W_{vex, bar}(u) = \infty\cdot 1_{\{|u|>1\}}, \quad W_{vex,\alpha} = (\alpha+1) \,\max\{|u|-1, 0\}
\]
respectively.

\begin{lemma}
Let $u_0\in H^1(\Omega)$ and $|u_0|\leq 1$ almost everywhere. Consider the functions 
\begin{align*}
u_\infty &= \argmin_u \int_\Omega \frac\eps2\,\|\nabla u\|^2 + \frac{W_{bar}'(u_0)\,u + \infty\cdot 1_{\{|u>1\}}}\eps \hspace{1.65cm}+ \frac1{2(\tau/\eps)}\,\|u-u_0\|^2\dx \\
u_1 &= \argmin_u\int_\Omega \frac\eps2\,\|\nabla u\|^2 + \frac{W_{\alpha}'(u_0)u + (\alpha+1)\max\{|u|-1, 0\}}\eps \:\:+ \frac1{2(\tau/\eps)}\,\|u-u_0\|^2\dx 
\end{align*}
where $u$ ranges over the same function classes in both lines: $H^1(\Omega)$ if Neumann or periodic boundary conditions are considered and an affine subspace if Dirichlet boundary conditions are considered.

If $\alpha\geq 0$ and the Dirichlet boundary condition (if present) takes values in $[-1,1]$, then $u_1 = u_\infty$. 
\end{lemma}

We define $u/|u|$ arbitrarily but consistently for $u=0$ in case $\{u_0=0\}$ is not a null set.

\begin{proof}
It suffices to show that $|u_1|\leq 1$ almost everywhere since $W_{bar}'(u_0) \equiv W_\alpha'(u_0)$, i.e.\ the two expressions coincide for all $u$ which take values in $[-1,1]$. Thus, if the $u$ can be seen to take values in $[-1,1]$ a priori also for the second problem, the minimization problems coincide. Consider 
\[
\tilde u_1 = \Pi_{[-1,1]}(u_1) = \frac{u_1}{\max\{|u_1|,1\}} = \min \big\{1, \max\{-1, u_1\}\big\}.
\]
Clearly $\|\tilde u_1 - u_0\|_{L^2} \leq \|u_1-u_0\|_{L^2}$ since $u_0$ takes values in $[-1,1]$, i.e.\ the difference is pointwise decreasing under projection. The inequality is strict unless $\tilde u_1\equiv u_1$. Additionally, we find that the Dirichlet energy of $\tilde u_1$ is lower than that of $u_1$ \cite[Satz 5.20]{dobrowolski2010angewandte}. Finally, we note that if $u_1\geq 1$, then
\begin{align*}
 W_\alpha'(u_0)u_1 +  (\alpha+1)\max\{|u_1|-1, 0\} = \big(W_\alpha'(u_0)+1 + \alpha\big) u_1 - (\alpha+1) 
	\geq \big(W_\alpha'(u_0)+1 + \alpha\big) - (\alpha+1)
\end{align*}
since $W_\alpha'\in [-1,1]$, i.e.\ the first term is non-negative. This shows that also the final term in the energy is decreasing when replacing $u_1$ by $\tilde u_1$ (strictly if $\alpha>0$ and $|u_1|$ is not bounded by $1$ almost everywhere).
\end{proof}

The same statement remains valid for any $W=W_{vex}+W_{conc}$ satisfying $|W_{conc}'|\leq 1$ inside $(-1,1)$ and $u\,W_{vex}'(u)\geq (|u|-1)$ outside $(-1,1)$. Thus, we can either minimize 
\[
\frac1{2\tau}\|u-u_n \|_{L^2(\Omega)}^2 + \int_\Omega \frac12\,\|\nabla u\|^2 + \frac{W_{conc}'(u_n)}{\eps^2}\,u\dx 
\]
in the class $\{u \in H^1(\Omega) : |u|\leq 1\}$ (`box constraints') or employ an $L^1$-penalty which is only active if $|u|>1$. Naturally, either minimization problem is more difficult than for potentials with quadratic convex part, and thus only worth the additional investment if the effective time step limitation is less severe. In a simple but non-trivial example, we illustrate that this cannot be expected, again in the simpler case $\tau = +\infty$ and for open sets in Euclidean spaces $\R^d$ with $d\geq 3$. Considering Figure \ref{figure explore eps and tau}, we conjecture that the same is true in $\R^2$ and for $\tau < \infty$.

\begin{example}\label{example obstacle}
Denote by $B_r= B_r(0)$ the ball or radius $r$ centered at the origin. Assume that $u_0\in H^1(B_R)$ is any function such that $1\geq u_0>0$ in $B_r(0)$ and $-1\leq u_0<0$ in $B_R\setminus \overline{B_r}$ for $R>r>0$, i.e.\ $W_{bar}'(u_0) = -1$ on $B_r$ and $W_{bar}'(u_0) = 1$ on $B_R\setminus B_r$. We consider the limiting convex-concave splitting time-stepping scheme for $\tau = +\infty$, for which the next step from $u_0$ is characterized as the minimizer of the functional
\[
\int_{B_R}\frac\eps2 \,\|\nabla u\|^2 - \frac{W_{bar}'(u_0)}\eps\,u\dx,
\]
over
\[
\left\{u\in H^1(\Omega) : |u|\leq 1\right\} \qquad\text{or}\quad \{u \in H^1(B_R) : |u| \leq 1 \text{ and } u=-1\text{ on }\partial B_R\},
\]
depending on whether we consider Dirichlet or Neumann boundary conditions. This is the perhaps simplest instance of a {\em double obstacle problem}, one of the most classical variational inequalities \cite{kinderlehrer2000introduction}.

It is well-known that a unique solution to the double-obstacle problem exists and is $C^{1,1}$-regular under very general circumstances, but in general not $C^2$-regular \cite{gerhardt1985global, lee2019regularity}. The solution satisfies the Euler-Lagrange equation $-\Delta u = -W_{bar}'(u_0)/\eps^2$ on the `free' set $\{x: u(x) \in (-1,1)\}$ and the inequalities $-\Delta u \geq -W_{bar}'(u_0)/\eps^2$ on the coincidence set $\{u = -1\}$ with the lower obstacle, $-\Delta u\leq - W_{bar}'(u_0)/\eps^2$ on the coincidence set $\{u=1\}$ with the upper obstacle, which express that there is no energetic incentive to `detach' from the obstacle in the admissible direction, i.e.\ into $(-1,1)$.

In particular, since $\Delta u\equiv 0$ in either coincidence set, the solution $u$ coincides with the upper obstacle at most on a subset of $B_r$ and with the lower obstacle at most in $B_R\setminus B_r$. In our case, we posit that the problem can be summarized as
\[
\begin{pde}
u &= 1 &\text{in }B_{r_i}\\
-\Delta u &= \frac1{\eps^2} &\text{in }B_r\setminus B_{r_i}\\
-\Delta u &= -\frac1{\eps^2} &\text{in }B_{r_o}\setminus B_{r}\\
u &= -1 &\text{in }B_{R\setminus r_o}
\end{pde}
\]
for radii $0 < r_i < r <r_o < R$ if $R\gg r\gtrsim 1$, i.e.\ if the radii are large enough to require the presence of both coincidence sets. We will determine $r_i, r_o$ below and see that the restriction is in fact minimal and that naturally $r_i, r_o = r+ O(\eps)$.

Note that only $W_{bar}'(u_0)$ is used and the exact function $u_0$ is irrelevant. The problem therefore becomes radially symmetric and since the solution is unique, it exhibits the same symmetry, even if $u_0$ does not. This simplification only arises for $\tau = +\infty$. By abuse of notation, we denote $u(s) = u(s\,e_1)$.\footnote{\ For ourradial solutions it is easy to see that the variational inequality would be violated if $u$ were not $C^1$-smooth at the spheres $\partial B_s$ with $s\in \{r_i, r, r_o\}$ without an appeal to abstract theory. For instance, since $u\leq 1$, $-\Delta u$ would be a negative measure at the sphere $\partial B_{r_i}$ where the maximum is attained, contradicting the variational inequality and incentivizing `detaching'.} 

We can solve the problem explicitly, using the fact that the only radially symmetric harmonic functions on $\R^d$ are constants and the fundamental solution 
\[
G_d(x) = \begin{cases} \frac1{2\pi} \,\log(\|x\|) & d=2\\ -\,\frac{1}{d(d-2)\,\omega_d} \|x\|^{2-d} & d\geq 3\end{cases}
\] 
of the Poisson equation where $\omega_d$ denotes the volume of the unit ball in $d$ dimensions, and that $\Delta \|x\|^2 = 2d$. There is no other radial solution $v$ to the equation $\Delta v= 1$ other than a shift by a constant of $G_d$. The problem can be solved explicitly by an elementary but lengthy calculation which we present in Appendix \ref{appendix obstacle}. See Figures \ref{figure obstacle} and \ref{figure comparison to optimal} for an illustration of the transitions found by the convex-concave splitting with varying $d$ and $\eps$ and Figure \ref{figure radii varying eps} for an illustration of the inner radius $r_i$, the outer radius $r_o$ and the `new' radius $r_{new} = u^{-1}(0)$ as a function of $\eps$. Evidently, $r_{new,\eps} = r + O(\eps^2)$, as for potentials with quadratic convex part. For an analytic demonstration, see Appendix \ref{appendix obstacle}. A numerical approximation for the potential with smoother concave part $1-u^2$ is given in Figure \ref{figure other potentials} and Appendix \ref{appendix singular numerics}.

\begin{figure}
\includegraphics[width =.24\textwidth]{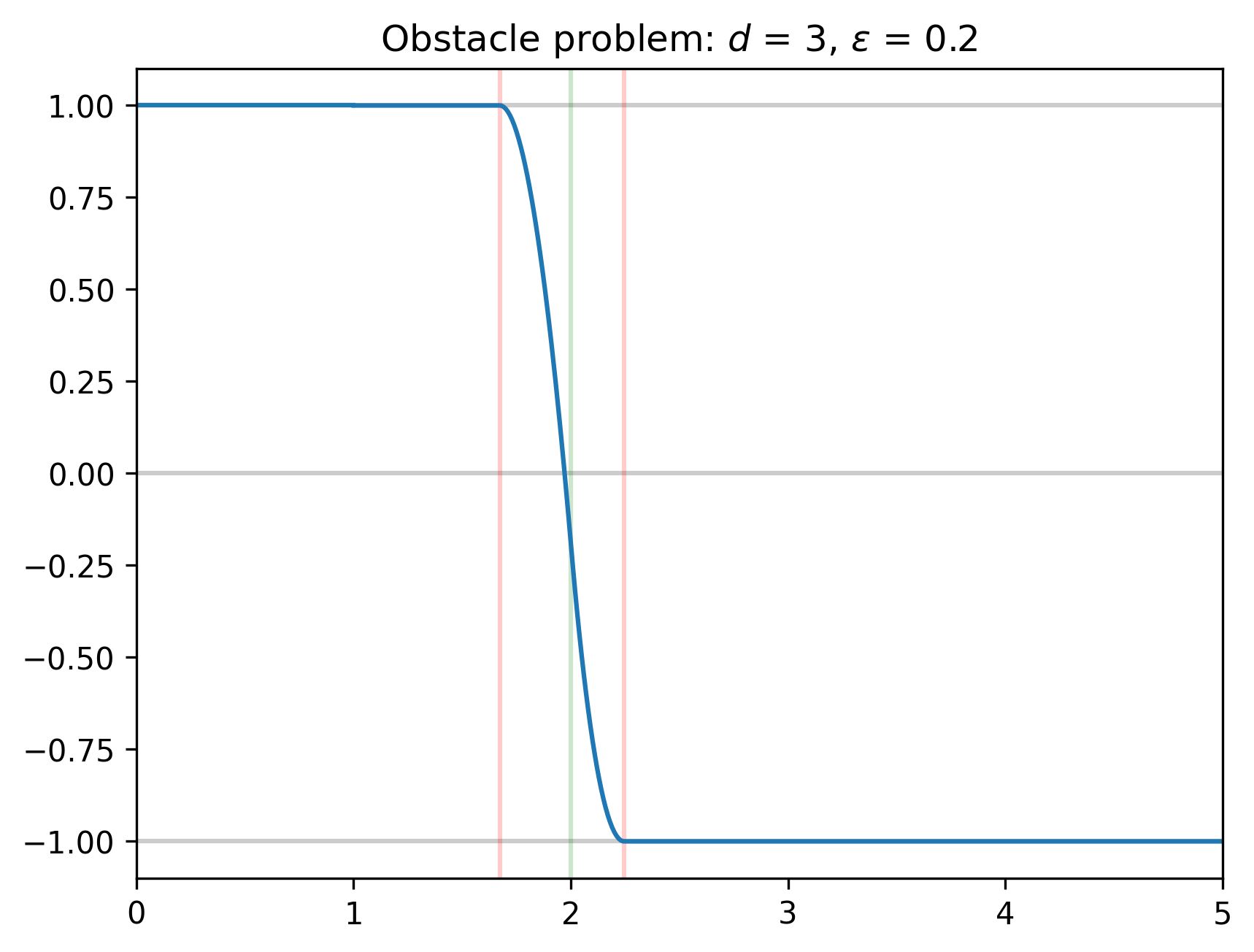}\hfill
\includegraphics[width =.24\textwidth]{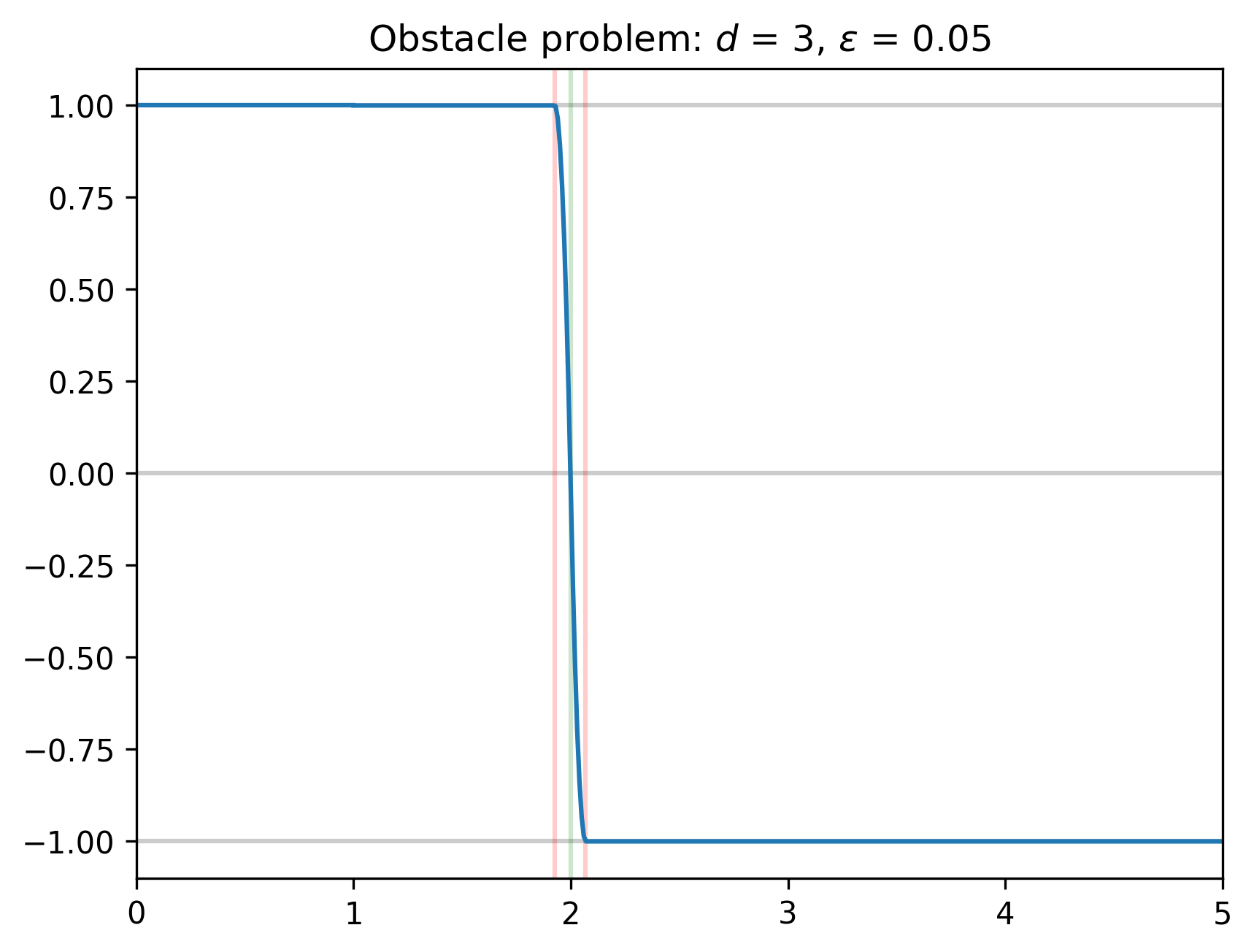}\hfill
\includegraphics[width =.24\textwidth]{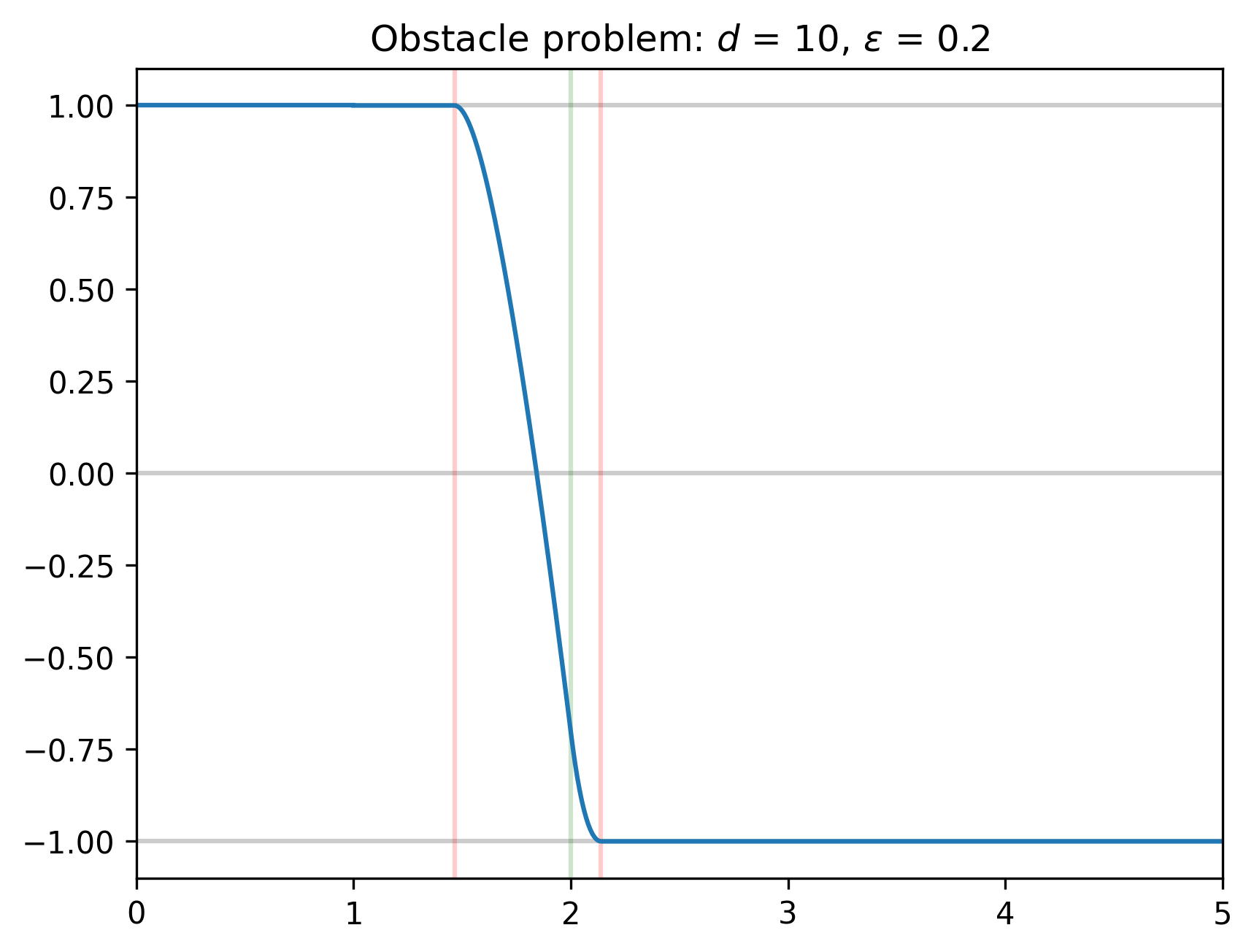}\hfill
\includegraphics[width =.24\textwidth]{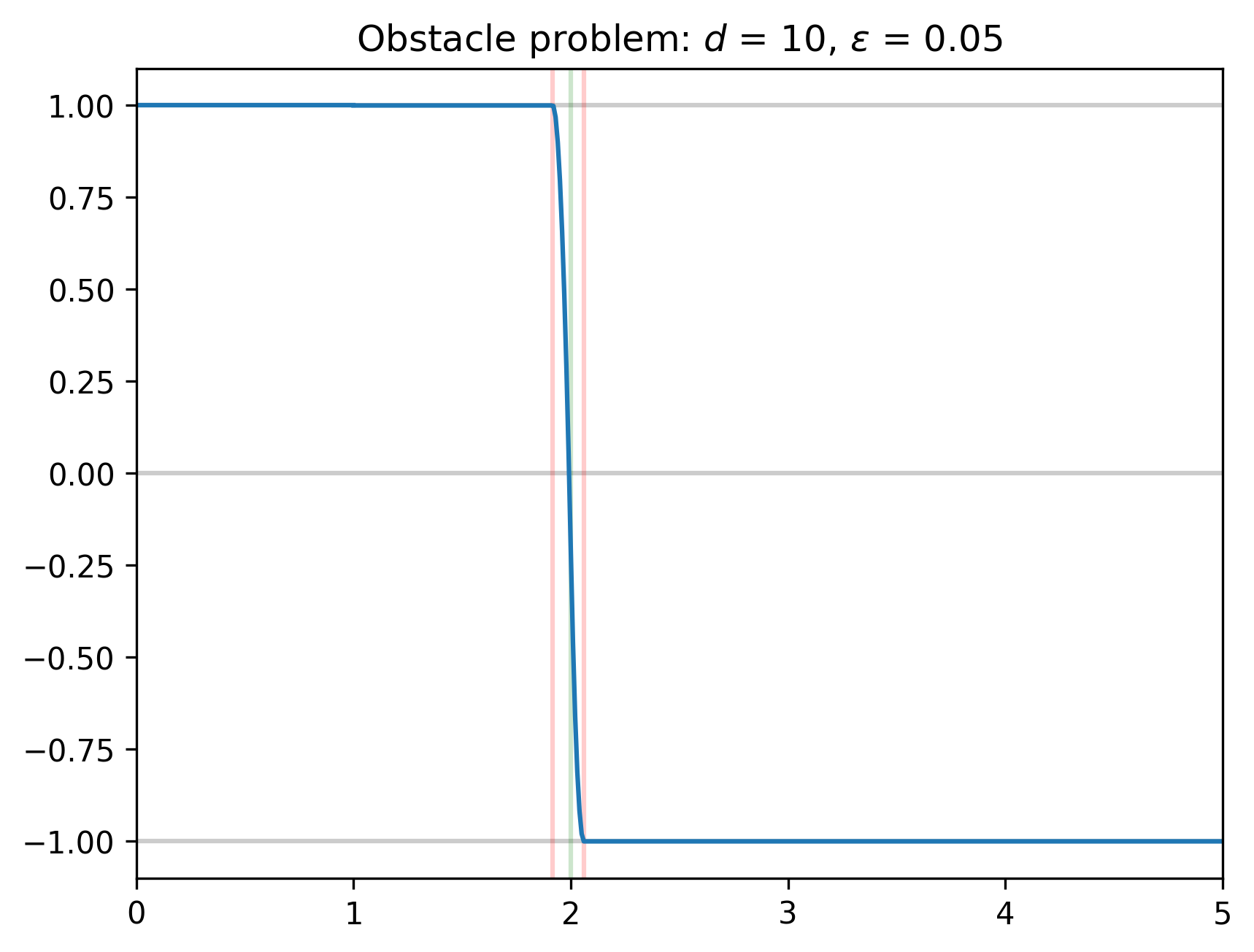}\hfill
\caption{
Solutions to the time-stepping problem of Example \ref{example obstacle} for the evolution of a ball with radius $2$ in radial direction in dimension $d=3$ (left two) and dimension $d=10$ (right two) with varying values of $\eps$. In the first and third plot, we select $\eps = 0.2$ large, in the second and fourth smaller with $\eps = 0.05$. The initial radius $r=2$ is marked by a vertical green line, the radii $r_i$ and $r_o$ are marked by vertical red lines. The intersection of $u$ with the (grey) line $y=0$ moves much farther from the green line for large $\eps$.
\label{figure obstacle}
}
\end{figure}

\begin{figure}
\includegraphics[width =.24\textwidth]{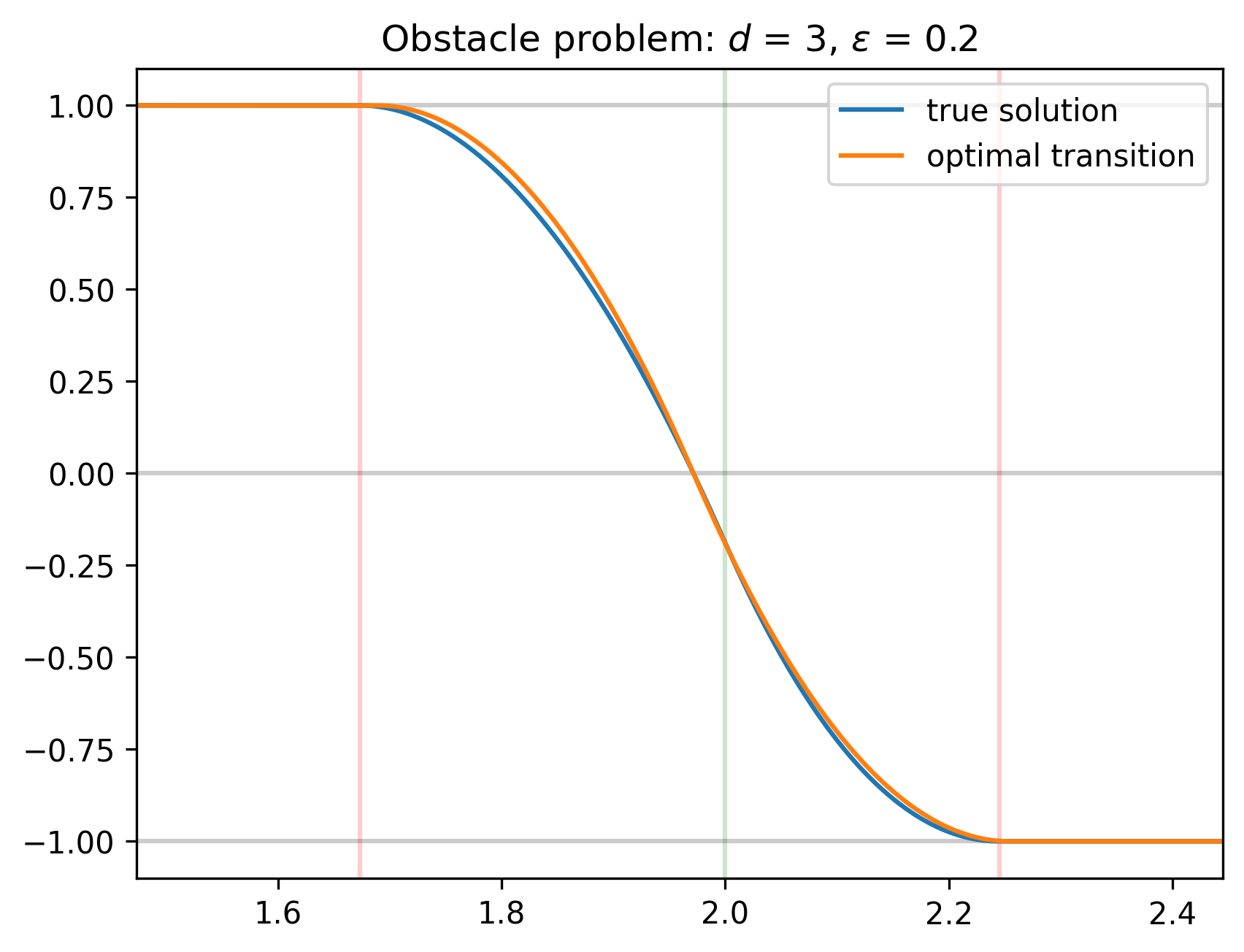}\hfill
\includegraphics[width =.24\textwidth]{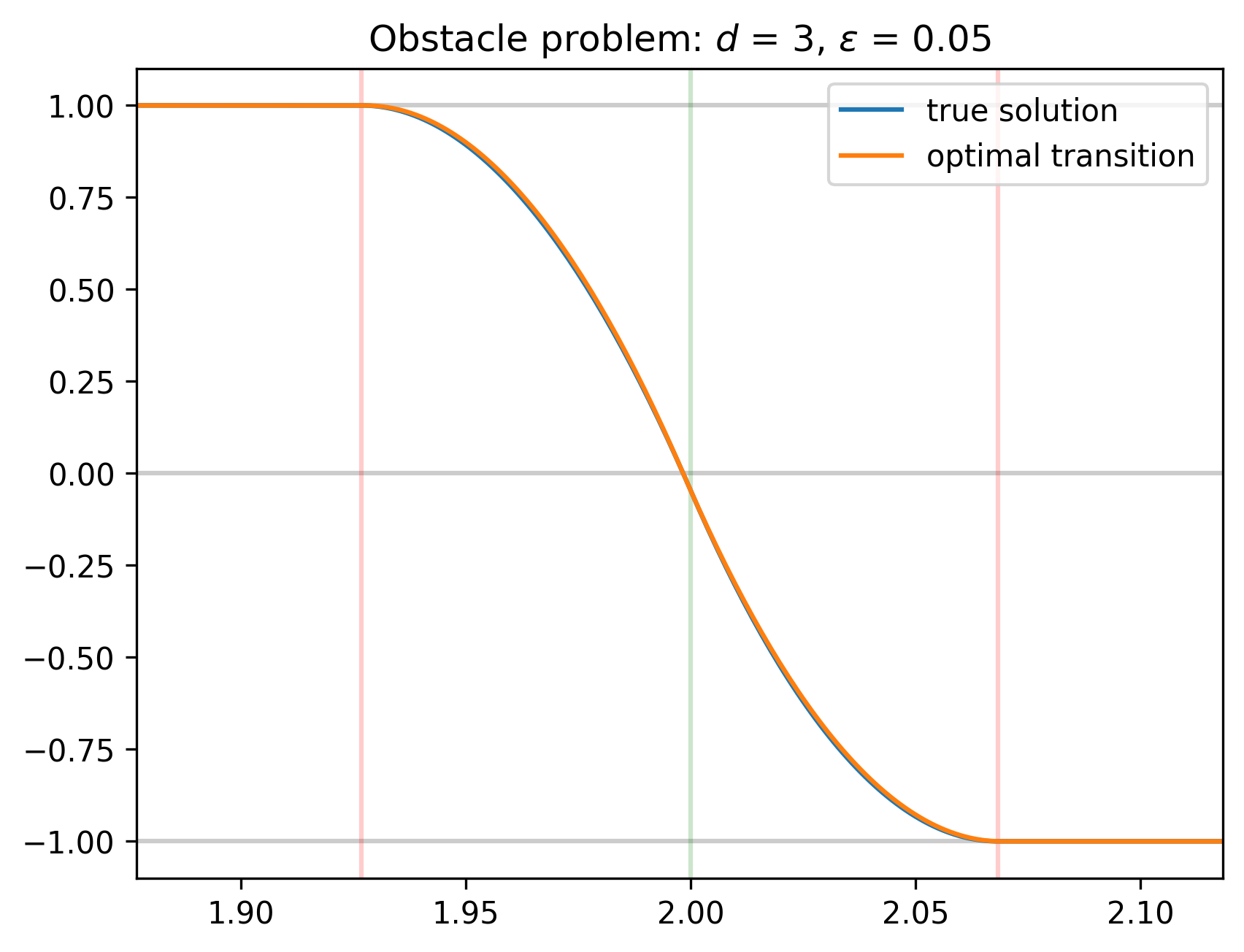}\hfill
\includegraphics[width =.24\textwidth]{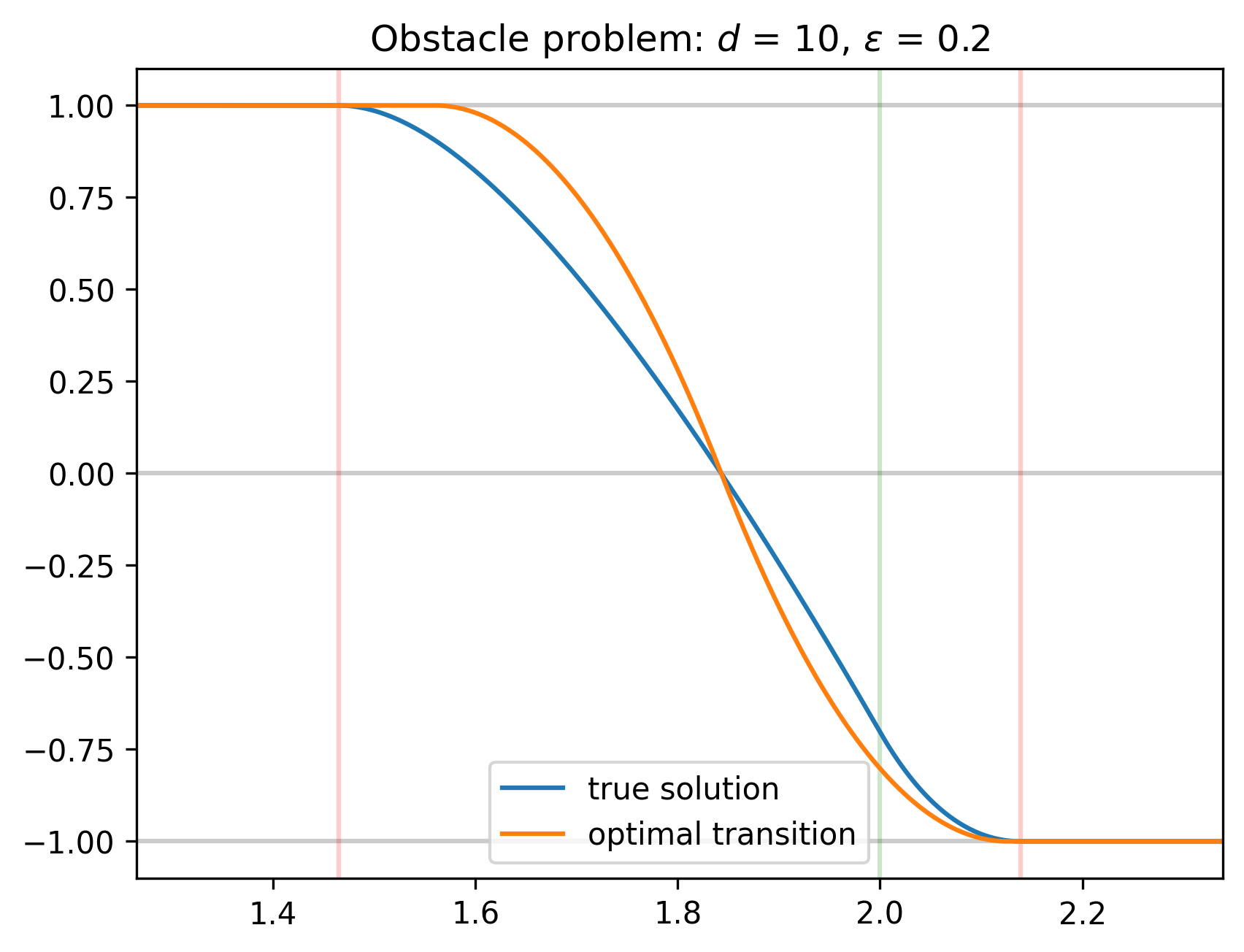}\hfill
\includegraphics[width =.24\textwidth]{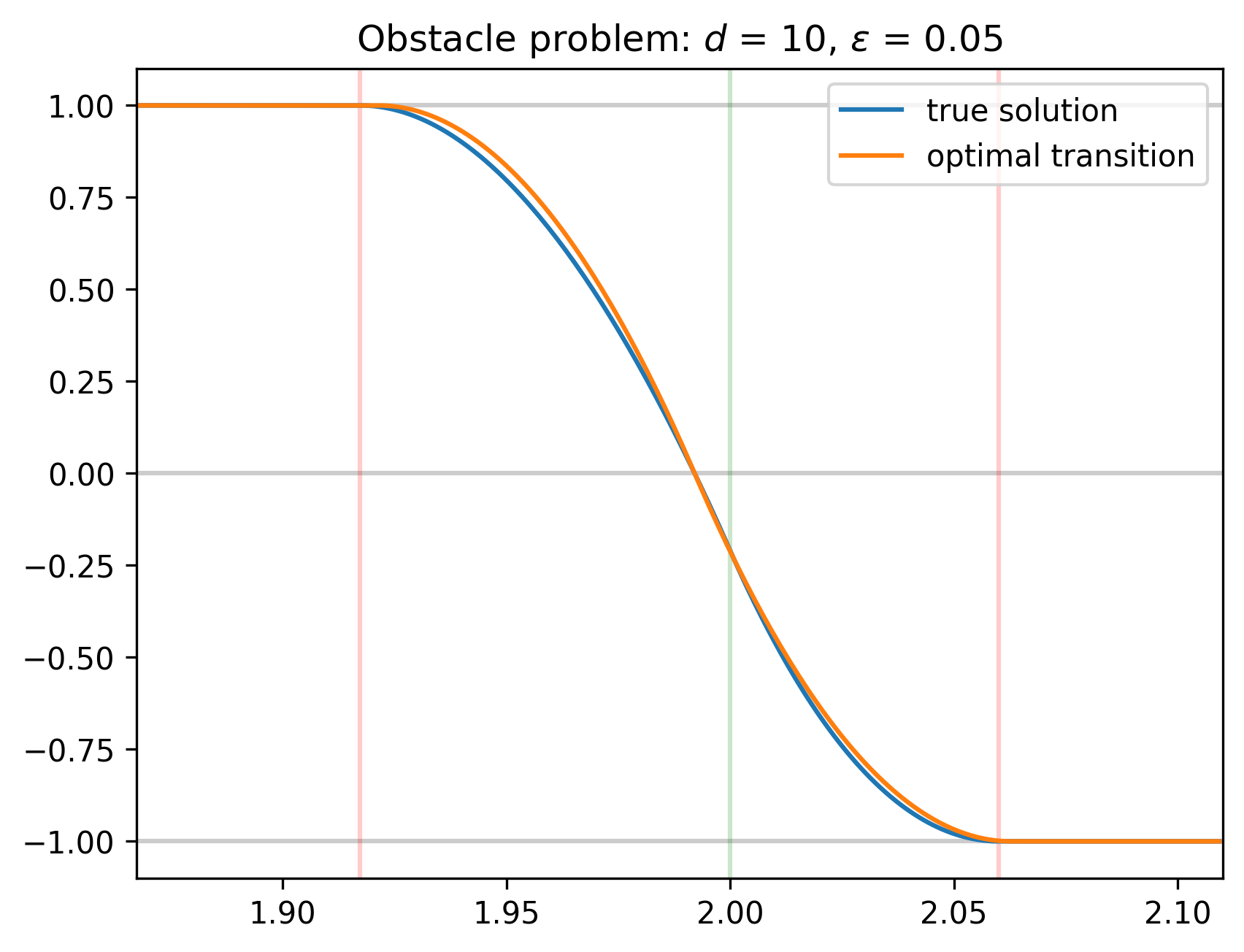}\hfill
\caption{
Solutions to the time-stepping problem of Example \ref{example obstacle} with barrier potentials, zoomed in to length scale $\eps$, as well as suitably scaled versions of the optimal transition, centered at the $x$-axis intercept of $u$. The initial condition is a ball of radius $r=2$ in three dimensions (left two) and dimension $d=10$ (right two). In the first and third plot, we select $\eps = 0.2$ large, in the second and fourth $\eps = 0.05$. The initial radius $r=2$ is marked by a vertical green line, the radii $r_i$ and $r_o$ are marked by vertical red lines. Notably, $u$ agrees with the optimal transition to high degree if $\eps$ is small or $d=3$ and $\eps$ is moderately small.
\label{figure comparison to optimal}
}
\end{figure}

\begin{figure}
\includegraphics[width =.32\textwidth]{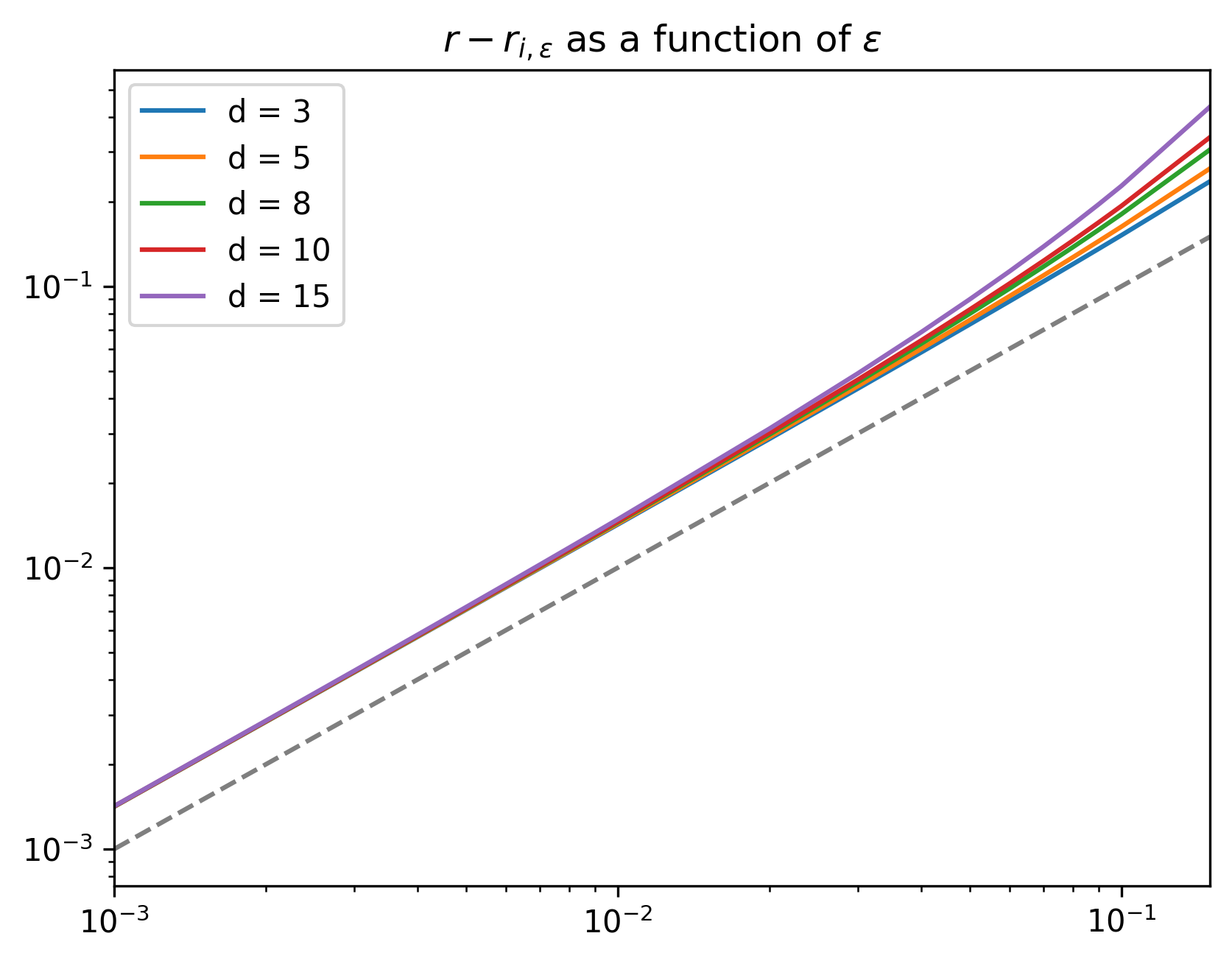}\hfill
\includegraphics[width =.32\textwidth]{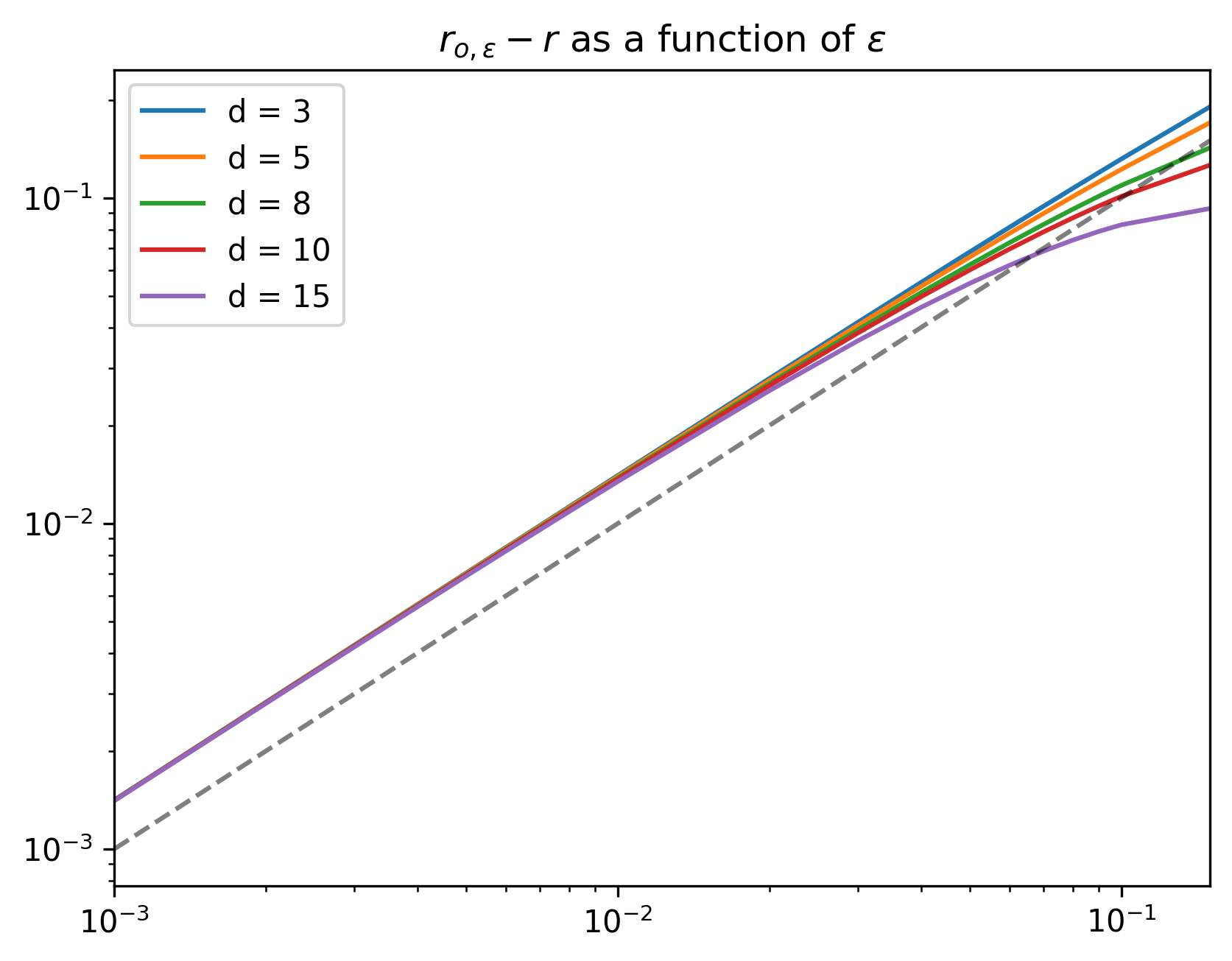}\hfill
\includegraphics[width =.32\textwidth]{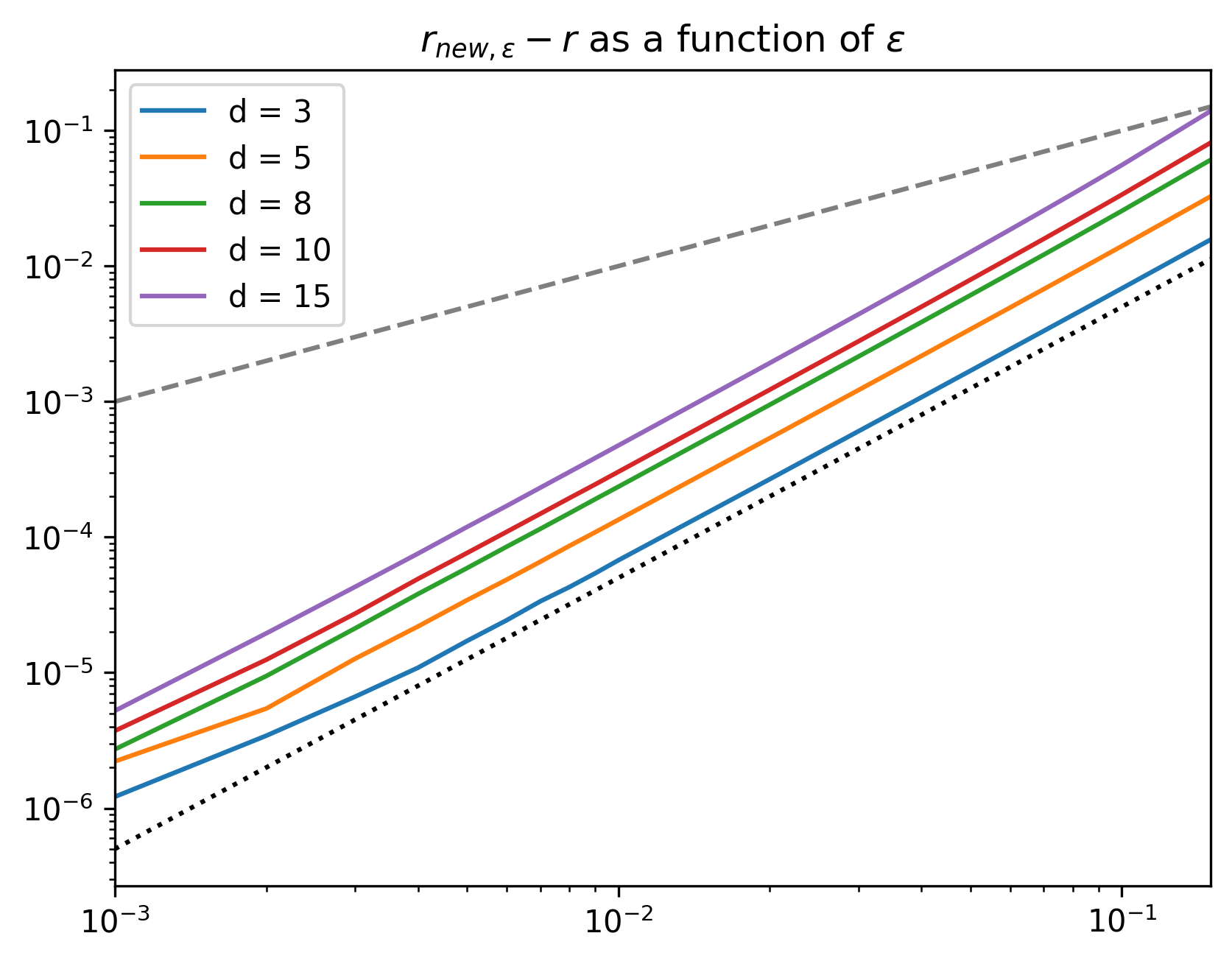}
\caption{
$r-r_{i,\eps}$ (left), $r_{o,\eps}-r$ (middle) and $r-r_{new,\eps}$ (right) as a function of $\eps$ for various dimensions. The dashed black line $\psi(\eps) = \eps$ is given for reference. Clearly, the distance of both $r_i$ and $r_o$ from $r$ grows {\em linearly} with $\eps$, independently of dimension. This is to be expected for a function which resembles an optimal profile for a radius within distance $O(\eps)$ from $r$. For $r_{new}-r$, on the other hand, the growth is {\em quadratic} in $\eps$. For reference, we give the additional line $\tilde\psi(\eps) = \eps^2/2$ (dotted, black) in the right plot. We selected $r=2$ for all experiments as in Figures \ref{figure obstacle} and \ref{figure comparison to optimal}.
\label{figure radii varying eps}
}
\end{figure}

\end{example}

\section{The standard potential}\label{section numerical}
In this section, we consider the classical potential 
\[
W(u) = (u^2-1)^2 = \underbrace{u^4}_{=:W_{vex}} + \underbrace{1-2u^2}_{=:W_{conc}}
\]
numerically.
As previously, we implement the Laplacian in the Fourier domain on a uniform $n\times n$-grid with $n=512$ with periodic boundary conditions, i.e.\ the spatial resolution is $h = 1/512 \approx 0.002$ for the unit square. The initial condition is $2\cdot \chi_E - 1$ where $E$ is a circle of radius $r_0 = 0.4$. 

In each time step, a convex minimization problem must be solved for the partially implicit time step. In our experiments, this was implemented by Newton-Raphson iteration. While we did not optimize code for performance, we note that a single time-step with this potential is expected to take significantly higher computational effort than using a potential with quadratic convex part or a barrier potential.

Empirically, we observe the splitting to be not just $\eps$-slow as guaranteed by Corollary \ref{corollary simple}, but $\eps^2$-slow in Figure \ref{figure other potentials}. The nominal time-step size is $\tau = 10^5$ as in Figure \ref{figure explore eps and tau}.

The optimal transition for $W$ satisfies
\[
\begin{pde} u' &= \sqrt{2(1-u^2)^2} & x\neq 0\\ u&= 0 &x=0\end{pde} \qquad \Ra\quad u'(x) = \sqrt{2}\,\big(1-u^2(x)\big)\qquad \Ra\quad u(x) = \tanh(\sqrt 2\,x).
\]
If we define the width of the interface as the width of the transition from e.g.\ $-0.95$ to $0.95$, we find that the transition width in the $\eps$-scaling is
\[
2\cdot\frac{\tanh^{-1}(0.95)}{\sqrt 2} \,\eps \approx 2.6\eps.
\]
Since the smallest $\eps$ considered in our experiments is $\eps = 0.01$, the transition width is larger than $0.026$ and a transition is resolved over approximately 10 grid cells. Such a resolution is generally considered sufficient for numerical purposes.

\begin{figure}
\includegraphics[width = .33\linewidth]{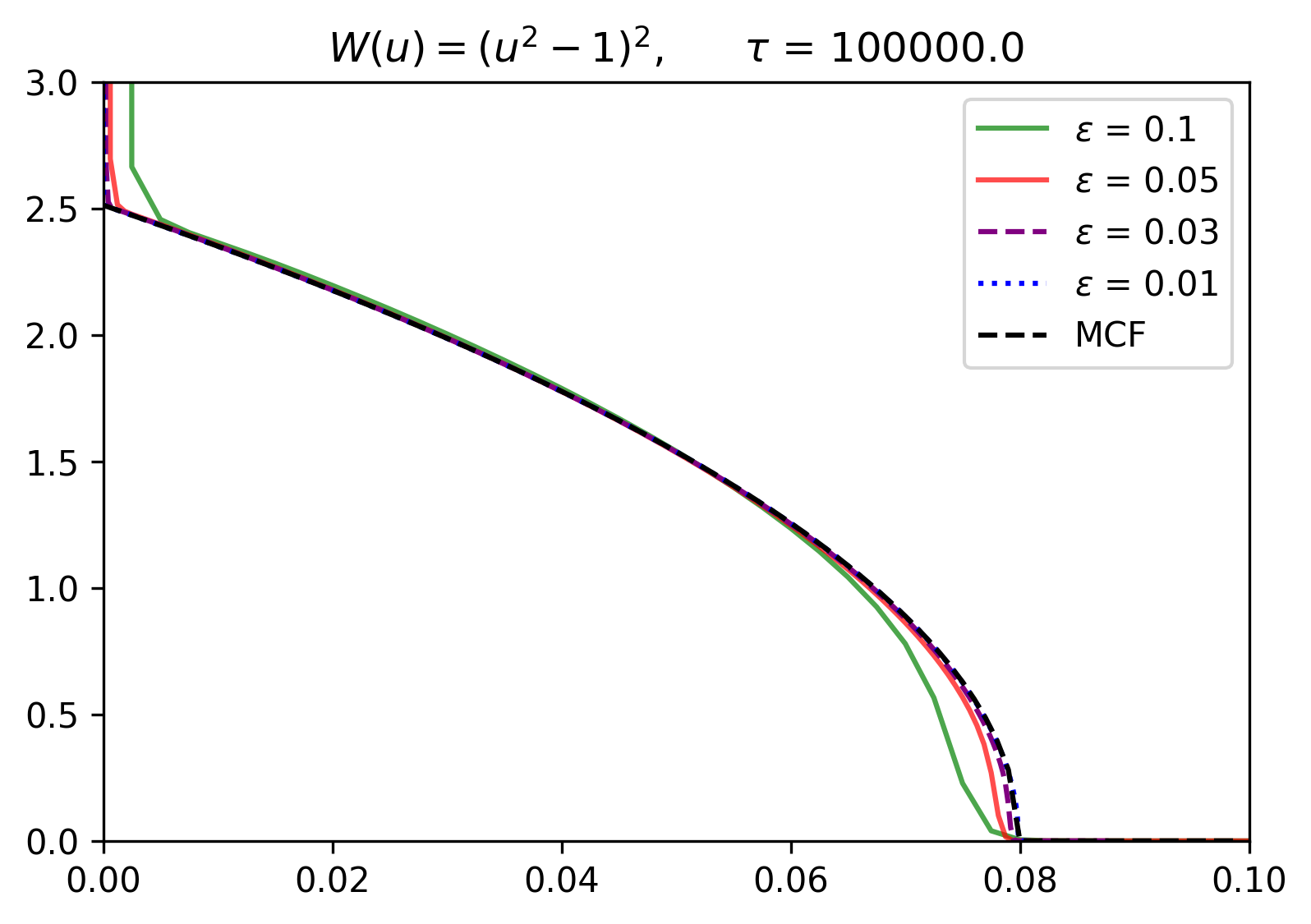}\hspace{2mm}
\includegraphics[width = .33\linewidth]{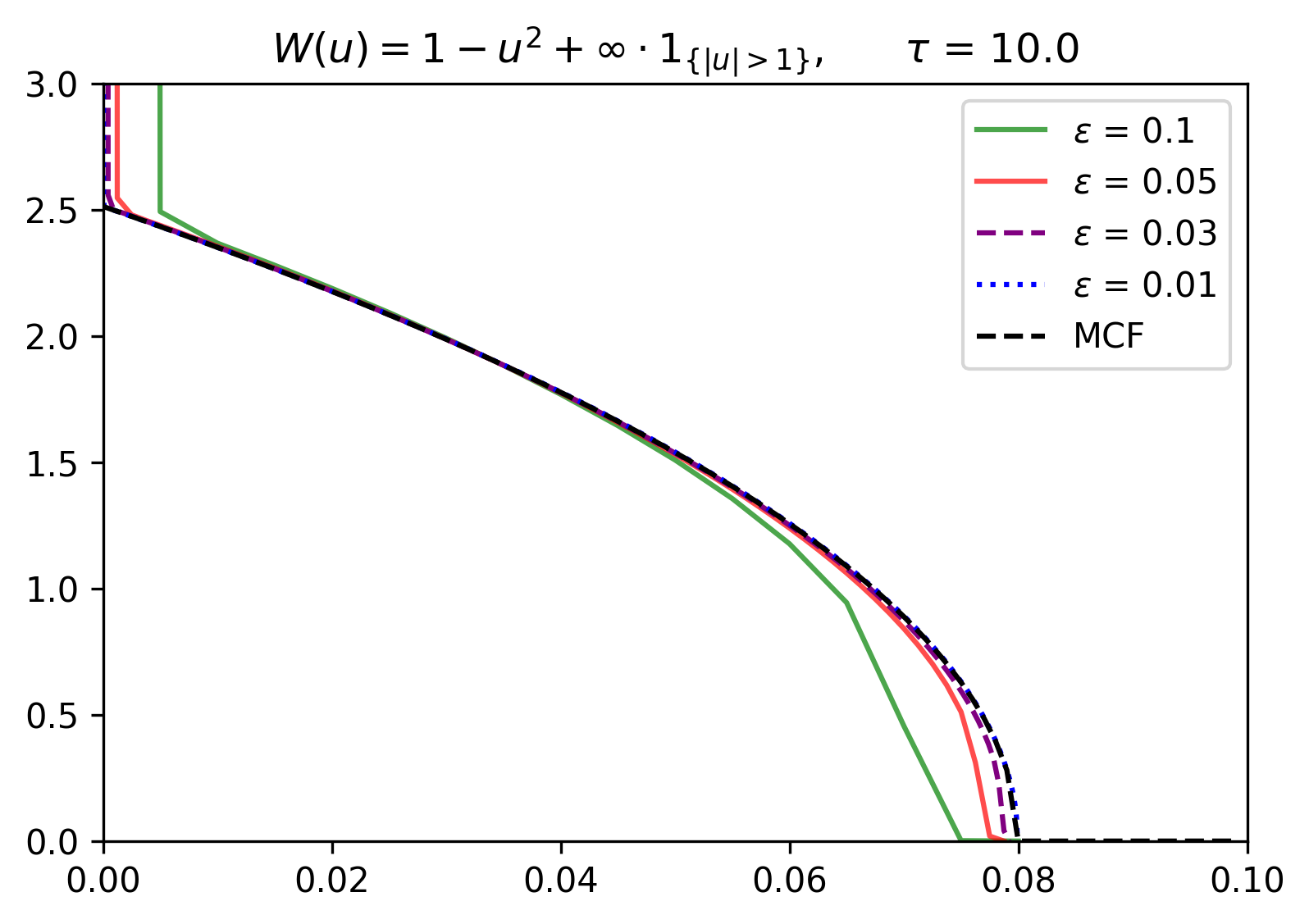}

\caption{{\bf Left:} The convex-concave splitting time-stepping scheme the standard potential $W(u) = (u^2-1)^2$. Times are computed according to an `effective step size' $0.25\cdot \eps^2$.  {\bf Right:} The convex-concave splitting time-stepping scheme the barrier potential $W(u) = 1-u^2 + \infty\cdot 1_{\{|u|>1\}}$, implemented as described in Appendix \ref{appendix singular numerics}. Times are computed according to an `effective step size' $0.5\cdot\eps^2$.
\label{figure other potentials}
}
\end{figure}

\section{Conclusion}\label{section conclusion}

In Lemma \ref{lemma energy descent}, the effective time-step size of the gradient descent scheme with convex-concave splitting for the Allen-Cahn equation may scale with different powers of $\eps$, depending on the `curvature' of the potential. Curiously, in both settings where we obtain the optimal scaling with $\eps$ analytically -- potentials with quadratic convex part and barrier potentials -- the effective time step size scales as $\eps^2$. The same scaling is observed numerically for the standard potential. 

It is tempting to conjecture that the neglected gradient term in Lemma \ref{lemma energy descent} may be responsible for this universally observed behavior. However, we note that the function $u_1 \equiv -1$ satisfies the estimate
\[
\frac\eps2 \,\|\nabla (u_1 - u_0)\|_{L^2(\Omega)}^2 + E_\eps(u_1) \leq E_\eps(u_0)
\]
for {\em any} function $u_0$. Thus, in the limit $\tau\to \infty$, Lemma \ref{lemma energy descent} would not imply any time step restrictions if $\bar c_W=0$, assuming that a constant minimizer is compatible with the boundary conditions. We therefore conjecture that a more subtle explanation is needed and that we cannot rely purely on the simple energy dissipation estimate.

In the following, we offer a simple heuristic for why the universal $\eps^2$-scaling would be expected, at least in a toy model. Consider the problem of minimizing the strictly convex functional
\[
F_{\tau, \kappa,\eps}(u) = \int_{-1}^1 \left( \frac\eps2\,|u'|^2 + \frac{W_{vex}(u) + W_{conc}'(\phi_\eps)\,u}\eps\right)(1+\kappa x)\,\dx + \frac1{2\tau}\|u -\phi_\eps\|_{L^2(-1,1)}^2
\]
where $\phi_\eps(x) = \phi(x/\eps)$ is a re-scaling of the optimal profile and $\eps >0$, $\tau \in (0,\infty]$ and $\kappa \in (-1/2, 1/2)$. For simplicity, we impose the boundary condition $u(\pm 1) = \phi_\eps(\pm 1)$ and observe that if $\kappa = 0$, then $u=\phi_\eps$ is the unique minimizer of $F_{\tau,\kappa,\eps}$ for any $\tau$. 

The additional parameter $\kappa$ serves as a proxy for the (mean) curvature of the transition interface: If $\kappa>0$, there is an incentive to move the transition where $u' \gg 1$ to the left; while $\kappa>0$ incentivizes transitions further to the right where it is `cheaper'. 

The presence of curvature $\kappa\neq 0$ changes both the shape and location of the optimal transition. Neglecting the second effect (and the boundary condition), we heuristically focus on the first:
\begin{align*}
\int_{-1}^1 \frac{\eps}2\,|\phi_\eps'|^2(x&-t)(1+\kappa x) \dx \approx \int_{-1}^1 \frac{\eps}2\,|\phi_\eps'|^2(x)(1+\kappa (x+t)) \dx
	 = \int_{-1}^1 \frac{\eps}2\,|\phi_\eps'|^2(x)(1+\kappa (x+t)) \dx\\
	 & = \int_{-1}^1 \frac{\eps}2\,|\phi_\eps'|^2(x)\,(1+\kappa x)\dx + \kappa t\int_{-1}^1 \frac{\eps}2   \,\big|\phi_\eps'|^2\dx = \int_{-1}^1 \frac{\eps}2\,|\phi_\eps'|^2(x)\,(1+\kappa x)\dx + \frac{c_W\,\kappa t}2
\end{align*}
where the approximate inequality relies on the fact that $\phi_\eps'$ essentially vanishes close to $\pm 1$. If $\phi'$ is odd/$W$ is even, the first term does not actually depend on $\kappa$. The equality holds exactly if $\phi$ transitions between the potential wells on a finite segment of the $x$-axis and $\eps$ is small enough. It is correct to order roughly $O(e^{-|t|/(\sqrt{W''(1}\eps)})$ if $W''(\pm 1)>0$ as $\phi'\to 0$ exponentially fast at $\infty$.

 If $W_{vex}(1) = W_{vex}(-1)$, the same analysis can be made for the minimization of 
\[
\int_{-1}^1 \frac{W_{vex}(\phi_\eps(x-t))}\eps \,(1+\kappa x) \dx\qquad \text{or equivalently}\quad \int_{-1}^1 \frac{W_{vex}(\phi_\eps(x-t)) - W_{vex}(1)}\eps  \,(1+\kappa x) \dx
\]
with an impact $\approx c_{conv}\kappa t$. The assumption does not pose any restriction since we can modify the splitting $W = W_{vex} + W_{conc} = (W_{vex}+\alpha x) + (W_{conc}-\alpha x)$ for our analysis without changing the algorithm: For constant right hand side, explicit and implicit time-step coincide. 

Assuming that $W_{conc}'$ and $\phi_\eps$ are odd functions, the `frozen' concave part gives us the contribution
\begin{align*}
\int_{-1}^1 \frac{W_{conc}'(\phi_\eps(x))}\eps&\,\phi_\eps(x-t)\,(1+\kappa x)\dx \approx \int_{-1}^1 \frac{W_{conc}'(\phi_\eps(x))}\eps\,\left(\phi_\eps(x) - \frac t\eps\,\phi_\eps'(x) + \frac {t^2}{\eps^2}\,\phi_\eps''(x)\right)(1+\kappa x)\dx\\
	&= \int_{-1}^1 \frac{W_{conc}'(\phi_\eps)\,\phi_\eps}\eps - \kappa t\,\frac{W_{conc}'(\phi_\eps)\,(x/\eps)\phi_\eps'}\eps + \frac{t^2}{\eps^2}\,\frac{W_{conc}'(\phi_\eps)\,\phi_\eps''}\eps \dx\\
	&= \int_{-1}^1 \frac{W_{conc}'(\phi_\eps)\,\phi_\eps}\eps\dx - \kappa\,c_{1, conc}t + c_{2,conc}\frac{t^2}{\eps^2}
\end{align*}
where $\phi_\eps^{(k)}(x) = \phi^{(k)}(x/\eps)$ (i.e.\ we first differentiate and the insert the argument) and
\begin{align*}
c_{2,conc} & \approx\int_{-\infty}^\infty W_{conc}'(\phi)\,\phi'' \dx = - \int_{-\infty}^\infty W_{conc}''(\phi)\,(\phi')^2\dx =  -\int_{-\infty}^\infty W_{conc}''(\phi)\,\sqrt{2\,W(\phi)}\,\phi' \dx\\
	&=- \sqrt 2\int_{-1}^1 W_{conc}''(z)\,\sqrt{W(z)}\dz \geq -\sqrt 2\int_{-1}^1 W''(z)\,\sqrt{W(z)}\dz= \frac1{\sqrt 2} \int_{-1}^1\frac{(W')^2}{\sqrt W}\dz>0.
\end{align*}
Taking $\tau = +\infty$ in this na\"ive approximation, the overall benefit of shifting the interface by a distance $t\ll \eps$ is roughly comparable to
\[
c\kappa t + c_{2,conc}\frac{t^2}{\eps^2} 
\]
for some constant $c$ which depends on the splitting $W_{vex}, W_{conc}$, but not $\eps$. The optimal shift $t$ would be $t = c\kappa \eps^2/(2\,c_{2,conc})$, and even smaller if $\tau<+\infty$. Naturally, the expansion is only valid if indeed $|t|\ll \eps$ -- which is consistent with the `optimal' choice of $t$ derived. We believe that this observation, despite its reliance on heuristic approximations, provides accurate geometric insight into the precise scaling of the convex-concave splitting approximation.

Let us note that $\eps^2$-fast motion can also be achieved by the semi-implicit scheme
\[
(1-\tau\,\Delta) u_{n+1} = - \frac{\tau}{\eps^2}\, {W'(u_n)}
\]
for time-step size $\tau < 2\eps^2/ [W']_{Lip}$ which treats the entire double-well potential $W$ explicitly, at least if $W'$ is Lipschitz-continuous (e.g.\ for $W= W_R$ as in Appendix \ref{appendix W_R}. In practice, the same holds for other potentials since numerical solutions $u_n$ take values in e.g.\ $[-2,2]$, even if the strict maximum principle which confines analytic solutions to $(-1,1)$ may be violated.

The draw-back of the semi-implicit scheme is that it requires the user to carefully choose $\tau$ -- but this can evidently be done easily in a principled way. The benefit over the splitting scheme is that a solver for $(1-\tau\Delta)$ can be pre-computed e.g.\ by LU factorization and that the same system is solved in every step. The same is true for potentials $W$ with quadratic convex part, but not in general.

We compare this to the guarantees for the semi-implicit scheme
\[
x_{n+1} = x_n - \tau\big(\nabla f(x_{n+1})+ \nabla g(x_n)\big)
\]
which achieves the guarantee
\[
(f+g)(x_{n+1}) \leq (f+g)(x_n) - \tau\left(1- \frac{[\nabla g]_{Lip}}2\,\tau\right) \,\|\nabla f(x_{n+1})+\nabla g(x_n)\|^2
\]
if $f$ is convex and $\nabla g$ is Lipschitz-continuous. The optimal step size is $\tau = 1/[\nabla g]_{Lip}$. In the case of the Allen-Cahn equation, we note that the gradient of the energy is $\eps$-small as demonstrated in analyses of phase-field approximations of Willmore's energy \cite{bellettini1993approssimazione, roger2006modified, dondl2017uniform, dondl2017phase} and that the gradient of the double-well contribution is $1/\eps$-Lipschitz. This, again, yields the $\eps^2$-scaling. 

\section*{Acknowledgements}

SW acknowledges the support of the NSF through grant DMS 2424801. SW is grateful to Selim Esedoglu for a helpful conversation concerning the thresholding scheme with general kernels.

\appendix 

\section{Smooth double-well potentials with quadratic convex part}\label{appendix W_R}

We construct a smooth doublewell potential $W$ such that $W(u) = u^2 + W_{conc}(u)$ where $W_{conc}$ is a differentiable concave function. 
The functions we consider are smoother versions of $W(u) = u^2  + 1 - 2|u|$ for which most analytic results apply without further modification.

\begin{example}
For $R, \beta, \gamma >0$, the potential
\[
W(u) = W_{\beta, \gamma, R}(u) := u^2 + \beta - \gamma \sqrt{R u^2 +1}
\]
satisfies $\lim_{u\to\pm \infty} W(u)/u^2 = 1$. As $\gamma>0$, the function
\[
W_{conc}(u) := \beta -\gamma\sqrt{Ru^2+1}
\]
is smooth and concave.
We see that
\[
0 = W'(u) = 2u - \frac{2\gamma Ru}{2\sqrt{Ru^2+1}} = \left(2 - \frac{\gamma R}{\sqrt{Ru^2+1}}\right)u = 0
\]
if $u=0$ and at $u = \pm 1$ if $\gamma = \frac{2\sqrt{R+1}}R$. The critical point at zero is a local maximum and 
\[
W(\pm 1) = 1 + \beta - \gamma \sqrt{R+1} = 0
\]
if $\beta = \gamma\sqrt{R+1}-1 = \frac{2(R+1)-1}R = 1 + 2/R$. At the origin, we have
\[
W(0) = \beta -\gamma = 1+ \frac2R - \frac{2\sqrt{R+1}}R = \frac{R+2 - 2\sqrt{R+1}}R = \frac{(\sqrt{R+1}-1)\sqrt{R+1}}R>0
\]
for all $R>1$. This ansatz therefore yields a one parameter family of doublewell potentials with quadratic convex part and potential wells at $\pm 1$, which we denote by $W_R$. We note that $W$ is even and that
\[
\gamma = \frac{2\sqrt{R+1}}R \quad\Ra\quad W''(1) = 2 - \gamma \,\frac{R}{\sqrt{R+1}}+ \gamma \,\frac{R^2}{(R+1)^{3/2}} = 2 - 2 + \frac{2R}{R+1}= \frac{2R}{R+1}>0.
\]
As $R\to\infty$, we observe that
\begin{align*}
\lim_{R\to \infty} W_R(u) &= \lim_{R\to\infty}\left(u^2 + 1 + \frac2R - \frac{2\sqrt{R+1}}R \sqrt{Ru^2+1}\right)\\
    &= u^2+1 - 2\,\lim_{R\to \infty}\left(\sqrt{\frac{R(R+1)}{R^2}u^2+ \frac{R+1}{R^2}}\right)\\
    &= u^2 + 1 - 2|u|,
\end{align*}
i.e.\ we can see $W_R$ as a smooth approximation to $\overline W$ -- see Figure \ref{fig:ws}.

\begin{figure}
    \centering
    \includegraphics[width = .33\textwidth]{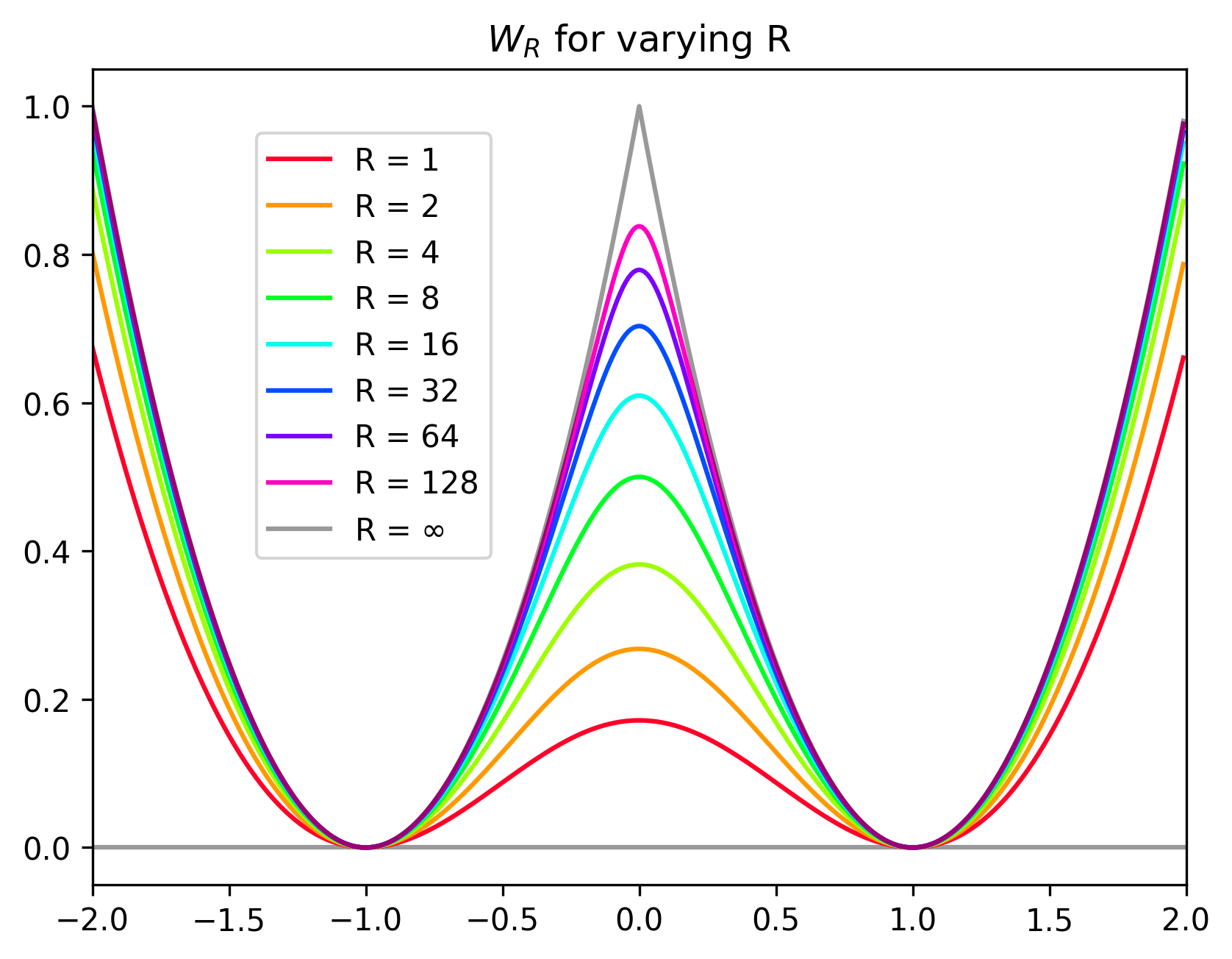}
    \includegraphics[width = .33\textwidth]{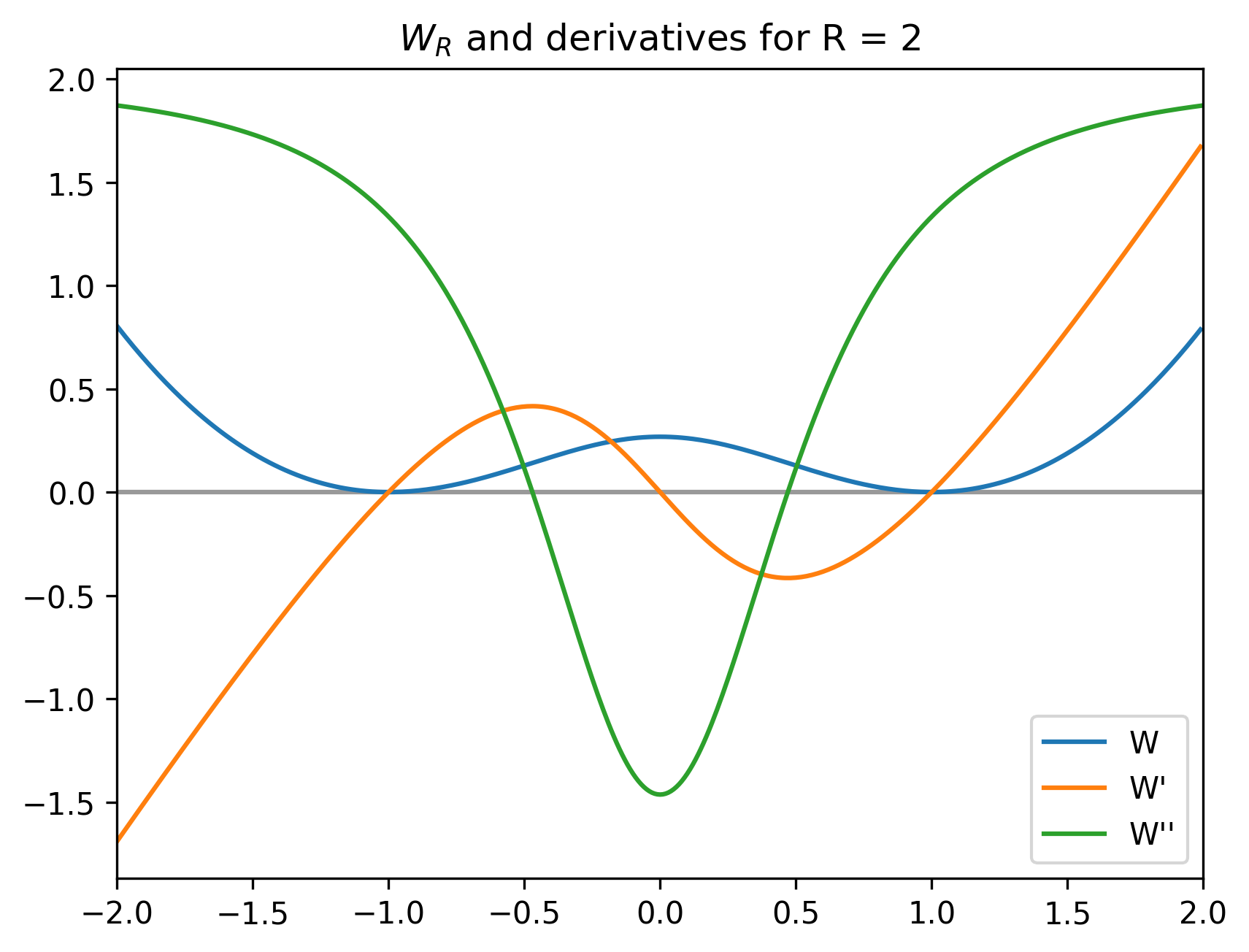}\
    
    \caption{{\bf Left:} The doublewell potential $W_R$ for varying values of $R$. {\bf Right:} $W_R$ and its first two derivatives for $R=1$.}
    \label{fig:ws}
\end{figure}
\end{example}

The fact that $W_R(u)\nearrow \overline W(u) =  (|u|-1)^2$ is no coincidence.

\begin{lemma}\label{lemma w bound}
Assume that $W(u) = u^2+ W_{conc}(u)$ where $W_{conc}$ is a concave $C^1$-function, $W\geq 0$ and $W(\pm 1) = 0$. Then $W(z) \leq (1-|z|)^2$ for all $z\in[-1,1]$.
\end{lemma}

\begin{proof}
We observe that $W$ is minimal at $\pm 1$ and differentiable, so
\[
0 = W(u) = u^2 + W_{conc}(u) = 1 + W_{conc}(u) \qquad \forall\ u\in\{-1,1\}
\]
and
\[
0 = W'(u) = 2u + W_{conc}'(u) \qquad \Ra\quad W_{conc}'(u) = -2\,sign(u) \qquad\forall\ u\in \{-1,1\}.
\]
Since $W_{conc'}$ is non-increasing on $[0, \infty)$, we see that $W_{conc}' \geq -2$ on $[0,1]$ and thus 
\[
W_{conc}(z) = - \int_z^1 W_{conc}'(t)\dt + W_{conc}(1) \leq -1-\int_z^1(-2)\dt = 2(1-z)-1 = 1 - 2z \qquad\ \forall\ z\in [0,1],
\]
i.e.\ 
\[
W(u) = u^2 + W_{conc}(u) \leq u^2 + 1-2u\qquad \forall\ u\in [0,1].
\]
The case $u\in[-1,0]$ follows analogously.
\end{proof}

\section{The thresholding scheme with radially symmetric kernels}\label{appendix thresholding}

We briefly summarize the convergence result of \cite{ishii1995generalization, ishii1999threshold} for thresholding schemes with general kernel convolutions. We specialize the result of \cite[Theorem 3.3]{ishii1999threshold} to the setting of radially symmetric kernels, which was initially treated in the unfortunately hard to obtain article \cite{ishii1995generalization}.

Assume that $K$ is a measurable function on $\R^d$ and $K(Ox) = K(x)$ for all $x\in\R^d$ and $O\in O(d)$. Assume additionally that
\[
\int_{\R^d}K(x) \dx =1, \qquad \int_{\R^d}|x|^2\,K(x)\dx < +\infty, \qquad \int_{\R^{d-1}} K(\xi, 0) \big(1+\|\xi\|^2\big)\d\xi <\infty.
\]
This comprises assumptions (3.1) -- (3.4) of \cite{ishii1999threshold}, noting that
the continuity condition (3.3) for hyperplane integrals of $K$ in \cite{ishii1999threshold} is automatically satisfied for radially symmetric kernels since the integrals do not depend on the choice of hyperplane. Additionally, we require that
\[
\lim_{\rho\to 0^+} \sup_{r\in (0,\rho)}\left|\int_{B_{R(\rho)}(0_{d-1})} K\left(\xi, r \big(a + \xi^TS\xi\big)\right)\big(a + \xi^TS\xi\big)\d\xi - \int_{\R^{d-1}} K\left(\xi, 0\right)\,\big(a+ \xi^TS\xi\big)\d\xi\right| =0
\]
where $a\in \R$ and $S\in \R^{(d-1)\times (d-1)}_{sym}$ are arbitrary and $R(\rho)$ can be any function such that $R(\rho)\to +\infty$ and $\sqrt{\rho}\,R(\rho)\to 0$ as $\rho\to0^+$. This `localization' condition (3.7) of \cite{ishii1999threshold} implies consistency with MCF for quadratic graph approximations of smooth surfaces. Here $0_{d-1}$ denotes the origin in $\R^{d-1}$.

For the kernel $K(x) = \frac{\exp(\|x\|)}{4\pi\,\|x\|}$ on $\R^3$, we observe that
\begin{align*}
\int_{\R^3} K(x) \dx 
    &= 4\pi\int_0^\infty \frac{\exp(-r)}{4\pi\,r}r^2\dr
    =\int_0^\infty r\,\exp(-r)\dr 
    = \int_0^\infty \exp(-r)\dr 
    = 1\\
\int_{\R^3} K(x)\,\|x\|^2\dx 
    &= \int_0^\infty r^3\exp(-r)\dr 
    = \int_0^\infty 3r^2 \exp(-r)\dr 
    = 6\int_0^\infty r\,\exp(-r)\dr 
    = 6\\
\int_{\R^2}K(\xi,0)\,\big(1+\|\xi\|^2\big)\d\xi
    &= 2\pi \int_0^\infty \frac{\exp(-r)}{4\pi\,r}\,(1+r^2) \,r\dr
    = \frac12 \int_0^\infty (1+r^2) \exp(-r)\dr = \frac32,
\end{align*}
i.e.\ the conditions (3.1) -- (3.4) of \cite{ishii1999threshold} hold for $K$.
The consistency condition (3.7) of \cite{ishii1999threshold} can be established in two steps: For any fixed $a, S$ we have
\[
\sup_{r\in(0,\rho)}\left|\int_{\R^{2}\setminus B_{R(\rho)}(0)} K\left(\xi, r \big(a + \xi^TS\xi\big)\right)\big(a + \xi^TS\xi\big)\d\xi\right| \leq 2\pi \int_{B_{R(\rho)}^c}\frac{\exp(-r)}{4\pi r} \big(|a| + \|S\|_{op}r^2\big)\dr \to 0^+
\]
as $\rho\to 0$ since the integrand is in $L^1(\R^2)$ and $R(\rho)\to \infty$ as $\rho\to 0$. Additionally, we have
\begin{align*}
\limsup_{\rho\to 0^+}\sup_{r\in (0,\rho)} \left|\int_{\R^2} \big\{K(\xi,0) - K\big(\xi, r(a+\xi^TS\xi)\big)\big\} (a+\xi^TS\xi) \d\xi \right| \leq \lim_{\rho \to 0^+} \int_{\R^2} f_\rho(\xi) \d\xi 
\end{align*}
where 
\begin{align*}
f_\rho(\xi) &= \sup_{0<r<\rho} \big|K(\xi,0) - K\big(\xi, r(a + \xi^TA\xi)\big)\big|(|a| + \|S\|_{op}\|\xi\|^2\big)\\
    &\leq \left[K(\xi, 0) - K\big(\xi, \rho\big\{a + \|S\|_{op}\|\xi\|^2\big\}\big) \right]\,\big\{|a| + \|S\|_{op}\|\xi\|^2\big\}.
\end{align*}
The inequality follows since the function $s\mapsto \exp(-s)/s$ is monotone decreasing, so $s\mapsto K(\xi, s)$ is monotone decreasing. We note that $\lim_{\rho\to 0^+} f_\rho(x) = 0$ for all $x\in\R^3$ and $|f_\rho(x)| \leq C(1+\|x\|^2)\exp(-\|x\|)$, so $\int_{\R^2}f_\rho \dx \to 0^+$ as $\rho\to 0^+$ by the dominated convergence theorem. The two estimates together establish that the compatibility condition (3.7) of \cite{ishii1999threshold} holds for the kernel $K$.

We compute the velocity function as in \cite{ishii1999threshold} with $\Sigma\in\R^{3\times 3}_{sym}$ as a proxy for the second fundamental form and $n$ as a proxy for the normal of the evolving surface:
\[
v(\Sigma, n ) = \frac{-\frac12 \int_{\langle n\rangle ^\bot} \xi^T\Sigma\xi\,K(\xi)\,\d\xi }{\int_{\R^2}K(\xi,0)\,\d\xi} = - \int_{\langle n\rangle^\bot}\xi^T\Sigma\xi\,\frac{\exp(-\|\xi\|)}{4\pi\,\|\xi\|}\,\d\xi.
\]
We choose an orthonormal coordinate system $\{n, e_1, e_2\}$ and note that 
\begin{align*}
v(\Sigma, n) &= \frac{-1}{4\pi} \int_0^\infty\int_0^{2\pi} r\big(\cos\phi\, e_1+\sin\phi \,e_2)^T\,\Sigma \, r(\cos \phi\,e_1 + \sin\phi\,e_2) \frac{\exp(-r)}r\,r\dr\\
    &=\frac{-1}{4\pi} \left(\int_0^{2\pi}\cos^2\phi\,e_1^T\Sigma\,e_1 + \sin^2\phi\,e_2^T\Sigma\,e_2\,\d\phi\right) \left(\int_0^\infty\exp(-r)r^2\dr\right)\\
    &= - \frac1{4\pi} \cdot\pi \big(e_1^T\Sigma e_1 + e_2^T \Sigma e_2\big)\cdot 2
    = - \frac12\,\mathrm{tr}_{\langle n\rangle^\bot}\Sigma.
\end{align*}
If $\Sigma$ is indeed the second fundamental form of a surface and $n$ is its normal, this is precisely the negative of the (geometric) mean curvature, i.e.\ the average of its principal curvatures. For surfaces, the geometric mean curvature differs from the `analytic' mean curvature (the sum of the principal curvatures) by a factor two. The analytic mean curvature appears more naturally e.g.\ in the Allen-Cahn equation. 
By \cite[Theorem 3.3]{ishii1999threshold}, we find that the thresholding scheme indeed approximates a viscosity solution to mean curvature flow at the time-scale $\frac{\eps^2}2$, at least on the whole space $\R^3$.

\section{On the uniform convexity of power law functions}\label{appendix power law}

In this Section, we prove Lemma \ref{lemma power law}. The special cases $p=2$ and $p=4$ follow directly from algebraic manipulation. If $f_p(u) := |u|^p$, then
\begin{align*}
f_2(u+v) &= (u+v)^2
    = u^2 + 2uv+v^2
    = f_p(u) + f_p'(u)\,v + v^2
\end{align*}
and
\begin{align*}
f_4(u+v) &= (u+v)^4\\
	 &= u^4 + 4u^3v + 6u^2v^2 + 4uv^3 + v^4\\
	 &= f_4(u) + f_4'(u)\,v + v^2 \big(6u^2+4uv + v^2\big)\\
	 &=  f_4(u) + f_4'(u)\,v + v^2\left(\sqrt{6}\,u+\frac2{\sqrt 6}\,v\right)^2 + v^2\left(1- \frac{4}6\right)\,v^2\\
	 &\geq f_4(u) + f_4'(u)\,v + \frac13\,v^4.
\end{align*}
The inequality holds with equality if $v = -3u$. We will see in the proof of the general case that $v = -cu$ for $c\approx 2$ is indeed the most restrictive case also for general $p$.

\begin{proof}[Proof of Lemma \ref{lemma power law}]
If $u=0$, the inequality holds for any $\beta\leq 1$, so we may assume that $u\neq 0$. Dividing by $|u|^p$ and denoting $\xi = v/|u|$, we see that the inequality is equivalent to
\[
\big|1 + \sign(u)\xi\big|^p = \big|\sign(u) + \xi\big|^p \geq 1 + p\,\sign(u)\,\xi + \beta\,|\xi|^p \qquad \forall\ u\neq 0,\: \xi\in \R.
\]
Replacing $\zeta = \sign(u)\xi$, it reduces further to the claim that
\begin{equation}\label{eq improved bernoulli}
\big|1+ \zeta\big|^p \geq 1 + p\zeta + \beta\,|\zeta|^p\qquad\forall\ \zeta \in\R.
\end{equation}
This can be seen as an improved version of Bernoulli's inequality which attains the right kind of growth both at $\xi=0$ and as $|\xi|\to \infty$. 

We differentiate between three cases: (i) $\zeta\geq 0$, (ii) $\zeta \in[-1,0)$ and $\zeta \in(-\infty,-1]$. In the first case, the inequality reduces to
\[
\big(1+ \zeta\big)^p - 1 - p\zeta - \beta\,\zeta^p \geq 0.
\]
We note that the inequality holds with equality for $\zeta = 0$ independently of $\beta$ and 
\[
\frac{d}{d\zeta}\left\{\big(1+ \zeta\big)^p - 1 - p\zeta - \beta\,\zeta^p\right\} = p(1+\zeta)^{p-1} - p - \beta p\zeta^{p-1} = 0
\]
for $\zeta=0$. The second derivative satisfies
\[
\frac{d^2}{d\zeta^2}\left\{\big(1+ \zeta\big)^p - 1 - p\zeta - \beta\,\zeta^p\right\} = p(p-1) \left\{\big(1+\zeta\big)^{p-2}- \beta \zeta^{p-2} \right\}\geq 0
\]
for $\beta \leq 1$ since $p\geq 2$ implies that $(1+\zeta)^{p-2}\geq \zeta^{p-2}$. Thus the first derivative is increasing from zero, meaning that the function itself is increasing (and, in fact, convex with minimizer at $0$).

In the second case, we write $z= -\zeta$ and reduce the inequality to
\[
(1-z)^p \geq 1-pz + \beta z^p\qquad \forall\ z \in [0,1].
\]
We again note that the inequality holds as equality for $z=0$ and compute
\[
\frac{d}{dz} \big\{(1-z)^p - 1 + pz - \beta\,z^p\big\} = -p(1-z)^{p-1} + p - p\beta\,z^{p-1} = p \big(1- z^{p-1} - \beta (1-z)^{p-1}\big) \geq 0
\]
for $\beta \leq 1$ as $p-1\geq 1$. 
Thus \eqref{eq improved bernoulli} holds for $\zeta\in[-1,\infty)$. It remains to consider the third case $\zeta \in (-\infty, -1]$. Here, we will encounter sharper restrictions on $\beta$.

After replacing $z= -\zeta$ again, the inequality reads as
\[
(z-1)^p \geq 1 -pz + \beta\,z^p \qquad \LRa\quad \frac{(z-1)^p + pz -1}{z^p}\geq \beta \qquad \forall\ z\in [1,\infty)
\]
We can find admissible $\beta$ as
\begin{align*}
\beta(p) := \min_{z\geq 1} \frac{(z-1)^p + pz -1}{z^p} &= \min\left\{\min_{z\in[1,R]}\frac{(z-1)^p + pz -1}{z^p} , \min_{z\in[R,\infty)} \frac{(z-1)^p + pz -1}{z^p}\right\}\\
	&\geq \min\left\{\min_{z\in[1,R]}\frac{pz -1}{z^p} , \min_{z\in[R,\infty)} \frac{(z-1)^p}{z^p}\right\}.
\end{align*}
We note that
\[
\frac{pz -1}{z^p} \geq \frac{(p-1)z}{z^p} = (p-1)\,R^{1-p} \quad\forall\ z\in[1,R], \qquad \frac{(z-1)^p}{z^p} = \left(1- \frac1z\right)^p \geq \left(1-\frac1R\right)^p \quad\forall\ z\in [R, \infty).
\]
As the bound holds for any choice of $R>1$, we can evaluate at $R=2$ to obtain
\[
\beta(p) \geq \min\{(p-1)\,2^{1-p}, \:2^{-p}\} = 2^{-p}. \qedhere
\]
\end{proof}

This bound scales almost optimally as $p\to\infty$ as
\[
\beta(p) = \min_{z\geq 1} \frac{(z-1)^p + pz -1}{z^p}  \leq\frac{(2-1)^p + p\cdot 2 -1}{2^p} = \frac{2p}{2^p} = p\,2^{1-p}.
\]

\section{Calculations for Example \ref{example obstacle}}\label{appendix obstacle}
We make the ansatz
\[
u(s) = \begin{cases} 1 & s < r_i\\ a + b\,G_d(s) - \frac1{2d\,\eps^2}\,s^2 & r_i < s <r\\ c+e\,G_d(s) + \frac1{2d\,\eps^2}s^2 & r < s < r_o\\ -1 &s > r_o\end{cases}
\]
for coefficients $a, b, c, e$ which make $u$ $C^1$-regular at the radii $r_i, r, r_o$. Note that also $r_i, r_o$ are problem parameters, leading to the system
\[
\begin{cases}
a + b \,G_d(r_i) - \frac1{2d\,\eps^2} \,r_i^2 & = 1\\
b\,G_d'(r_i) - \frac{1}{d\,\eps^2}\,r_i &= 0\\
a + b \,G_d(r) - \frac1{2d\,\eps^2} \,r^2 & = c + e \,G_d(r) + \frac1{2d\,\eps^2} \,r^2\\
b\,G_d'(r) - \frac{1}{d\,\eps^2}\,r &= e\,G_d'(r) + \frac1{d\,\eps^2}r\\
c + e \,G_d(r_o) + \frac1{2d\,\eps^2} \,r_o^2 & = -1\\
e\,G_d'(r_o) + \frac{1}{d\,\eps^2}\,r_o &= 0\\
\end{cases}
\]
of six equations in six unknowns $a,b, c,e, r_i, r_o$. Note that $G_d'(s) = \frac1{d\omega_d} s^{1-d}$ for all $d\geq 2$, so the second and sixth equation simplify to
\[
0 = b\,\frac1{d\omega_d} \,r_i^{1-d} - \frac{1}{d\,\eps^2}\,r_i \qquad \Ra \quad b = \frac{d\,\omega_d}{d\,\eps^2}\,r_i^d = \frac{\omega_d}{\eps^2}\,r_i^d \qquad \text{and analogously} \quad e = - \frac{\omega_d}{\eps^2}\,r_o^d.
\]
Inserting this into the first and fifth equation, we obtain
\[
a = 1 - b\,G_d(r_i) + \frac1{2d\,\eps^2}\,r_i^2 = 1 - \frac{\omega_d}{\eps^2} \,r_i^d\,\frac{-1}{d(d-2)\omega_d}r_i^{2-d} + \frac1{2d\,\eps^2}\,r_i^2 =
1 +\frac{r_i^2}{d\,\eps^2}\left(\frac1{d-2} + \frac12\right)
\]
in dimension $d\geq 3$ and analogously
\[
a = 1 + \frac{r_i^2}{2(d-2)\,\eps^2}, \qquad c = -\left(1 + \frac{r_o^2}{2(d-2)\,\eps^2}\right).
\]
Having determined $a,b,c,e$ for given $r_i, r_o$, we need to exploit the smoothness of $u$ at the contact sphere $\{s=r\}$ to find $r_i, r_o$. Again, we start with the simpler condition for the derivative
\[
\left(\frac{1}{d\,\eps^2}\,r_i^dr^{-d} - \frac1{d\,\eps^2}\right)r =  b\,G_d'(r) - \frac{1}{d\,\eps^2}\,r = e\,G_d'(r) + \frac1{d\,\eps^2}r = \left(\frac{-1}{d\,\eps^2}\left(\frac{r_o}r\right)^d + \frac1{d\,\eps^2}\right) r,
\]
which yields that $(r_i/r)^d + (r_o/r)^d = 2$. The continuity condition becomes
\begin{align*}
1 + \frac{r_i^2}{2(d-2)\eps^2} - \frac{1}{d(d-2)\eps^2} r_i^dr^{2-d} - \frac1{2d\,\eps^2} \,r^2  = -\left(1 + \frac{r_o^2}{2(d-2)\eps^2} \right) + \frac{1}{d(d-2)\eps^2} r_o^dr^{2-d} + \frac1{2d\,\eps^2} \,r^2 
\end{align*}
or equivalently 
\[
2 + \frac 1{2(d-2)\eps^2}\,(r_i^2 + r_o^2)  -\frac{1}{d(d-2)\eps^2}\left\{ \left(\frac{r_i}r\right)^d + \left(\frac{r_o}r\right)^d\right\}r^2 = \frac2{2d\,\eps^2}r^2.
\]
Simplifying by the smoothness identity $(r_i/r)^d + (r_o/r)^d = 2$ and multiplying by $2d(d-2)\eps^2)$, this simplifies to 
\[
4d(d-2)\eps^2 + d(r_i^2 + r_o^2) - 4 r^2 = 2(d-2)\,r^2
\]
and overall
\[
r_i^2 + r_o^2 = 2r^2 - 4(d-2)\eps^2, \qquad r_i^d + r_o^d = 2r^2.
\]
We can thus find $r_i$ as the root of the auxiliary function
\[
\xi_\eps(s) = \sqrt{2\,r^2 - 4(d-2)\eps^2 - s^2} - \sqrt[d]{2r^d - s^d}.
\]
We observe that
\[
\xi_\eps(0) = \sqrt{2r^2 - 4(d-2)\eps^2} - \sqrt[d]{2}r \approx (\sqrt{d}-\sqrt[d]{2})r>0, \qquad \xi_\eps(r) = \sqrt{r^2 - 4(d-2)\eps^2} - r<0
\]
if $\eps$ is a small positive number. In particular, $\xi_\eps$ is well-defined on $[0,r]$ and guaranteed to have a root which can be found by bisection. Note that
\begin{align*}
\xi_\eps'(s) &= s^{d-1}\big(2r^d-s^d\big)^\frac{1-d}d - s\big(2r^2 - s^2- 4(d-2)\eps^2 \big)^{-\frac12}\\
	&<   s^{d-1}\big(2r^d-s^d\big)^\frac{1-d}d - s\big(2r^2 - s^2 \big)^{-\frac12}\\
	&= \left(\frac{2r^d - s^d}{s^d}\right)^\frac{1-d}d - \left(\frac{2r^2 - s^2}{s^2}\right)^{-\frac12}
	\quad\leq 0
\end{align*}
i.e.\ $\phi_\eps$ is strictly monotone decreasing on $(0,r)$. In particular, the radii $r_i, r_o$ are uniquely determined. To see that $\xi_\eps'<0$, note that
\[
\left(2\zeta^d-1\right)^\frac{d-1}d \geq (2\zeta^2-1)^{1/2}\qquad\forall\ \zeta \in (1,\infty)
\]
where $\zeta$ serves as the proxy for $r/s$ and we took inverses. This inequality holds true since $\zeta^d\geq \zeta^2$ for $d\geq 2$ and $\zeta\geq 1$, and also $z^{(d-1)/d}\geq z^\frac12$ for $z\geq 1$ and $d\geq 2$.

\begin{figure}
\includegraphics[width =.24\textwidth]{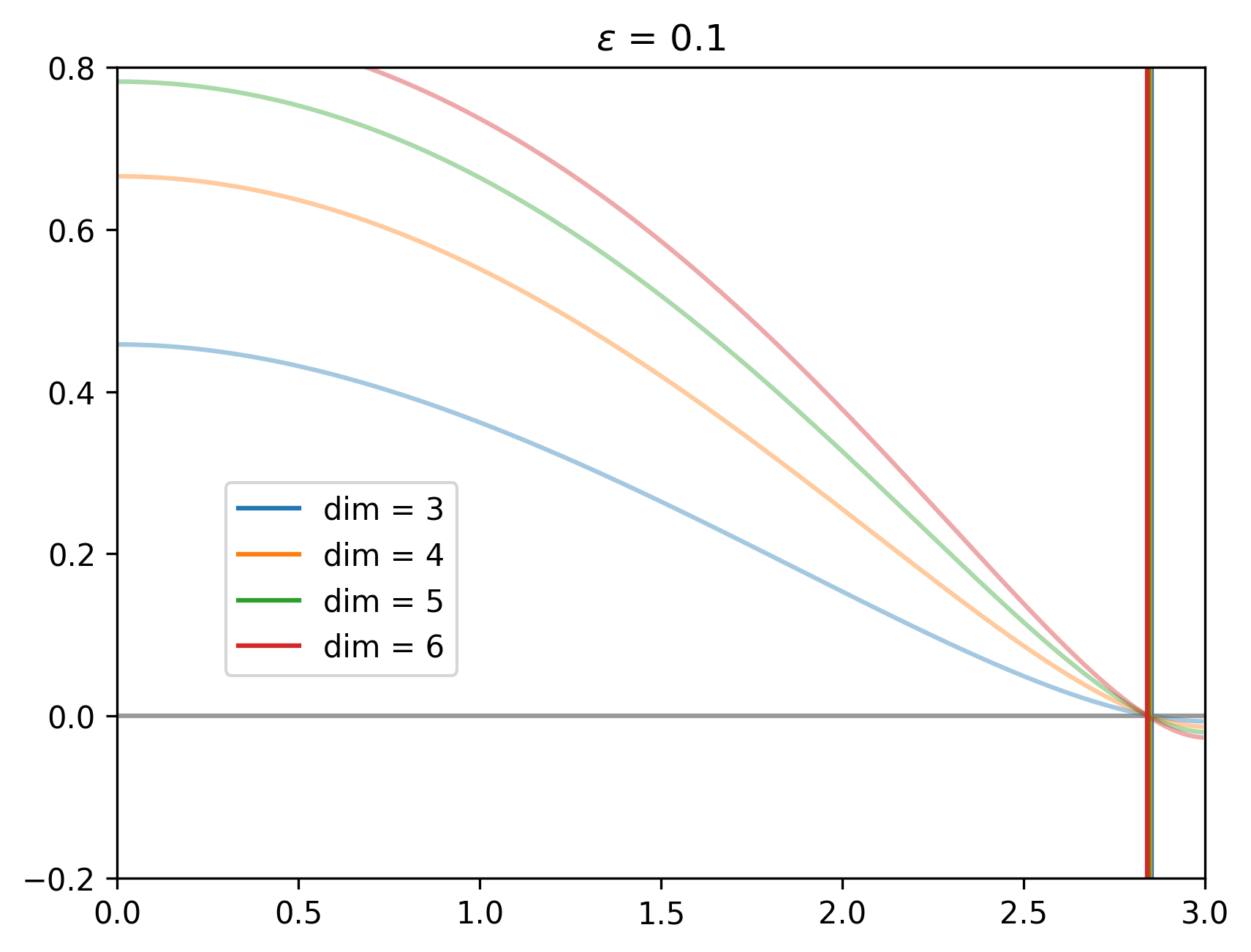}\hfill
\includegraphics[width =.24\textwidth]{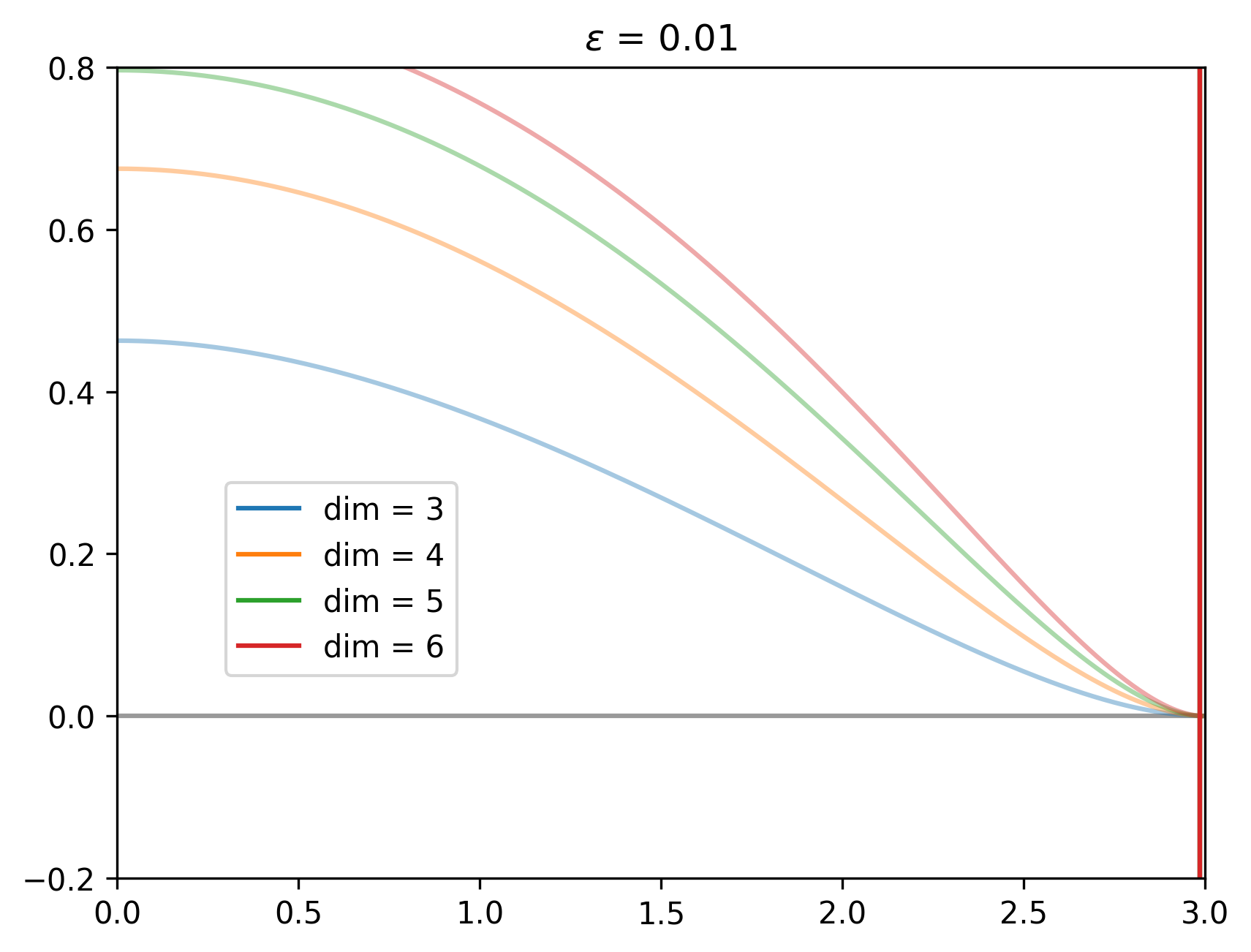}\hfill
\includegraphics[width =.24\textwidth]{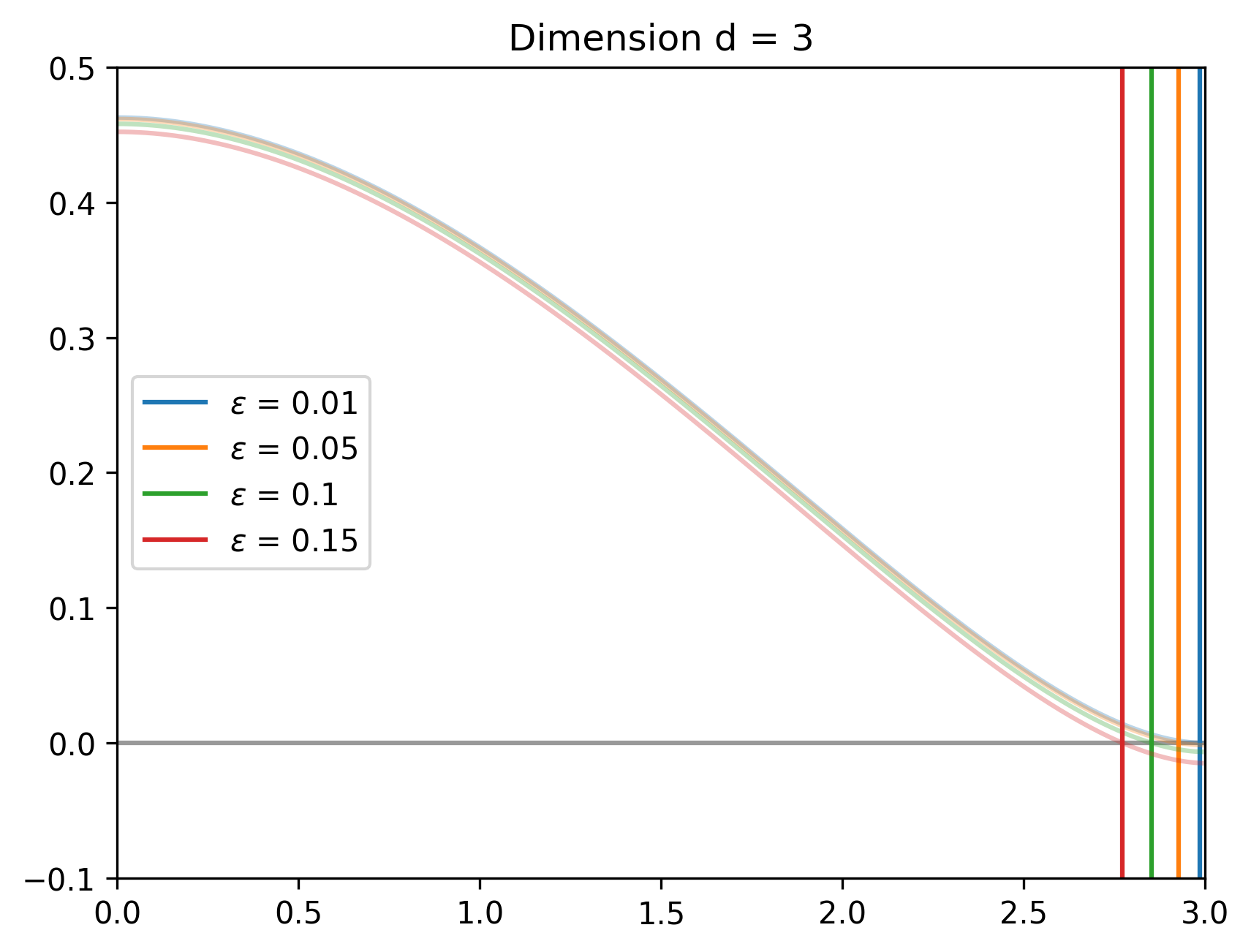}\hfill
\includegraphics[width =.24\textwidth]{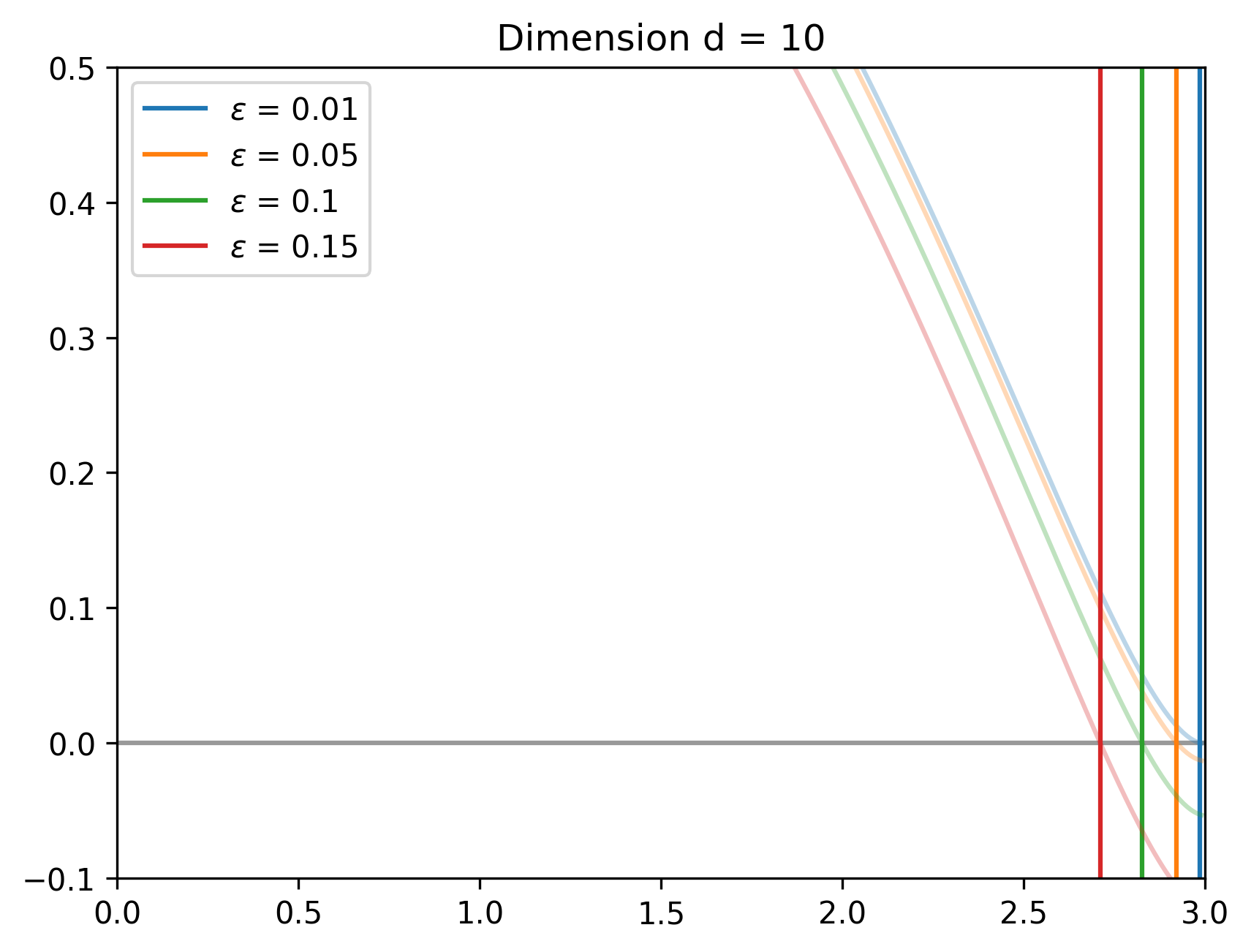}\hfill
\caption{
The auxiliary function $\xi_\eps$ for fixed $\eps$ and varying dimension (left two) and fixed dimension, and varying $\eps$ (right two). The vertical lines represent $r_i$, the intersection with the $x$-axis.
\label{figure phi}
}
\end{figure}

Knowing $r_i$, also $r_o$ can be found e.g.\ as $\sqrt[d]{2r^d -r_i^d}$, and the coefficients $a, b, c, e$ can be computed:
\[
u(s) = \begin{cases} 1 &s<r_i\\ 1+ \frac{r_i^2}{2(d-2)\eps^2} - \frac{r_i^d}{d(d-2)\eps^2}\,s^{2-d} - \frac1{2d\,\eps^2}s^2 & r_i < s<r\\
	-1- \frac{r_o^2}{2(d-2)\eps^2} + \frac{r_o^d}{d(d-2)\eps^2}\,s^{2-d} + \frac1{2d\,\eps^2}s^2 & r<s<r_o\\ -1 & s>r_o.\end{cases}
\]
In particular, $u$ is monotone decreasing in $s$, concave on $(r_i, r)$ and convex on $(r, r_o)$ -- see also Figure \ref{figure obstacle}.
Notably, there is no difference in the solutions to the time-stepping scheme between the Dirichlet boundary condition $-1$ and the homogeneous Neumann boundary conditions if $R> r_o$. Moreover, the containing domain $\Omega$ makes no difference and does not have to be radial, as long as $\overline{B_{r_o}}\subseteq\Omega$. 

Our primary interest is to find $r_{new} = u^{-1}(0)$ to see how quickly the circle of radius $r$ shrinks. Clearly, $r_{new} \in (r_i, r)$, so if $r_i$ is close to $r$, then so is $r_{new}$. It seems plausible that this should be the case, given that $\lim_{\eps\to 0^+} \phi_\eps(r) =0$, suggesting heuristically that $r_{i,\eps}\to r$ as $\eps\to 0$. However, it is not immediately clear which exponent $m$ governs the speed of convergence $r_{i,\eps} - r = O(\eps^m)$ (if any).

Note that
\begin{equation}\label{eq obstacle derivative at zero}
u'(s) = \frac{1}{d\,\eps^2}\left(\left(\frac{r_i}{s}\right)^{d} - 1\right)s\qquad \forall\ s\in (r_i, r)
\end{equation}
and the expression on the right is non-positive for all $s\geq r_i$. Hence
\begin{align*}
2 &\geq u(r_i) - u(r)
	=-\int_{r_i}^{r}u'(s)\ds
	= r_i^2 \int_{r_i}^r \frac{1}{d\,\eps^2}\left(1-\left(\frac{r_i}{s}\right)^{d} \right)\frac s{r_i}\,\frac1{r_i}\ds
	= \frac{r_i^2}{d\eps^2}\int_1^{r/r_i}\left(1-\sigma^{-d}\right)\sigma\d\sigma.
\end{align*}
Observe that the integrand $\psi_d(\sigma)= (1-\sigma^{-d})\sigma$ satisfies
\[
\psi_d(1) =0, \qquad \psi_d'(\sigma) = d\,\sigma^{-d-1}\sigma  + (1-\sigma^{-d}), \qquad \psi_d'(1) = d,
\]
so 
\[
2 \geq \frac{r_{i,\eps}^2}{d\eps^2}\, \int_1^{r/r_{i,\eps}}d(\sigma-1) + O\big( (\sigma-1)^2\big)\d\sigma \geq\frac{r_{i,\eps}^2}{d\eps^2}\,\frac d2 \left(\frac r{r_{i,\eps}}-1 + err\right)^2 = \frac{(r- r_{i,\eps})^2}{2\eps^2} + err
\]
where $err$ is an error term which we may neglect for small $\eps$ at the cost of slightly more restrictive constants. We thus achieve the estimate $r-r_{i,\eps} \leq 2\eps$, up to a small error. The estimate is fairly close. The optimal profile $\phi$ for $W$ satisfies
\[
\frac{(\phi')^2}2 = 1-|\phi| \qquad \Ra\quad \phi' = \sqrt{2 (1-\phi)} \qquad\Ra \quad -\frac{d}{dx} \sqrt{2(1-\phi)} =  \frac{\phi'}{\sqrt{2(1-\phi)}} =1
\]
for $x>0$. Since $\phi(0) = 0$, we have
\[
\sqrt{2(1-\phi(x))} = \sqrt{2} -x \qquad\Ra\quad 1-\phi(x) = \left(1- \frac x{\sqrt 2}\right)^2 \qquad \Ra\quad \phi(x) = 1 - \left(1- \frac x{\sqrt 2}\right)^2 = \sqrt{2}x - \frac{x^2}2
\]
close to zero. Specifically, the solution
\[
\phi(x) = \begin{cases} \sqrt 2\,x- \sign(x)\,\frac{x^2}2 & |x|\leq \sqrt 2\\ \sign(x) & |x|\geq \sqrt 2\end{cases}
\]
is $C^{1,1}$-smooth on $\R$ and transitions between the potential wells on the finite segment $[-\sqrt 2, \sqrt 2]$, so the `optimal' $u$ should transition between the potential wells on a segment of width $2\sqrt 2\eps$. We compare the solution $u$ to an optimal transition layer centered at the same radius $r_{new}$ with length-scale $\eps$ in Figure \ref{figure comparison to optimal}.

Since $r- 2\eps < r_{i,\eps} < r_{new,\eps} < r$, this yields the additional guarantee that $|r_{new,\eps}-r| < 2\eps$, i.e.\ the motion is $\eps$-slow. No better result can be expected from estimating $r_{i,\eps}$ since the transition between the potential wells $1,-1$ must happen on a length-scale $r_{o,\eps} - r_{i,\eps}\sim\eps$ for phase-fields of moderate energy. However, we see in Figure \ref{figure radii varying eps} that also here, the correct scaling for the time step size should be $O(\eps^2)$, not $O(\eps)$. 

The strategy is similar as in Section \ref{section conclusion}, using (radial) {\em inner} variations of the optimal function. Here, these arguments can be made rigorous more easily. Specifically, we use that $u\equiv 1$ on $[0, r_i]$ and $u\equiv -1$ on $[r_o, \infty)$, meaning that we can easily `shift' the transition to the left or right by small amounts without violating boundary conditions. Since $u$ is smooth enough, the inner variation induces a smooth curve of $H^1$-functions taking values in $[-1,1]$.

Consider $w$ to be the function of shape $u$ with zero level set centered at $r$ and $w_t = w(s+t)$ in radial variables. Since $\sign(u_0)w_t \equiv 1$ outside of $B_{r_o-t}\setminus B_{r_i-t}$, the energy of the time-step functional obeys
\begin{align*}
\int_\Omega \frac\eps2 &\,\|\nabla w_t\|^2 + \frac{1-\sign(u_0)\,w_t}\eps\dx \\
	&= d\,\omega_d\int_{r_i-t}^{r_o-t}\left(\frac\eps2\,\big|w'(s+t)\big|^2 + \frac{1 - \sign(u_0(s))\,w(s+t)}\eps\right)s^{d-1}\ds+\big|\big(B_R\setminus B_{r_o-t}\big) \cup B_{r_i-t}\big|\\
	&= d\,\omega_d \int_{r_i}^{r_o} \left(\frac\eps2 \,\big|w'(s)\big|^2 + \frac{1-\sign(u_0(s-t))\, w(s)}\eps \right)(s-t)^{d-1}\ds+\omega_d\,\frac{R^d + (r_i-t)^d-(r_o - t)^{d}}\eps
\end{align*}
where $\omega_d$ is the Lebesgue volume of the $d$-dimensional unit ball. We have $u = w_t$ for some $t\in (-2\eps, 0)$ where the derivative of the expression is zero, i.e.\
\begin{align*}
0 &= d\omega_d \int_{r_i}^{r_o} \left(\frac\eps2 \,\big|w'(s)\big|^2 + \frac{1-\sign(u_0(s-t)) \,w(s)}\eps\right)(d-1)\,(s-t)^{d-2}\ds\\
	&\hspace{2cm}+2d\omega_d\,\frac{w(r+t)}{\eps} + d\omega_d\,\frac{(r_o-t)^{d-1}-(r_i-t)^{d-1}}\eps.
\end{align*}
The factors $d\omega_d$ in all terms can be cancelled, so
\[
 \frac{w(r+t)}\eps = - \frac{d-1}2\int_{r_i}^{r_o} \left(\frac\eps2 \,\big|w'(s)\big|^2 + \frac{1-\sign(u_0(s-t)) \,w(s)}\eps\right)\,(s-t)^{d-2}\ds - \frac{(r_o-t)^{d-1}-(r_i-t)^{d-1}}{2\eps}.
\]
It is easy to see that the expression on the right hand side is uniformly bounded independently of $\eps$ in terms of the energy and $(r_o-r_i)/\eps$. Accordingly, also the expression on the left must be uniformly bounded. From the explicit expression \eqref{eq obstacle derivative at zero}, we can see that $w'(s) = O(1/\eps)$ with a lower bound of the same order for $s$ close to $r$, so $w(r+t)/\eps = O(t/\eps^2)$ with a lower bound of the same order. It follows that $t=O(\eps^2)$.

\section{Numerical approximation of Allen-Cahn equation with singular potential}\label{appendix singular numerics}
Let us briefly discuss how we numerically approximate the Allen-Cahn equation
\[
	\partial_t u = \Delta u - \frac{W'(u)}{\eps^2}, \qquad W(u) = \infty \cdot 1_{\{|u|>1\}} + W_{conc}(u) = \infty \cdot 1_{\{|u|>1\}} + 1-u^2
\]
on the unit square with periodic boundary conditions. The convex-concave splitting algorithm corresponds to iteratively minimizing
\[
\int_{(0,1)^2} \frac12\,\|\nabla u\|^2 + \frac{W_{conc}'(u_n)}{\eps^2}\,u+ \frac1{2\tau}\,(u-u_n)^2\dx
\]
under the constraint $-1\leq u\leq 1$. 
We use the Operator Splitting Quadratic Program (OSQP) algorithm \cite{Stellato2020}, which is designed to solve linearly constrained quadratic optimization problems of the form
\[
	\min_{x\in\R^N} \frac{1}{2} x^T P x + q^T x\text{ s.t. } L \preceq Ax \preceq U.
\]
The matrix $P \in \R^{N\times N}$ and the vector $q\in \R^N$ define a quadratic problem and $L,U\in \R^M$ and $A\in \R^{M\times N}$ define linear constraints. In our case, the quadratic contribution $x^TPx/2$ corresponds to a discretization of
\[
\frac12\int \|\nabla u\|^2 + \frac1\tau\,u^2 \dx = \frac12\int \left(-\Delta u + \frac1\tau\,u\right)u \dx
\]
on a grid $u = (u_{ij})_{i,j=1}^m$ and the linear contribution $q^Tx$ comprises a discrete approximation 
\[
\frac1{m^2}\sum_{i, j=1}^m \left(\frac{W_{conc}'(u_{n, ij})}{\eps^2}- \frac{u_{n,ij}}\tau\right)u_{ij}
 \qquad\text{of}\quad \int \left(\frac{W_{conc}'(u_n)}{\eps^2} - \frac1\tau\,u_n\right)u\dx.
\]
The constraint matrix $A$ is an identity matrix and the vectors $L, U$ have constant entries $-1$ and $1$ respectively, i.e.\ $ -1 \leq u_{ij}\leq 1$ for all $i, j=1,\dots, m$.

We use our own implementation of the OSQP algorithm where we use the indirect
method described in \cite[Section 3.1]{Stellato2020} to solve the linear
systems that arise in the iterations of the algorithm as they can be
diagonalized efficiently using the FFT algorithm. 
The key contribution of our own implementation compared to available solvers is that we do not require a $N\times N$-matrix $P$ (with $N=m^2$ and rearrangement of the square array $u$), but only a linear operator $x\mapsto Px$ which can be implemented efficiently by the fast Fourier transform (FFT) for the map $u\mapsto -\Delta u + \frac1\tau\,u$.

Also in this experiment, we selected the spatial resolution $512\times 512$.

\bibliographystyle{alpha}
\bibliography{./bibliography.bib}

\newcommand{\etalchar}[1]{$^{#1}$}
\begin{thebibliography}{ADG{\etalchar{+}}24}

\bibitem[ADG{\etalchar{+}}24]{akande2024momentum}
Oluwatosin Akande, Patrick Dondl, Kanan Gupta, Akwum Onwunta, and Stephan
  Wojtowytsch.
\newblock Momentum-based minimization of the ginzburg-landau functional on
  euclidean spaces and graphs.
\newblock {\em arXiv preprint arXiv:2501.00389}, 2024.

\bibitem[AP24]{altschuler2024acceleration}
Jason~M Altschuler and Pablo~A Parrilo.
\newblock Acceleration by stepsize hedging: Silver stepsize schedule for smooth
  convex optimization.
\newblock {\em Mathematical Programming}, pages 1--14, 2024.

\bibitem[AP25]{altschuler2025acceleration}
Jason~M Altschuler and Pablo~A Parrilo.
\newblock Acceleration by stepsize hedging: Multi-step descent and the silver
  stepsize schedule.
\newblock {\em Journal of the ACM}, 72(2):1--38, 2025.

\bibitem[Bar15]{bartels2015numerical}
S{\"o}ren Bartels.
\newblock {\em Numerical methods for nonlinear partial differential equations},
  volume~47.
\newblock Springer, 2015.

\bibitem[BK91]{bronsard1991motion}
Lia Bronsard and Robert~V Kohn.
\newblock Motion by mean curvature as the singular limit of ginzburg-landau
  dynamics.
\newblock {\em Journal of differential equations}, 90(2):211--237, 1991.

\bibitem[BKS18]{bosch2018generalizing}
Jessica Bosch, Steffen Klamt, and Martin Stoll.
\newblock Generalizing diffuse interface methods on graphs: nonsmooth
  potentials and hypergraphs.
\newblock {\em SIAM Journal on Applied Mathematics}, 78(3):1350--1377, 2018.

\bibitem[BP93]{bellettini1993approssimazione}
Giovanni Bellettini and Maurizio Paolini.
\newblock Approssimazione variazionale di funzionali con curvatura.
\newblock In {\em Seminario di Analisi Matematica}, pages 87--97. Tecnoprint,
  1993.

\bibitem[Bro90]{bronsard1990slowness}
Lia Bronsard.
\newblock On the slowness of phase boundary motion in one space dimension.
\newblock {\em NASA STI/Recon Technical Report N}, 91:21480, 1990.

\bibitem[BvG19]{budd2019graph}
Jeremy Budd and Yves van Gennip.
\newblock Graph mbo as a semi-discrete implicit euler scheme for graph
  allen--cahn.
\newblock {\em arXiv preprint arXiv:1907.10774}, 2019.

\bibitem[BvGL21]{budd2021classification}
Jeremy Budd, Yves van Gennip, and Jonas Latz.
\newblock Classification and image processing with a semi-discrete scheme for
  fidelity forced allen--cahn on graphs.
\newblock {\em GAMM-Mitteilungen}, 44(1):e202100004, 2021.

\bibitem[CP89]{carr1989metastable}
Jack Carr and Robert~L Pego.
\newblock Metastable patterns in solutions of $u_t= \eps^2u_{xx}- f (u)$.
\newblock {\em Communications on pure and applied mathematics}, 42(5):523--576,
  1989.

\bibitem[DLW17]{dondl2017phase}
Patrick~W Dondl, Antoine Lemenant, and Stephan Wojtowytsch.
\newblock Phase field models for thin elastic structures with topological
  constraint.
\newblock {\em Archive for Rational Mechanics and Analysis}, 223(2):693--736,
  2017.

\bibitem[Dob10]{dobrowolski2010angewandte}
Manfred Dobrowolski.
\newblock {\em Angewandte Funktionalanalysis: Funktionalanalysis,
  Sobolev-R{\"a}ume und elliptische Differentialgleichungen}.
\newblock Springer-Verlag, 2010.

\bibitem[DW17]{dondl2017uniform}
Patrick~W Dondl and Stephan Wojtowytsch.
\newblock Uniform regularity and convergence of phase-fields for willmore’s
  energy.
\newblock {\em Calculus of Variations and Partial Differential Equations},
  56(4):90, 2017.

\bibitem[EO15]{esedoglu2015threshold}
Selim Esedoglu and Felix Otto.
\newblock Threshold dynamics for networks with arbitrary surface tensions.
\newblock {\em Communications on pure and applied mathematics}, 68(5):808--864,
  2015.

\bibitem[FH89]{fusco1989slow}
Giorgio Fusco and Jack~K Hale.
\newblock Slow-motion manifolds, dormant instability, and singular
  perturbations.
\newblock {\em Journal of Dynamics and Differential Equations}, 1:75--94, 1989.

\bibitem[FHX96]{fusco1996traveling}
Giorgio Fusco, Jack~K Hale, and Jianping Xun.
\newblock Traveling waves as limits of solutions on bounded domains.
\newblock {\em SIAM Journal on Mathematical Analysis}, 27(6):1544--1558, 1996.

\bibitem[FLS20]{fischer2020convergence}
Julian Fischer, Tim Laux, and Theresa~M Simon.
\newblock Convergence rates of the {A}llen--{C}ahn equation to mean curvature
  flow: A short proof based on relative entropies.
\newblock {\em SIAM Journal on Mathematical Analysis}, 52(6):6222--6233, 2020.

\bibitem[Ger85]{gerhardt1985global}
Claus Gerhardt.
\newblock Global $c^{1, 1}$-regularity for solutions of quasilinear variational
  inequalities.
\newblock {\em Archive for rational mechanics and analysis}, 89(1):83--92,
  1985.

\bibitem[GSW25]{grimmer2025accelerated}
Benjamin Grimmer, Kevin Shu, and Alex~L Wang.
\newblock Accelerated objective gap and gradient norm convergence for gradient
  descent via long steps.
\newblock {\em INFORMS Journal on Optimization}, 2025.

\bibitem[Ilm93]{ilmanen1993convergence}
Tom Ilmanen.
\newblock Convergence of the {A}llen-{C}ahn equation to {B}rakke's motion by
  mean curvature.
\newblock {\em Journal of Differential Geometry}, 38(2):417--461, 1993.

\bibitem[IPS99]{ishii1999threshold}
Hitoshi Ishii, Gabriel~E Pires, and Panagiotis~E Souganidis.
\newblock Threshold dynamics type approximation schemes for propagating fronts.
\newblock {\em Journal of the Mathematical Society of Japan}, 51(2):267--308,
  1999.

\bibitem[Ish95]{ishii1995generalization}
Hitoshi Ishii.
\newblock A generalization of the bence, merriman and osher algorithm for
  motion by mean curvature.
\newblock {\em Curvature flows and related topics}, pages 111--127, 1995.

\bibitem[KS00]{kinderlehrer2000introduction}
David Kinderlehrer and Guido Stampacchia.
\newblock {\em An introduction to variational inequalities and their
  applications}.
\newblock SIAM, 2000.

\bibitem[LO16]{laux2016convergence}
Tim Laux and Felix Otto.
\newblock Convergence of the thresholding scheme for multi-phase mean-curvature
  flow.
\newblock {\em Calculus of Variations and Partial Differential Equations},
  55(5):129, 2016.

\bibitem[LO20]{laux2020thresholding}
Tim Laux and Felix Otto.
\newblock The thresholding scheme for mean curvature flow and de {G}iorgi's
  ideas for minimizing movements.
\newblock In {\em The role of metrics in the theory of partial differential
  equations}, volume~85, pages 63--94. Mathematical Society of Japan, 2020.

\bibitem[LPS19]{lee2019regularity}
Ki-Ahm Lee, Jinwan Park, and Henrik Shahgholian.
\newblock The regularity theory for the double obstacle problem.
\newblock {\em Calculus of Variations and Partial Differential Equations},
  58:1--19, 2019.

\bibitem[LS18]{laux2018convergence}
Tim Laux and Theresa~M Simon.
\newblock Convergence of the {A}llen-{C}ahn equation to multiphase mean
  curvature flow.
\newblock {\em Communications on Pure and Applied Mathematics},
  71(8):1597--1647, 2018.

\bibitem[MM77]{modica1977esempio}
Luciano Modica and Stefano Mortola.
\newblock Un esempio di {$\Gamma$}-convergenza.
\newblock {\em Boll. Un. Mat. Ital. B}, 14:285--299, 1977.

\bibitem[Mod87]{modica1987gradient}
Luciano Modica.
\newblock The gradient theory of phase transitions and the minimal interface
  criterion.
\newblock {\em Archive for Rational Mechanics and Analysis}, 98:123--142, 1987.

\bibitem[RS06]{roger2006modified}
Matthias R{\"o}ger and Reiner Sch{\"a}tzle.
\newblock On a modified conjecture of de giorgi.
\newblock {\em Mathematische Zeitschrift}, 254:675--714, 2006.

\bibitem[SBG{\etalchar{+}}20]{Stellato2020}
Bartolomeo Stellato, Goran Banjac, Paul Goulart, Alberto Bemporad, and Stephen
  Boyd.
\newblock Osqp: an operator splitting solver for quadratic programs.
\newblock {\em Mathematical Programming Computation}, 12(4):637–672, February
  2020.

\bibitem[Woj23]{wojtowytsch2023stochastic}
Stephan Wojtowytsch.
\newblock Stochastic gradient descent with noise of machine learning type part
  i: Discrete time analysis.
\newblock {\em Journal of Nonlinear Science}, 33(3):45, 2023.

\end{thebibliography}

\end{document}